\documentclass[12pt]{amsart}  

\usepackage{setspace}
\usepackage[latin1]{inputenc}
\usepackage{float}
\usepackage{amsmath}
\usepackage{amsfonts}
\usepackage{amssymb}
\usepackage{latexsym}
\usepackage{graphicx}
\usepackage{subcaption}
\usepackage[symbol]{footmisc}
\usepackage{hyperref}
\usepackage{verbatim}
\usepackage[all]{xy}
\usepackage{graphics}
\usepackage{pdfsync}

\usepackage{mathtools}

\usepackage{ytableau}
\usepackage{mathabx} 
\usepackage{xifthen}                                
\usepackage[refpage]{nomencl}         
\usepackage{caption}
\usepackage{mathrsfs}
\usepackage{multirow}

\allowdisplaybreaks[2] \textwidth15.1cm \textheight22cm \headheight12pt \oddsidemargin.4cm
\usepackage{mathrsfs}

\theoremstyle{plain}
\newtheorem{theorem}{Theorem}[section]
\newtheorem{proposition}[theorem]{Proposition}

\newtheorem{lemma}[theorem]{Lemma}
\newtheorem{corollary}[theorem]{Corollary}
\newtheorem{conjecture}[theorem]{Conjecture}

\theoremstyle{definition}
\newtheorem{definition}[theorem]{Definition}
\newtheorem{example}[theorem]{Example}

\newtheorem{remark}[theorem]{Remark}

\numberwithin{equation}{section}
\newcommand{\bC}{\mathbb{C}}
\newcommand{\bN}{\mathbb{N}}

\newcommand{\bP}{\mathbb{Z}_+}
\newcommand{\bQ}{\mathbb{Q}}

\newcommand{\tcr}[1]{\textcolor{red}{#1}}
\newcommand{\tcb}[1]{\textcolor{blue}{#1}}

\newcommand{\Sym}{\ensuremath{\operatorname{Sym}}}

\newcommand{\set}{\mathrm{set}} 
\newcommand{\des}{\mathrm{des}} 
\newcommand{\comp}{\mathrm{comp}} 

\newcommand{\suchthat}{\;|\;}

\usepackage{xcolor}

\definecolor{dred}{rgb}{0.72,0,0} 
\newcommand{\bemph}[1]{\emph{#1}} 

\newcommand{\Asc}{\mathrm{Asc}}
\newcommand{\Dsc}{\mathrm{Des}}
\newcommand{\Des}{\mathrm{Des}}
\newcommand{\asc}{\mathrm{asc}} 
\newcommand{\dsc}{\mathrm{des}} 
\newcommand{\Inv}{\mathrm{Inv}}

\newcommand{\Cat}[1]{{\sf Cat}_{#1}}
\newcommand{\Mot}[1]{{\sf Mot}_{#1}}
\newcommand{\Stir}{{\sf Stir}} 
\newcommand{\Nar}{{\sf Nar}} 

\newcommand{\lpbt}{\mathcal{LT}} 

\newcommand{\pbt}{\mathcal{T}} 
\newcommand{\pbtu}{\pbt} 
\newcommand{\pbtl}{\pbt^{\ell}} 
\newcommand{\pbtun}{\pbtu_{n}} 
\newcommand{\pbtln}{\pbtl_{n}} 

\newcommand{\berntree}{\mathcal{T}^{B}} 

\newcommand{\leftt}{\ell(T)} 
\newcommand{\rightt}{\mathrm{r}(T)} 

\newcommand{\hypsemi}{\mathcal{I}} 
\newcommand{\hypcat}{\mathcal{C}} 
\newcommand{\hyplin}{\mathcal{L}} 
\newcommand{\hypbraid}{\mathcal{B}} 
\newcommand{\hypshi}{\mathcal{S}} 

\newcommand{\alpx}{\mathsf{x}} 
\newcommand{\rasc}{\mathrm{rasc}}
\newcommand{\lasce}{\mathrm{lasc}}
\newcommand{\rdes}{\mathrm{rdes}}
\newcommand{\ldes}{\mathrm{ldes}}
\newcommand{\pre}[1]{\mathrm{pre}(#1)} 
\newcommand{\wt}[1]{\mathrm{wt}(#1)} 

\newcommand{\typ}[1]{\mathbf{c}(#1)} 
\newcommand{\ltyp}[1]{\widecheck{\mathbf{c}}(#1)} 
\newcommand{\utyp}[1]{\widehat{\mathbf{c}}(#1)} 

\newcommand{\al}{\alpha}
\newcommand{\shape}[1]{\mathrm{sh}(#1)} 
\newcommand{\can}[1]{\mathrm{can}(#1)} 
\newcommand{\ex}{\mathrm{ex}} 
\newcommand{\mbx}[1]{\mathbf{x}^{#1}} 
\newcommand{\mbxs}[1]{\mathbf{x}_{#1}} 

\newcommand{\noncross}[1]{\mathrm{nc}(#1)} 
\newcommand{\noncrossing}[1]{\mathrm{NC}(#1)} 
\newcommand{\mnoncrossing}[1]{\mathrm{mNC}(#1)} 
\newcommand{\anoncrossing}[2]{\mathrm{mNC}^*(#1,#2)} 

\newcommand{\concat}{\bullet} 
\newcommand{\nconcat}{\odot} 

\DeclareMathOperator{\blk}{bk} 
\newcommand{\weightt}{\wt{T}} 

\newcommand{\ra}{\bar\rho}
\newcommand{\rd}{\rho}
\newcommand{\la}{\bar\lambda}
\newcommand{\ld}{\lambda}

\newcommand{\da}{\mathrm{da}}
\usepackage{tcolorbox}
\newcommand{\canopy}{\nu}

\newcommand{\ptl}{\mathscr{P\!T}^\ell}

\newcommand{\alps}{\mathsf{s}} 
\newcommand{\alpt}{\mathsf{t}} 
\renewcommand{\bP}{\mathbb{P}} 

\DeclareMathOperator{\inorder}{in} 
\DeclareMathOperator{\std}{std} 
\DeclareMathOperator{\rt}{root} 
\DeclareMathOperator{\SYT}{SYT}
\usepackage{enumerate}
\usepackage{tabu}
\newcommand{\foulkes}{F} 

\DeclareMathOperator{\leftleaves}{ll} 
\DeclareMathOperator{\rightleaves}{rl} 
\DeclareMathOperator{\leftmarks}{lm}  
\DeclareMathOperator{\rightmarks}{rm}  

\definecolor{mygreen}{rgb}{0.23, 0.45, 0.23}

\newcommand\ifrac[2]{#1/#2}

\newcommand{\regions}[1]{\mathrm{Regions}(#1)} 

\usepackage[normalem]{ulem} 
\newcommand{\tcgray}[1]{\textcolor{gray}{#1}}
\let\oldsout=\sout 
\renewcommand\sout[1]{\tcgray{\oldsout{#1}}}

\renewcommand\sout[1]{}

\newcommand{\rootof}[1]{\mathrm{root}(#1)} 
\newcommand{\node}[1]{\mathrm{Nodes}(#1)} 
\newcommand{\edge}[1]{\mathrm{Edges}(#1)} 

\newcommand{\Comp}{\mathrm{Comp}} 

\newcommand{\ja}{\mathrm{ja}} 
\newcommand{\jd}{\mathrm{jd}} 
\newcommand{\sdd}{\mathrm{sdD}} 
\newcommand{\sdu}{\mathrm{sdU}} 
\newcommand{\sad}{\mathrm{saD}} 
\newcommand{\sau}{\mathrm{saU}} 
\newcommand{\sa}{\mathrm{sa}} 
\newcommand{\sd}{\mathrm{sd}} 

\newcommand{\nodeone}{i}
\newcommand{\nodetwo}{j}
\newcommand{\nodethree}{k}
\newcommand{\vstart}{\mathrm{start}}
\newcommand{\vend}{\mathrm{end}}

\usepackage{algorithmicx, tabularx}
\usepackage[noend]{algpseudocode}
\makeatletter
\newcommand{\multiline}[1]{%
  \begin{tabularx}{\dimexpr\linewidth-\ALG@thistlm}[t]{@{}X@{}}
    #1
  \end{tabularx}
}
\makeatother

\algrenewcommand\algorithmicrequire{\textbf{Input:}}
\algrenewcommand\algorithmicensure{\textbf{Output:}}

\makeindex
\makenomenclature

\usepackage{etoolbox}
\renewcommand\nomgroup[1]{%
  \item[\bfseries
  \ifstrequal{#1}{B}{Trees}{%
  \ifstrequal{#1}{D}{Marked trees}{%
  \ifstrequal{#1}{R}{Right-leaning trees}{%
  \ifstrequal{#1}{A}{Compositions, partitions, Sym, NSym}{%
  \ifstrequal{#1}{C}{Labeled trees}{%
  \ifstrequal{#1}{H}{Hyperplane arrangements}{%
  \ifstrequal{#1}{E}{Set partitions}{%
  \ifstrequal{#1}{Q}{Miscellaneous}{%
  }}}}}}}}%
]}

\makeatletter
\newcommand{\addresseshere}{%
  \enddoc@text\let\enddoc@text\relax
}
\makeatother

\begin{document}

\title[Labeled  binary trees and Schur positivity]{Labeled binary trees, subarrangements of the Catalan arrangements, and Schur positivity}
\author{Ira M. Gessel}
\address{Department of Mathematics, Brandeis University, Waltham, MA 02453, USA}
\email{\href{mailto:gessel@brandeis.edu}{gessel@brandeis.edu}}

\author{Sean T. Griffin}
\address{Department of Mathematics, University of Washington, Seattle, WA 98195, USA}
\email{\href{mailto:stgriff@math.washington.edu}{stgriff@math.washington.edu}}

\author{Vasu Tewari}
\address{Department of Mathematics, University of Pennsylvania, Philadelphia, PA 19104, USA}
\email{\href{mailto:vvtewari@math.upenn.edu}{vvtewari@math.upenn.edu}}
\thanks{
The first author was supported by a grant from the Simons Foundation (\#427060, Ira Gessel).
The second author was supported by the ARCS Foundation Fellowship, as well as NSF Grant DMS-1101017.
The third author was supported by an AMS-Simons Travel Grant.}

\subjclass[2010]{Primary 05A19, 05E05,  05E18 ; Secondary 05A05, 05E10, 06A07}
\keywords{labeled trees, Schur positivity, hyperplane arrangements, ascents-descents}
\date{\today}

\begin{abstract}
In 1995, the first author introduced a multivariate generating function  {$G$} that tracks the distribution of ascents and descents in labeled binary trees.
In addition to proving that $G$ is symmetric, he conjectured that $G$ is Schur positive.
We prove this conjecture by expanding $G$ positively in terms of ribbon Schur functions. We obtain this expansion using a weight-preserving bijection whose inverse is inspired by the Push-Glide algorithm of Pr\'{e}ville-Ratelle and Viennot.
In fact, this weight-preserving bijection allows us to establish a stronger version of the first author's conjecture showing that the generating function restricted to labeled binary trees with a fixed canopy is still Schur positive.

We also discuss applications in the setting of hyperplane arrangements.
We show that a certain specialization of $G$ equals the Frobenius characteristic of the natural $\mathfrak{S}_n$-action on regions of the semiorder arrangement, which we then expand in terms of  {the Frobenius characteristics} of Foulkes characters.
We also construct an $\mathfrak{S}_n$-action on regions of the Linial arrangement using a set of trees studied by Bernardi, and subsequently compute the character of this action by employing Lagrange inversion. The resulting expression generalizes Postnikov's formula for the number of regions in the Linial arrangement.
As a final application, we prove $\gamma$-nonnegativity for the distribution of the number of right edges over local binary search trees.
\end{abstract}

\maketitle
\tableofcontents

\section{Introduction}\label{sec: Introduction}
The study of permutation statistics is a classical theme in algebraic combinatorics with its genesis in work by MacMahon \cite{MacMahon}.
An important statistic introduced by MacMahon is the descent statistic on permutations.
The generating function for the distribution of this statistic gives rise to the well-known Eulerian polynomials, which show up in many areas in mathematics. The reader is referred to \cite{Petersen} for a detailed survey.
Since the work of MacMahon, the descent statistic on permutations has been studied in depth, and yet it continues to inspire new research \cite{ShareshianWachs-1, ShareshianWachs-2}.
In this article, we study ascent-descent statistics on labeled plane binary trees.
For brevity's sake, by a tree, we will always mean a  plane binary tree.
Whether the tree is labeled or not will be clear from context.
We remark that the notion of descents has  been studied with regards to other combinatorial objects before, such as in the case of standard Young tableaux ($\mathrm{SYTs}$).
However, viewing $\mathrm{SYTs}$ as $P$-partitions reveals that their descents are in fact descents of permutations in disguise.
In contrast, the ascent-descent statistics that we study here are indeed different, as they depend on the embedding of the labeled trees in the plane and take into account the orientation of the edges.

More specifically, the ascent and descent statistics on labeled trees each come in two flavors depending on whether one compares the label of the parent node to the label of its right child or its left child. {Note that in our trees, we are allowing nodes to have just a left child or just a right child, see Subsection~\ref{subsec: unlabeled trees}. We always draw our trees embedded in the plane with the root on top.}
Given a positive integer $n$, let $\pbtln$ (respectively $\pbtun$) denote the set of labeled (respectively unlabeled) plane binary trees on $n$ nodes. 
The labels on the nodes are drawn from the set of positive integers $\bP$, allowing repeats.
A \bemph{standard labeling} of a tree $T \in \pbtun$ is a labeling of its nodes with distinct labels drawn from $[n]\coloneq \{1,\ldots,n\}$, and we call a tree  with a standard labeling a \bemph{standard labeled tree}.
For a labeled node $v$ in $T$, denote by $v^\ell$ the label of $v$.
If $v$ is the left child of $w$, we say the edge between them is a \bemph{left ascent} if $v^\ell \leq w^\ell$. Otherwise, we say the edge is a \bemph{left descent}.
Similarly, if a node $v$ has a right child $w$, we say the edge between them is a \bemph{right ascent} if $v^\ell \leq w^\ell$. Otherwise, we say the edge is a \bemph{right descent}.
{One can think of these four statistics by listing the labels of the edge from left to right and then considering whether this pair is an ascent or descent.}
For any labeled tree $T\in \pbtln$, let $\lasce(T)$, $\ldes(T)$, $\rasc(T)$ and $\rdes(T)$ denote the number of left ascents, left descents, right ascents, and right descents in $T$, respectively. {See Figure~\ref{fig:Ascents-Descents} for two examples of labeled trees with $3$ left ascents, $1$ left descent, $3$ right ascents, and $1$ right descent.}

We recover the case of ascents and descents of permutations by considering labeled trees in which no node has a left child, or alternatively, by considering labeled trees in which no node has a right child.
Thus, the study of these statistics on labeled binary trees is a natural generalization of the study of ascents and descents on permutations.

The first author, in the 1990s, initiated the study of these statistics and considered the following generating function tracking their distribution over the set of standard labeled trees.
\begin{equation}
B\coloneqq B(x;\la,\ld,\ra,\rd)=
\sum_{n\geq 1}
\sum_{\substack{T\in \pbtln\\T \text{ standard}}}
\la^{\lasce(T)}\ld^{\ldes(T)} \ra^{\rasc(T)}\rd^{\rdes(T)}
 \frac{x^n}{n!}.
\end{equation}
In unpublished work, the first author showed that $B$ satisfies the functional equation
\begin{equation}\label{eq: B functional equation}
\frac{(1+\la B)(1+\ra B)}{(1+\ld B)(1+\rd B)}=e^{[(\la\ra-\ld\rd)B+\la+\ra-\ld-\rd]x}.
\end{equation}
Subsequently, different proofs of Equation \eqref{eq: B functional equation} were  given by Kalikow \cite{Kalikow} and Drake \cite{Drake}.
From the definition of $B$, we observe that $B(x;\la,\ld,\ra,\rd)=B(x;\rd,\ra,\ld,\la)$ and $B(x;\la,\ld,\ra,\rd)=B(x;\ld,\la,\rd,\ra)$.
The former is explained by reflecting a standard labeled tree across a vertical line passing through its root, while the latter follows from changing the label of a node from $i$ to $n-i+1$ in a standard labeled tree in $\pbtln$.
Equation~\eqref{eq: B functional equation} brings to light another pair of symmetries, that
\begin{align}
B(x;\la,\ld,\ra,\rd)=B(x;\ra,\ld,\la,\rd)=B(x;\la,\rd,\ra,\ld).
\end{align}
These equalities are not obvious from the definition and a simple bijective proof for them remains elusive, although a complicated bijection can be derived from the work of Kalikow \cite{Kalikow}.

One impetus to return to the study of $B$ has been fueled by connections with  enumerative aspects of the theory of hyperplane arrangements.
Let $B_n\coloneq B_n(\la,\ld,\ra,\rd)$ denote  the coefficient of $x^n/n!$ in $B$ for $n\geq 1$. The expansion for $B_n$ when $1\leq n\leq 5$ is given in Appendix~\ref{app: expansions}.
The first author observed that certain evaluations of $B_n$ coincide with the number of regions in various well-known deformations of Coxeter arrangements \cite{Gessel-Oberwolfach}.
This viewpoint has been pursued in \cite{Corteel-Forge-Ventos, Forge, Tewari}, and a complete explanation has been offered by Bernardi \cite{Bernardi}.
Given a subset $A$ of $\{-1,0,1\}$, we can consider the arrangement in $\mathbb{R}^n$ consisting of all hyperplanes $x_i-x_j=a$ where $i<j$ and   $a\in A$.
For $A=\{0\}$, $A=\{-1,0,1\}$, $A=\{-1,1\}$, $A=\{0,1\}$ and $A=\{1\}$, the corresponding hyperplane arrangements in $\mathbb{R}^n$ are the \bemph{braid arrangement} $\mathcal{B}_n$, the \bemph{Catalan arrangement} $\hypcat_n$, the \bemph{semiorder arrangement} $\hypsemi_n$, the \bemph{Shi arrangement} $\hypshi_n$, and the \bemph{Linial arrangement} $\hyplin_n$ respectively.
These arrangements are very well studied \cite{Athanasiadis-Linusson, Headley, Postnikov-Stanley,Shi-Lecture notes, Shi-JLMS,Stanley-Pnas} and are  instances of deformations of Coxeter arrangements  called \bemph{truncated affine arrangements} \cite{Postnikov-Stanley}.
Various aspects of truncated affine arrangements have been studied in great detail in \cite{Athanasiadis-survey, Athanasiadis-JACo,Athanasiadis-EuJC,Postnikov-Stanley} and we refer  the reader to them for further information.
Remarkably, we have the following equalities in which the left-hand side is an  evaluation of $B_n$ and the right-hand side is the number of regions in a Coxeter arrangement deformation,
\begin{align}
B_{n}(1,1,1,1)&=\text{ number of regions in } \hypcat_n=\frac{n!}{n+1}\binom{2n}{n},\\
B_{n}(1,0,1,1)&=\text{ number of regions in } \hypshi_n=(n+1)^{n-1},\\
B_{n}(1,1,0,0)&=\text{ number of regions in } \hypbraid_n=n!,\\
B_{n}(1,0,1,0)&=\text{ number of regions in } \hyplin_n=\frac{1}{2^n}\sum_{k=0}^n \binom{n}{k}(k+1)^{n-1},\\
B_{n}(1,\zeta_6^{-1},1,\zeta_6)&= \text{ number of regions in } \hypsemi_n\label{eq:semi}.
\end{align}
In \eqref{eq:semi}, $\zeta_6$ denotes a primitive sixth root of unity.
Section~\ref{sec: Hyperplane arrangements} of this article is devoted to a representation-theoretic understanding of these equalities.

Our primary object of study is a multivariate generalization of $B$ introduced by the first author.
Let $\alpx =\{x_1,x_2,\ldots\}$ be a {set of commuting} indeterminates.
With every $T\in \pbtln$, we associate a monomial $\alpx^T$ as follows. For a node $v\in T$ labeled $i$, let $x_v$ be $x_i$.
Then
\begin{equation}
\alpx^T\coloneq\prod_{v\in T}x_v.
\end{equation}
Now consider the formal  power series in $\alpx$ with coefficients in $\mathbb{Q}[\la,\ld,\ra,\rd]$,
\begin{equation}
G\coloneq G(\alpx;\la,\ld,\ra,\rd)=\sum_{n\geq 1}\sum_{T\in \pbtln}
\la^{\lasce(T)}\ld^{\ldes(T)}\ra^{\rasc(T)}\rd^{\rdes(T)}\,\alpx^T.
\end{equation}
It transpires that $G$ is a symmetric function\footnote{ {The second and third authors refer to $G$ as \emph{Gessel's tree symmetric function} for this reason.}} in $\alpx$ with coefficients in $\mathbb{Q}[\la,\ld,\ra,\rd]$. This non-obvious fact follows from the following functional equation satisfied by $G$.
\begin{theorem}\label{thm: Gessel functional equation}
Let $H(z)=\sum_{n\geq 0}h_nz^n$ where $h_n$ denotes the $n$th complete homogeneous symmetric function. We have
\begin{equation}
\label{eq: functional equation}
\frac{(1+\la G)(1+\ra G)}{(1+\ld G)(1+\rd G)}=H((\la\ra-\ld\rd)G+\la+\ra-\ld-\rd).
\end{equation}
\end{theorem}
\noindent Observe that the functional equation for $B$ in~\eqref{eq: B functional equation} can be obtained from~\eqref{eq: functional equation} by applying the homomorphism sending $h_n$ to $\ifrac{x^n\!}{n!}$.
This homomorphism has the crucial feature of sending the coefficient of $x_1x_2\cdots x_n$ in any symmetric function to the coefficient of $\ifrac{x^n\!}{n!}$ in its image.

Given that $G$ is a symmetric function, it is natural to ask for its expansion in the basis of Schur functions. This brings us to our first new result, which was originally conjectured by the first author \cite{Gessel-unpublished} in 1995.
\begin{theorem}\label{conj: schur positivity}
$G$ is Schur positive.
\end{theorem}
\noindent Here we mean that $G$ may be expressed as a sum of Schur functions $s_\lambda$ with coefficients in the semiring $\bN[\la,\ld,\ra,\rd]$.
Theorem~\ref{conj: schur positivity} follows from another recursive functional equation satisfied by $G$, which is also one of our main results.
\begin{theorem}\label{thm: ribbon functional equation}
We have
  \begin{equation}
    G = \sum_{n\geq 1}\sum_{\alpha\vDash n} (\la\ra\, G+\la+\ra)^{n-\ell(\alpha)}(\ld\rd\,G+\ld+\rd)^{\ell(\alpha)-1}\,r_\alpha,
  \end{equation}
where $r_\alpha$ denotes the ribbon Schur function indexed by the composition $\alpha$ and $\ell(\alpha)$ denotes the length of $\alpha$.
\end{theorem}
In fact, Theorem~\ref{thm: ribbon functional equation} implies the much stronger fact that $G$ may be expressed as a sum of ribbon Schur functions with coefficients in the semiring $\bN[\la\ra,\ld\rd,\la+\ra,\ld+\rd]$.
For $n\geq 1$, let $G_n\coloneqq G_n(\alpx;\la,\ld,\ra,\rd)$ denote the sum of the terms in $G$ of total degree $n$ in $\alpx$. The expansion for $G_n$ when $1\leq n\leq 5$ is given in Appendix~\ref{app: expansions}.
We use a certain class of decorated noncrossing partitions called \bemph{marked interlacing partitions} to give an expansion of $G_n$ in terms of ribbon Schur functions, stated next. Here, $\wt{\pi}$ is a product of expressions of the form $\la\ra$, $\ld\rd$, $\la+\ra$, and $\ld+\rd$ depending on the marked interlacing partition $\pi$. See Section~\ref{sec: fixed canopy} for the relevant definitions and notation, and see Section~\ref{sec: ribbon expansion} for the proof of Theorem~\ref{cor: Ribbon Expansion}.
\begin{theorem}\label{cor: Ribbon Expansion}
The formal power series $G_n$ has the following expansion in terms of ribbon Schur functions.
\begin{equation}
G_n = \sum_{\substack{\pi \in \mnoncrossing{n}\\ \pi = B_1/\dots/B_k}}\wt{\pi} r_{c(B_1)}r_{c(B_2)}\dots r_{c(B_k)}\label{eq:ribbonexpansion}.
\end{equation}
\end{theorem}

A fact worth noting is that after expanding out the products of ribbon Schur functions on the right hand side of~\eqref{eq:ribbonexpansion}, the coefficient of $r_{\alpha}$ for every $\alpha\vDash n$  evaluates to the \bemph{Catalan number} $\Cat{n}\coloneqq \frac{1}{n+1}\binom{2n}{n}$ upon setting $\la=\ld=\ra=\rd=1$.
Additionally, the coefficients of both $r_{(1,\dots,1)}$ and $r_{(n)}$ upon setting $\la=\ld=t$, and $\ra=\rd=q$ respectively are the \bemph{homogenized Narayana polynomials} \[\Nar_n(q,t)\coloneqq\displaystyle\sum_{k=0}^{n-1}\frac{1}{n}\binom{n}{k}\binom{n}{k+1}q^kt^{n-1-k}.\]

We provide two proofs of Theorem \ref{thm: ribbon functional equation}, each with its own merits.
In Section~\ref{sec:first proof}, we give an  algebraic proof that follows from Theorem \ref{thm: Gessel functional equation} combined with a result of  MacMahon \cite[Vol. 1, p.~186]{MacMahon}.
Our second proof, which we postpone until Section~\ref{sec: Proof of Main Theorem}, utilizes a weight-preserving bijection whose inverse is inspired by the Push-Glide algorithm of Pr\'eville-Ratelle and Viennot \cite{Preville-Ratelle-Viennot}. We use our weight-preserving bijection to prove Theorem~\ref{cor: Ribbon Expansion} in Section~\ref{sec: ribbon expansion}. 
The weight-preserving bijection will then allow us to establish a further refinement of Theorem~\ref{conj: schur positivity}, which was conjectured by the first author \cite{Gessel-unpublished}.
\begin{theorem}\label{conj: refined Schur positivity}
Fix a positive integer $n$. Let $\canopy$ be a word of length $n-1$ in the alphabet $\{U,D\}$, and let $\pbt_{n,\canopy}^{\ell}$ denote the set of labeled trees on $n$ nodes with canopy $\canopy$.
We have that the generating function
\begin{align}\label{eq: fixed canopy sum}
G_{n,\canopy}\coloneq G_{n,\canopy}(\alpx;\la,\ld,\ra,\rd)=\sum_{T\in \pbt_{n,v}^{\ell}}\la^{\lasce(T)}\ld^{\ldes(T)}\ra^{\rasc(T)}\rd^{\rdes(T)}\,\alpx^T
\end{align}
is Schur positive.
\end{theorem}
\noindent Details on the terminology used in Theorem~\ref{conj: refined Schur positivity} can be found in Section \ref{sec: preliminaries}.

In fact, using our weight-preserving bijection, we obtain an expansion of $G_{n,\canopy}$ in terms of ribbon Schur functions using a class of decorated noncrossing partitions called \bemph{augmented interlacing partitions}. See Section~\ref{sec: fixed canopy} for the relevant definitions and notation, and see Section~\ref{sec: ribbon expansion} for the proof of Theorem~\ref{thm: ribbon expansion for fixed canopy}.
\begin{theorem}\label{thm: ribbon expansion for fixed canopy}
For $n\geq 1$ and $\canopy\in \{U,D\}^{n-1}$, we have
  \[
    G_{n,\canopy} = \sum_{\substack{\pi^\ast\in \anoncrossing{n}{\canopy}\\ \pi^\ast= B_1/\dots/B_k}}\wt{\pi^\ast} r_{c(B_1)}r_{c(B_2)}\dots r_{c(B_k)}.
  \]
\end{theorem}

In Section~\ref{sec: Hyperplane arrangements}, we connect specializations of $G$ to deformations of Coxeter arrangements, focusing in particular on semiorder and Linial arrangements. Our main results in this setting are the following.
\begin{theorem}\label{thm: intro hyperplane semiorder}
The Frobenius characteristic of the natural $\mathfrak{S}_n$-action on the set of regions of the semiorder arrangement $\hypsemi_n$ is $G_n(\alpx;1,\zeta_6^{-1},1,\zeta_6)$.
\end{theorem}
\begin{theorem}\label{thm: intro hyperplane Linial}
There exists an $\mathfrak{S}_n$-action on the set of regions of the Linial arrangement $\hyplin_n$ whose graded Frobenius characteristic is given by $G_n(\alpx; \la,0,\ra,0)$.
\end{theorem}
\noindent The proof of Theorem~\ref{thm: intro hyperplane semiorder} utilizes a cycle indicator computation relying on a result of Postnikov-Stanley \cite{Postnikov-Stanley}, whereas the proof of Theorem~\ref{thm: intro hyperplane Linial} utilizes crucially a recent bijection of Bernardi \cite{Bernardi} relating regions of $\hyplin_n$ to a certain class of labeled trees that we call Bernardi trees.

It is worth mentioning that the set of regions of the braid arrangement $\hypbraid_n$ and that of the Catalan arrangement $\hypcat_n$ also carry a natural $\mathfrak{S}_n$-action stemming from the fact that the set of hyperplanes defining both these arrangements is itself $\mathfrak{S}_n$-stable.
In the case of $\hypbraid_n$, the $\mathfrak{S}_n$-action gives rise to the regular representation of $\mathfrak{S}_n$, whereas in the case of $\hypcat_n$
we obtain a direct sum of $\Cat{n}$ many copies of the regular representation.
It can be seen that the Frobenius characteristics of the $\mathfrak{S}_n$-actions on the regions of $\hypbraid_n$ and $\hypcat_n$ are $ G_n(\alpx;1,1,0,0)$  and  {$G_n(\alpx;1,1,1,1)$}, respectively.

More interesting and equally well known is the fact that the regions of the Shi arrangement $\hypshi_n$ carry a $\mathfrak{S}_n$-action as well.
This can be realized, for instance, by identifying the regions of $\hypshi_n$ with parking functions via the Pak-Stanley labeling  {\cite{Stanley-Park}}.
The Frobenius characteristic of the resulting $\mathfrak{S}_n$-action on the regions of the Shi arrangement equals $\mathrm{PF}_n$, the Frobenius characteristic of the character of the well-known parking function representation \cite{Haiman}.
It is straightforward to see that the functional equation in Theorem~\ref{thm: Gessel functional equation} reduces to the functional equation satisfied by $\mathrm{PF}_n$ when we set $\la=\ld=\ra=1$ and $ \rd=0$. Thus, we obtain the following theorem.
\begin{theorem}\label{thm: intro hyperplane Shi}
The Frobenius characteristic of the $\mathfrak{S}_n$-action on the set of regions of the Shi arrangement  {$\hypshi_n$} coming from their identification with parking functions is given by $G_n(\alpx; 1,1,1,0)$.
\end{theorem}
Since the cases of the braid, Catalan, and Shi arrangements are well-understood, we briefly touch upon them in Section~\ref{sec: back to the future} and instead focus primarily on the semiorder and Linial arrangements in this article.

This article is the full version of the extended abstract \cite{FPSAC-17}.

\smallskip
\noindent {\bf Outline of the article.} In Section \ref{sec: preliminaries}, we introduce our main combinatorial objects, and develop most of the notation we need.
Section \ref{sec:first proof} provides a generating function proof of Theorem~\ref{thm: Gessel functional equation}. We then prove Theorem \ref{thm: ribbon functional equation} using a generating function identity of MacMahon. Theorem~\ref{conj: schur positivity} will then follow as a corollary.
In Section~\ref{sec: fixed canopy}, we define all definitions and notation used in Theorem~\ref{cor: Ribbon Expansion} and Theorem~\ref{thm: ribbon expansion for fixed canopy}.

In Section~\ref{sec: Hyperplane arrangements}, we discuss applications of our results to studying actions of the symmetric group on Coxeter deformations focusing in particular on  semiorder and Linial arrangements.
In Section~\ref{sec:gamma}, we prove $\gamma$-nonnegativity for the distribution of right edges over local binary search trees.
In Section~\ref{sec: Proof of Main Theorem}, we provide a direct combinatorial proof of Theorem~\ref{thm: ribbon functional equation} using a weight-preserving bijection, without utilizing generating function machinery. 
Our Corollary~\ref{cor: noncommutative inorder} gives a natural noncommutative analogue of Theorem \ref{thm: ribbon functional equation}.
In Section~\ref{sec: ribbon expansion}, we use the weight-preserving bijection defined in Section~\ref{sec: Proof of Main Theorem} to prove the ribbon Schur expansions stated in Theorems~\ref{cor: Ribbon Expansion} and~\ref{thm: ribbon expansion for fixed canopy}.
We conclude with further avenues in Section \ref{sec: back to the future}.
Appendix~\ref{app: expansions} lists expansions for $B_n$ and $G_n$ for some small values of $n$.

\smallskip

\noindent {\bf Acknowledgements.}
We are extremely grateful to Sara Billey for numerous helpful conversations, her many pertinent suggestions at various stages of this project, as well as help with improving the exposition.
{We would also like to thank the anonymous referee for their many valuable comments and suggestions.}
We are also grateful to T. Kyle Petersen and Christos Athanasiadis for enlightening correspondence(s) and helpful feedback. Further thanks go to Eugene Gorsky, Patricia Hersh, Jia Huang, Alejandro Morales, Igor Pak and Anne Schilling for interesting discussions. We would also like to thank Jonah Ostroff and Josh Swanson for \TeX-nical assistance.

\section{Combinatorial preliminaries}\label{sec: preliminaries}
In  this section, we introduce some of the main combinatorial objects of this article. For further details on compositions, words, and binary trees, we refer the reader to \cite{Stanley-EC1}.

\subsection{Compositions}\label{subsec: compositions}
A finite ordered list of positive integers $\alpha = (\alpha _1, \ldots , \alpha _\ell)$ is called a \bemph{composition}.
If $\sum_{i=1}^{\ell} \alpha_i=n$, then we say that $\alpha$ is a composition of \bemph{size} $n$ and denote this by $\alpha \vDash n$.
We call $\alpha _i$ the \bemph{parts} of $\alpha$ and denote the number of parts of $\alpha$ by $\ell(\alpha)$, also called the \bemph{length} of $\alpha$. A \bemph{partition} $\lambda$ of size $n$, denoted by $\lambda\vdash n$, is a composition $(\lambda_1,\ldots, \lambda_k)$ of size $n$ satisfying $\lambda_1\geq \dots \geq \lambda_k$.  {Occasionally, we write $|\lambda|$ for the size of $\lambda$. The partition $\lambda$ can be identified with its \bemph{Young diagram}, given by drawing left-justified rows of boxes, with $\lambda_i$ boxes in row $i$. We adhere to the English convention, where row $1$ is drawn as the top row. }

We define two operations on compositions.
Given compositions $\alpha= (\alpha_1,\ldots ,\alpha_{\ell})$ and $\beta=(\beta_1,\ldots,\beta_m)$, we define the \bemph{concatenation} of $\alpha$ and $\beta$, denoted by $\alpha \concat \beta$, to be the composition $(\alpha_1,\ldots,\alpha_\ell,\beta_1,\ldots,\beta_m)$.
The \bemph{near-concatenation} of $\alpha$ and $\beta$, denoted by $\alpha \nconcat \beta$, is defined to be the composition $(\alpha_1,\ldots, \alpha_{\ell-1},\alpha_{\ell}+\beta_1,\beta_2,\ldots,\beta_m)$.
For example, if  $\alpha=(2,1,\tcr{3})$ and $\beta=(\tcr{4},1)$, then $\alpha\concat\beta=(2,1,\tcr{3},\tcr{4},1)$ while $\alpha \nconcat \beta =(2,1,\tcr{7},1)$.

Recall the well-known bijection between compositions $\alpha=(\alpha_1,\ldots,\alpha_k)$ of $n$ and subsets $S\subseteq [n-1]$ given by $S=\{\alpha_1,\alpha_1+\alpha_2,\ldots,\alpha_1+\cdots +\alpha_{k-1}\}$.
We denote the set corresponding to $\alpha\vDash n$ by $\set(\alpha)$, and in the opposite direction, given $S\subseteq [n-1]$, we denote the corresponding composition of size $n$ by $\comp(S)$.
The inclusion order on subsets induces a natural poset structure on the set of compositions of size $n$.
More precisely, given $\alpha,\beta \vDash n$, we say that $\alpha \preccurlyeq \beta$ if and only if $\set(\beta)\subseteq \set(\al)$, and call $\preccurlyeq$ the \bemph{refinement order} on compositions.
For instance, consider $\alpha=(\tcr{1},\tcr{2},4,\tcb{2},\tcb{3},\tcb{2},1)$ and $\beta=(\tcr{3},4,\tcb{7},1)$, both compositions of size $15$. Then $\set(\beta)=\{3,7,14\}$ and $\set(\al)=\{1,3,7,9,12,14\}$.
Clearly we have that $\set(\beta)\subseteq \set(\al)$, and therefore $\alpha \preccurlyeq \beta$.
We denote this poset on compositions of size $n$ by $\Comp_n$ and refer to it as the \bemph{composition poset}.
Given compositions $\alpha,\gamma\vDash n$ such that $\alpha\preccurlyeq \gamma$, we denote the interval in $\Comp_n$ comprising compositions $\beta$ satisfying $\alpha\preccurlyeq \beta \preccurlyeq\gamma$ by $[\alpha,\gamma]$.

\subsection{Words and symmetric functions}\label{subsec: sym}
Let $\bP$ be the set of positive integers.
Let $\bP^+$ be the set of nonempty words on $\bP$, which is the set of finite sequences of positive integers with positive length.
If $w$ is a word with letters $w_1, w_2, \ldots, w_n$, we write $w=w_1\cdots w_n$.
To $w$, we associate the monomial $\alpx^w \coloneqq x_{w_1}x_{w_2}\cdots x_{w_n}$.
We denote the set of words in $n$ letters by $\bP^n$.
An \bemph{ascent} of $w$ is an index $1\leq i\leq n-1$ such that $w_i\leq w_{i+1}$. A \bemph{descent} of $w$ is an index $1\leq i\leq n-1$ such that $w_i>w_{i+1}$.
Let the \bemph{descent set} of $w$ be $\Dsc(w) \coloneqq \{\,1\leq i\leq n-1\suchthat w_i>w_{i+1}\,\}$. An \bemph{inversion} in $w$ is a pair of indices $1\leq i<j\leq n$ such that $w_i > w_j$. We denote the set of inversions of $w$ by $\Inv(w)$. Therefore, we have $i\in \Dsc(w)$ if and only if $(i,i+1)\in \Inv(w)$.

The \bemph{standardization} of  $w$, denoted by  $\std(w)$, is  the  permutation in $\mathfrak{S}_n$ obtained by replacing the entries of $w$ with 1, 2, \dots, $n$, keeping the same relative order, where repeated letters are considered as increasing from left to right.
For example, the standardization of $112123$ is $124356$. For $w\in \bP^n$, the standardization $\sigma = \std(w)$ has the key property that $\Inv(\sigma) = \Inv(w)$.

A \bemph{symmetric function} is a formal power series in the variables $\alpx$ which is invariant under swapping any two of the variables. It is well known that the set of symmetric functions with coefficients in $\bQ$, denoted by $\Sym$, forms a ring under the usual operations of addition and multiplication. For notions related to the ring of symmetric functions that are not made explicit here, we refer the reader to \cite{Stanley-EC2}. For the connection between symmetric functions and symmetric group characters via the Frobenius characteristic map, we refer the reader to \cite{Sagan}. Occasionally, we will refer to the Frobenius characteristic of an $\mathfrak{S}_n$-module when we mean the Frobenius characteristic of the character of that module.

We denote by $h_n$ the $n$th complete homogeneous symmetric function, which is the sum over all monomials in $\alpx$ of degree $n$. The $n$th elementary symmetric function, denoted by $e_n$, is the sum over all squarefree monomials of degree $n$. For $\lambda\vdash n$ of length $k$, let $h_\lambda \coloneq h_{\lambda_1}h_{\lambda_2}\cdots h_{\lambda_k}$ be the complete homogeneous symmetric function indexed by $\lambda$. Similarly, let $e_\lambda \coloneq e_{\lambda_1}e_{\lambda_2}\cdots e_{\lambda_k}$ be the elementary symmetric function indexed by $\lambda$.
Let $\alpha \vDash n$ of length $k$. For convenience, we occasionally write $h_\alpha$ to mean $ h_{\alpha_1}h_{\alpha_2}\cdots h_{\alpha_k}$, and similarly for $e_\alpha$. Observe that if $\lambda$ is the unique partition obtained by sorting the parts of $\alpha$, then $h_\alpha=h_\lambda$, and  similarly for $e_\alpha$.
Let $H(z) = \sum_{n\geq 0} h_n z^n$ and $E(z) = \sum_{n\geq 0} e_n z^n$. Furthermore, let $\omega$ be the automorphism of $\Sym$ which sends $h_n$ to $e_n$ for each $n$.

Given a partition $\lambda\vdash n$, a \bemph{standard Young tableau} (henceforth $\SYT$) of \bemph{shape} $\lambda$ is a filling of the boxes of the Young diagram of $\lambda$ with integers from $[n]$ using each label exactly once, such that the labeling increases left-to-right along rows and top-to-bottom down columns.
We denote the set of all $\SYT$ of shape $\lambda$ by $\SYT(\lambda)$.
For $T\in \SYT(\lambda)$, we define $\Dsc(T)$ to be the set of all $1\leq i\leq n-1$ such that $i$ belongs to a row above $i+1$ in $T$.

Let $\lambda\vdash m$ and $\mu\vdash n$ such that $\ell(\mu)\leq \ell(\lambda)$ and $\mu_i\leq \lambda_i$ for all $i\leq \ell(\mu)$. In this case, the Young diagram of $\mu$ forms a subdiagram of the Young diagram of $\lambda$, situated in the upper-left corner of $\lambda$. We can then form the \bemph{skew shape} $\lambda/\mu$ by removing the Young diagram formed by $\mu$ from the Young diagram for $\lambda$. If $\mu$ is the empty partition, then $\lambda/\mu$ is just the Young diagram for $\lambda$.

A \bemph{semi-standard Young tableau} of \bemph{shape} $\lambda/\mu$ is a filling of the boxes of the Young diagram of $\lambda/\mu$ with positive integers such that the labeling weakly increases left-to-right along rows and strictly increases top-to-bottom down columns. Given such a tableau $Y$, let $\alpx^Y$ be the monomial defined as the product over all $i$ of $x_i^{n_i}$, where $n_i$ is the number of times $i$ appears in $Y$. The \bemph{skew Schur function} $s_{\lambda/\mu}$ is the sum over all $\alpx^Y$ for $Y$ a semi-standard Young tableaux of shape $\lambda/\mu$. In the case when $\mu$ is the empty partition, define $s_\lambda \coloneq s_{\lambda/\mu}$ to be the \bemph{Schur function} indexed by $\lambda$.

Most of the results in this paper will be in terms of \bemph{ribbon Schur functions}.
These symmetric functions are special instances of skew Schur functions indexed by skew shapes that are \bemph{ribbons}, which are connected skew shapes that do not contain a $2\times 2$ box.
Assuming that we draw our skew shapes following the English convention, we can associate a composition to a ribbon by counting the number of boxes in every row of the ribbon from bottom to top.
This association allows us to consider ribbon Schur functions as being indexed by compositions.
We refer to the ribbon Schur function indexed by a composition $\alpha$ as $r_{\alpha}$.
If $\alpha\vDash n$, we can alternatively define $r_\alpha$ as
\begin{equation}\label{eq: ribbon schur}
  r_\alpha = \sum_{\substack{w\in \bP^n,\\ \Dsc(w)=\set(\alpha)}} \alpx^w.
\end{equation}
A useful property of ribbon Schur functions is that the product of two ribbon Schur functions is a sum of two ribbon Schur functions,
\begin{equation}\label{eq: Product of Two Ribbons}
r_{\alpha}r_{\beta} = r_{\alpha\concat\beta}+r_{\alpha\nconcat\beta},
\end{equation}
which follows easily from \eqref{eq: ribbon schur}. More generally, suppose $\alpha^{(1)},\dots,\alpha^{(m)}$ are compositions such that $\alpha^{(i)}\vDash n_i$ and $n_1+\dots+ n_m = n$. Letting $\beta\coloneqq \alpha^{(1)}\concat\cdots\concat\alpha^{(m)}$ and $\delta \coloneqq \alpha^{(1)}\nconcat\cdots\nconcat\alpha^{(m)}$, we have that
\begin{equation}\label{eq: Product of Ribbons}
r_{\alpha^{(1)}}r_{\alpha^{(2)}}\dots r_{\alpha^{(m)}} = \sum_{\gamma\in [\beta,\delta]}r_\gamma,
\end{equation}
where $[\beta,\delta]$ is the interval between $\beta$ and $\delta$ in $\Comp_n$.

The following proposition gives a positive expansion of a ribbon Schur function in the basis of Schur functions. This is a special case of the expansion of a skew Schur function into Schur functions.
\begin{proposition}{\cite[Equation 2.2.4]{WachsPosetTopology}}\label{prop: ribbon into schur}
Given a composition $\alpha\vDash n$, we have
\begin{align*}
r_{\alpha}=\sum_{\lambda\vdash n}b_{\lambda,\alpha}s_{\lambda},
\end{align*}
where $b_{\lambda,\alpha}$ is the number of $T\in \SYT(\lambda)$ satisfying $\comp(\Dsc(T))=\alpha$. In particular, ribbon Schur functions are Schur positive.
\end{proposition}

Let $\ex$ denote the homomorphism from $\Sym$ to $\bQ [[x]]$ mapping $h_n$ to $\ifrac{x^n\!}{n!}$~\cite[Section~7.8]{Stanley-EC2}.
This homomorphism is known as the \bemph{exponential specialization} and, as mentioned in the introduction, has the property that the coefficient of $\ifrac{x^n\!}{n!}$ in the image $\ex(f)$ of a symmetric function $f$ is the coefficient of $x_1\cdots x_n$ in $f$ \cite[Proposition 7.8.4]{Stanley-EC2}.

Let $\asc(w)$ and $\dsc(w)$ be the number of ascents and descents in $w$ respectively.
We use a 2-parameter weighted power series analogue of the Eulerian polynomial,
\begin{align}\label{eqn: word version of multivariate Eulerian}
  A(\alpx;s,t)\coloneqq \sum_{w\in \bP^+} s^{\asc(w)}t^{\dsc(w)} \alpx^w.
\end{align}
By Equation~\eqref{eq: ribbon schur},
\begin{equation}
\label{eq: A(x;s,t)}
  A(\alpx;s,t) = \sum_{n\geq 1}\sum_{\alpha\vDash n}s^{n-\ell(\alpha)}t^{\ell(\alpha)-1}\,r_\alpha.
\end{equation}

\subsection{Unlabeled trees and associated notions}\label{subsec: unlabeled trees}
A \bemph{plane binary tree} is a rooted tree in which every node has at most two children, of which at most one is called a \bemph{left child} and at most one is called a \bemph{right child}.
We denote the set of all  plane binary trees by $\pbt$ and the set of  plane binary trees on $n$ nodes for $n\geq 1$ by $\pbtun$.
Recall that we use the term \emph{tree} to mean a plane binary tree.
Elements of $\pbt$ will be considered \bemph{unlabeled} trees.
We denote the set of nodes of $T$ by $\node{T}$, the set of edges of $T$ by $\edge{T}$, and the root of $T$ by $\rootof{T}$. We abuse notation on occasion and write $v\in T$ when we mean $v\in \node{T}$.
The nodes of a tree can be categorized as \bemph{terminal nodes}, which are nodes with no children, and \bemph{internal nodes}, which are nodes with at least one child.

 Given a binary tree $T\in\pbtun$, let $\overline{T}$ be the binary tree obtained by appending {two children to every terminal node and one child to every node which is the parent of a single child.} We call these appended nodes \bemph{leaves}, and we call $\overline{T}$ the \bemph{completion} of $T$. In total, the completion $\overline{T}$ has $n+1$ leaves.

Given a binary tree $T\in \pbt$, we can define a partial order $\leq_T$ on $\node{T}$ by drawing the tree with its root on top and leaves below and declaring this to be the Hasse diagram of the partial order.
Precisely, we define a relation such that for $v,w\in \node{T}$ with $v$ a child of $w$, we have $v <_T w$.
Then $\leq_T$ is defined as the transitive closure of this relation.
If $v<_T w$, then we say that $v$ is a \bemph{descendant} of $w$ and that $w$ is an \bemph{ancestor} of $v$.

Additionally, we work with two different types of total orderings on the nodes of a binary tree $T$ derived from a traversal of its nodes.
The \bemph{preorder traversal} is defined recursively, where we first visit the root, then traverse the right subtree of $T$ in preorder, and finally the left subtree of $T$ in preorder. It should be noted that this is slightly different from the usual convention for preorder traversal, where one traverses the left subtree before the right subtree.
The \bemph{inorder traversal} is also defined recursively, where we first traverse the left subtree of $T$ in inorder, then visit the root of $T$, and finally traverse the right subtree of $T$ in inorder.
If we order the nodes of a given tree $T$ according to when they are visited in the preorder (respectively inorder) traversal, we obtain a total order  on $\node{T}$ which we call \bemph{preorder} (respectively \bemph{inorder}) and denote by $\preccurlyeq_p$ (respectively $\preccurlyeq_i$).
Figure \ref{fig:PreInorder} gives an example each of preorder and inorder, respectively.
\begin{figure}[H]
\centering
  \includegraphics[scale=0.4]{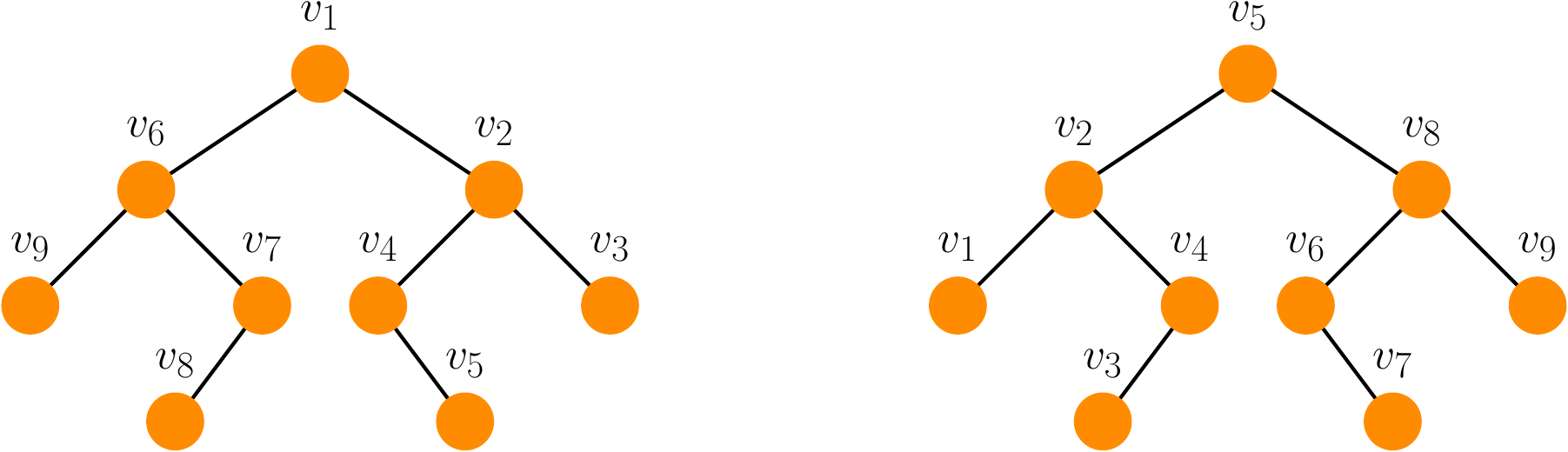}
  \captionof{figure}{The nodes are labeled according to preorder on the left and according to inorder on the right.}
  \label{fig:PreInorder}
\end{figure}

The preorder $\preccurlyeq_{p}$ on $T\in \pbtun$  allows us to associate a composition of  size $n$ with it, which we call the \bemph{composition type} of $T$ and denote by $\typ{T}$.
To compute $\typ{T}$, let $v_1 \prec_p \cdots \prec_p v_n$ be the nodes of $T$ in preorder.
Assume further that $v_{i_1}\prec_p\cdots \prec_p v_{i_k}$ are  all the terminal nodes in $T$. Note that $1\leq i_1<\cdots <i_k=n$.
We now define $\typ{T}\coloneq \comp{\{i_1,\dots, i_{k-1}\}}=(i_1,i_2-i_1,i_3-i_2,\dots, i_k-i_{k-1})$.
Clearly $\typ{T} \vDash n$.
The reader can verify that for the tree $T$ in Figure~\ref{fig:PreInorder}, we have $\typ{T}=(3,2,3,1)$.

We work under the convention that the edge $vw$ refers to the edge joining $v$ and $w$ in the tree, where $v$ comes before $w$ in inorder.
If $v$ is the left child of $w$, we say that $vw$ is the \bemph{left edge} of $w$. If $w$ is the right child of $v$, we say that $vw$ is the \bemph{right edge} of $v$.

Next, we give a definition of the canopy of a binary tree as a word on $\{U,D\}$. To translate between the canopy defined in \cite{Preville-Ratelle-Viennot} and the one defined here, simply flip the tree vertically and replace $U\leftrightarrow \bar{a}$ and $D\leftrightarrow b$ in the word $v(T)$ defined in their paper. See \cite[Proposition 2.2]{Preville-Ratelle-Viennot} for other equivalent definitions of canopy.
\begin{definition}\label{def: canopy2}
Given a binary tree $T\in \pbtun$, label each node $v$ except for the last node of $T$ in inorder with either a $D$ if $v$ has a right child or a $U$ if $v$ does not have a right child. Traverse the tree in inorder and read off the labels. The resulting word of length $n-1$ on $\{U,D\}$ is the \bemph{canopy} of $T$, denoted by $\can{T}$.
\end{definition}
Figure~\ref{fig:nodelabeling} shows a tree $T$ with its nodes labeled $U$ and $D$ according to Definition~\ref{def: canopy2}. Reading off the labels in inorder yields $\can{T}= UDUUDDUD$.

 \begin{figure}[h]
 	\centering
       \includegraphics[scale=0.4]{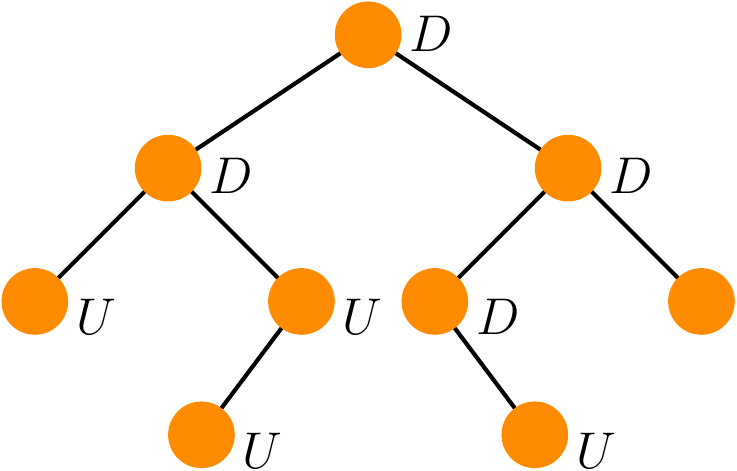}
       \caption{A tree $T$ with canopy $\can{T} = UDUUDDUD$.\label{fig:nodelabeling}}
 \end{figure}

\subsection{Labeled trees} \label{subsec: labeled trees}
A \bemph{labeled plane binary tree} (or simply a \bemph{labeled tree}) is a tree whose nodes have labels drawn from the set of positive integers $\bP$.
We denote the set of labeled trees by $\pbtl$ and the set of labeled trees on $n$ nodes for $n\geq 1$ by $\pbtln$.
Given $T\in \pbtln$, we denote by $\shape{T}$ the unlabeled tree obtained by removing the labels on the nodes of $T$.
Given a node $u$ in $T$, we refer to the label on $u$ as $u^{\ell}$.
 We associate two reading words with $T$: the \bemph{preorder reading word} denoted by $\pre{T}$, and the \bemph{inorder reading word} denoted by $\inorder(T)$. 
 If $v_1,\ldots,v_n$ are the nodes of $T$ in preorder, then $\pre{T}\coloneqq v_1^{\ell}\cdots v_n^{\ell}$.
On the other hand, if $v_1,\ldots,v_n$ are the nodes of $T$ in inorder, then $\inorder(T)\coloneqq v_1^{\ell}\cdots v_n^{\ell}$.

Let us recall from the Introduction that for a labeled tree, we have a refined classification for its edges given by left and right ascents and descents.
First, suppose that $v$ is the left child of $w$.
If $v^{\ell}\leq w^{\ell}$, then $vw$ is a \bemph{left ascent}. Otherwise it is a \bemph{left descent}.
Second, suppose that $v$ has a right child $w$.
If $v^{\ell}\leq w^{\ell}$, then $vw$ is a \bemph{right ascent}.  Otherwise it is a \bemph{right descent}.
Using this classification, we associate a weight $\weightt$ to a labeled tree $T$  as follows,
 \begin{equation}
\weightt\coloneqq\la^{\lasce(T)}\ld^{\ldes(T)}\ra^{\rasc(T)}\rd^{\rdes(T)}.
\end{equation}
Recall that $\lasce(T)$ (respectively, $\ldes(T)$, $\rasc(T)$, and $\rdes(T)$) is the number of left ascents (respectively, left descents, right ascents and right descents) in $T$. 
For the labeled tree $T$ in Figure~\ref{fig:Ascents-Descents},  each edge is labeled with $\la$, $\ld$, $\ra$, or $\rd$, corresponding to the edge's orientation and whether the labels on the edge form an ascent or descent.  {It may help the reader to remember that $\lambda$ corresponds to left edges, $\rho$ corresponds to right edges, unbarred parameters correspond to strict inequalities, and barred parameters correspond to weak inequalities.}
\begin{figure}[ht]
  \includegraphics[scale=0.5]{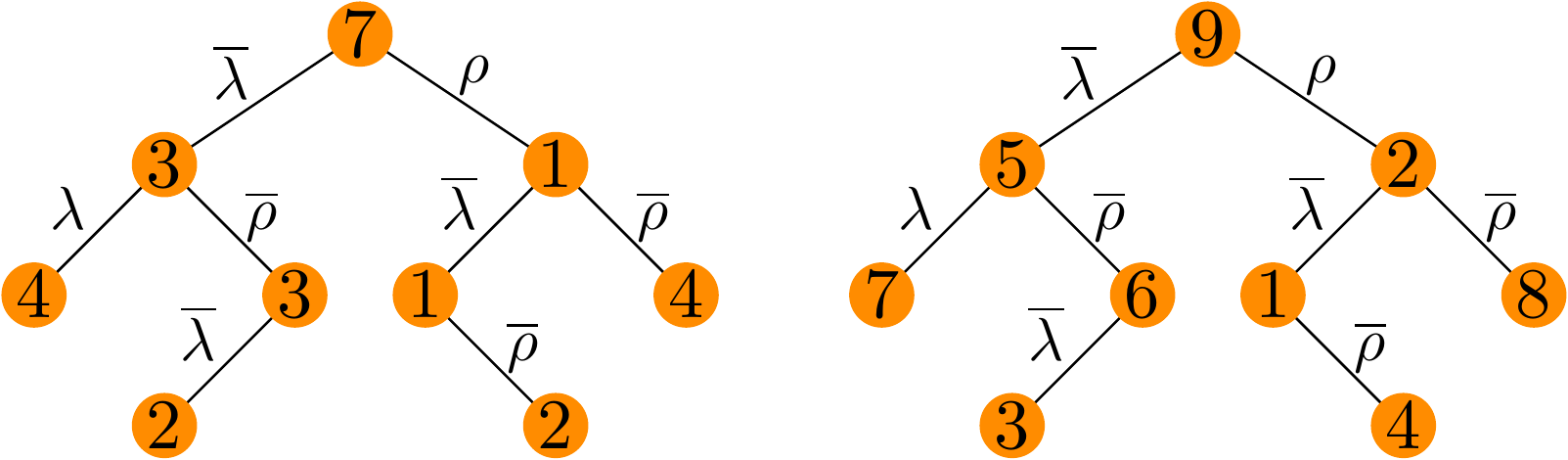}
  \captionof{figure}{On the left, a tree with weight $\weightt=\la^{3}\ld^{1}\ra^{3}\rd^{1}$. On the right, the inorder standardization  $\std(T)$.}
  \label{fig:Ascents-Descents}
\end{figure}

Letting $v_1,v_2,\dots,v_n$ be the nodes of $T$ listed in inorder, we can obtain a standard labeled tree by relabeling node $v_i$ with the $i$th letter of the permutation $\std(\inorder(T))$. We call the resulting standard labeled tree the \bemph{inorder standardization} of $T$, denoted by $\std(T)$. Given an edge $v_iv_j$ in $T$, it forms a (left or right) descent if and only if $(i,j)\in \Inv(\inorder(T))$. Since standardization preserves the inversion set of a word, we have $\wt{T} = \wt{\std(T)}$.

\newcommand{\lleaf}{\lambda_\circ}
\newcommand{\rleaf}{\rho_\circ}
\newcommand{\lmarked}{\lambda_m}
\newcommand{\rmarked}{\rho_m}
\section{Two functional equations for \texorpdfstring{$G$}{G}}\label{sec:first proof}
We begin by proving Theorem~\ref{thm: Gessel functional equation} which establishes a functional equation for $G$. We then use this identity to give our first proof of Theorem~\ref{thm: ribbon functional equation}.

A labeled tree is \bemph{increasing} if it has no left ascents or right descents. The term \emph{increasing} is motivated by the fact that if $T$ is an increasing tree, then the labels increase \textemdash{weakly toward the right and strictly toward the left}\textemdash{} along the path from the root of $T$ to any terminal node. The next lemma shows that these trees are in bijection with words on positive integers.
\begin{lemma}\label{lem: increasing trees}
Fix $n$ and $1\leq j \leq n$. Then the map $\varphi_n(T)=\inorder(T)$ from the set of increasing binary trees $T$ on $n$ nodes to $\bP^n$ is a bijection. Furthermore, $\varphi_n$ restricts to a bijection between the subset of trees whose completion has $j$ right leaves and the subset of words with $j-1$ descents.
\end{lemma}

\begin{proof}
 {First, let $T$ be an increasing binary tree on $n$ nodes whose completion has $j$ right leaves, and let $w = \inorder(T)$. It can be checked that the $i$th node of $T$ in inorder has no right child if and only if $i\in \Des(w)$ or $i=n$.} Recalling the definition of the completion $\overline{T}$ from Subsection~\ref{subsec: unlabeled trees}, a node of $T$ has a right leaf in $\overline{T}$ if and only if it has no right child. Therefore, we see that $j = \#\Des(w) + 1$, so $w$ has $j-1$ descents.

Next, we recursively define a map $\psi_n$ from $\bP^n$ to the set of increasing binary trees on $n$ nodes which maps a word with $j-1$ descents to a tree whose completion has $j$ right leaves. Let $\psi_1$ be the map which sends a word $w_1$ with length one to the tree on one node labeled with $w_1$. Clearly, $\psi_1$ is the inverse of $\phi_1$. For $n>0$, suppose we have constructed $\psi_i$ for $i<n$ such that $\phi_i\circ\psi_i(w) = \inorder(\psi_i(w)) = w$. Fixing $w\in \bP^n$ with $j-1$ descents, let us define $\psi_n(w)$ as follows. Let $i$ be the smallest index of the smallest letter in $w$. Let $T_1 = \psi_{i-1}(w_1\dots w_{i-1})$, and let $T_2 = \psi_{n-i}(w_{i+1}\dots w_n)$. Then define $\psi_n(w)=T$ to be the tree whose root is labeled $w_i$ and whose left and right subtrees are $T_1$ and $T_2$, respectively.

 {Finally, we show that $\psi_n$ is the inverse of $\phi_n$. By our inductive assumptions on the maps $\psi_{i-1}$ and $\psi_{n-i}$, we have that $\inorder(T_1) = w_1\dots w_{i-1}$ and $\inorder(T_2) = w_{i+1}\dots w_n$. Therefore, we have $\phi_n\circ\psi_n(w) = \inorder(T) = \inorder(T_1) w_i \inorder(T_2) = w$ by the recursive definition of the inorder traversal. Furthermore, it follows from our choice of $i$ that $T$ is increasing, hence $\psi_n$ is the right inverse of $\phi_n$. We leave it to the reader to check that $\psi_n$ is also the left inverse of $\phi_n$, hence it is the inverse of $\phi_n$. This completes our recursive construction of $\psi_n$, and hence $\phi_n$ is a bijection with inverse $\psi_n$.}
\end{proof}

\begin{proof}[Proof of Theorem~\ref{thm: Gessel functional equation}]
Recall that Theorem~\ref{thm: Gessel functional equation} claims that 
\begin{equation}
\label{eq: functional equation2}
\frac{(1+\la G)(1+\ra G)}{(1+\ld G)(1+\rd G)}=H\bigl((\la\ra-\ld\rd)G+\la+\ra-\ld-\rd\bigr).
\end{equation}
Let $\leftleaves(T)$ and $\rightleaves(T)$ denote the number of left leaves and right leaves in the completion $\overline{T}$, respectively. Let
\begin{align}
D(\alpx;\la, \ld, \lleaf,\ra,\rd,\rleaf)=\sum_{n\ge1}\sum_{T\in \pbtln} \la^{\lasce(T)}\ld^{\ldes(T)}\lleaf^{\leftleaves(T)}\ra^{\rasc(T)}\rd^{\rdes(T)}\rleaf^{\rightleaves(T)}
\alpx^T.
\end{align}

In $\overline{T}$, since every node has a left child that is either a node or a leaf, we have
\begin{equation}
n=\lasce(T) + \ldes(T) + \leftleaves(T)
\end{equation}
and similarly,
\begin{equation}
n = \rasc(T) + \rdes(T) + \rightleaves(T).
\end{equation}
Thus,
\begin{equation}
\label{e-DB}
D(\alpx;\la,\ld,\lleaf,\ra,\rd,\rleaf) =  G(\lleaf \rleaf\alpx;\la/\lleaf, \ld/\lleaf,\ra/\rleaf, \rd/\rleaf),
\end{equation}
where $\lleaf \rleaf \alpx$ means that each variable $x_i$ is replaced with $\lleaf \rleaf x_i$.

Let $T$ be a labeled tree.
 {If $pq$ is a left ascent of $T$, such that $p$ is the left child of $q$, we call $p$ a \bemph{left ascent-child}. If $pq$ is a right descent of $T$, such that $q$ is the right child of $p$, we call $q$ a \bemph{right descent-child}.}
We obtain a \bemph{marked tree} from $T$  by ``marking'' some of the nodes of~$T$. We require that every left ascent-child and right descent-child must be marked, the root must not be marked, and other nodes may be either marked or unmarked. 
 {
See Figure~\ref{fig:MarkedTree} for an example of a marked tree.
}
Let $\leftmarks(M)$ be the number of marked left children in $M$, and let $\rightmarks(M)$ be the number of marked right children in $M$.

\begin{figure}[t]
\includegraphics[scale=0.4]{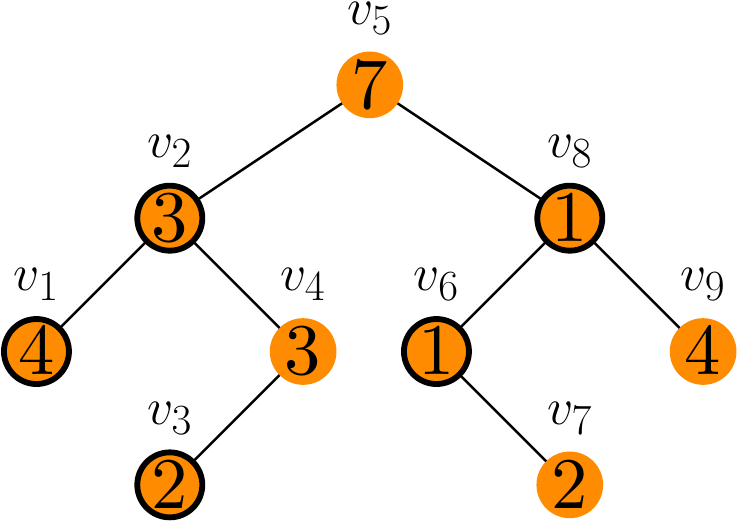}
\caption{An example of a marked tree on $9$ nodes with nodes $v_1$, $v_2$, $v_3$, $v_6$, and $v_8$ in inorder marked.\label{fig:MarkedTree}}
\end{figure}

Let
\begin{equation}
\label{e-Cn-def}
C(\alpx;\lleaf,\lmarked,\rleaf,\rmarked)=\sum_{M}
 \lleaf^{\leftleaves(M)}\lmarked^{\leftmarks(M)}\rleaf^{\rightleaves(M)}\rmarked^{\rightmarks(M)}
 \alpx^M
\end{equation}
where the sum is over all  marked  labeled  trees  $M$ {and where $\alpx^M$, $\leftleaves(M)$ and $\rightleaves(M)$ are equal to $\alpx^T$, $\leftleaves(T)$, and $\rightleaves(T)$ for $T$ the underlying tree of $M$, respectively.}
Then by the definition of marked trees we have
\begin{equation}
C(\alpx;\lleaf,\lmarked,\rleaf,\rmarked) = D(\alpx;\lmarked,1+\lmarked,\lleaf,1+\rmarked, \rmarked,\rleaf),
\end{equation}
and it follows from \eqref{e-DB} that
\begin{equation}
\label{e-CB}
C(\alpx;\lleaf,\lmarked,\rleaf,\rmarked) = G\left(\lleaf \rleaf \alpx;\frac{\lmarked}{\lleaf}, \frac{1+\lmarked}{\lleaf},\frac{1+\rmarked}{\rleaf}, \frac{\rmarked}{\rleaf}, \right).
\end{equation}
Substituting $x_i\mapsto x_i/(\lleaf\rleaf)$, $\lleaf\mapsto 1/(\ld-\la)$, $\lmarked\mapsto \la/(\ld-\la)$, $\rleaf\mapsto 1/(\ra-\rd)$ and $\rmarked\mapsto \rd/(\ra-\rd)$  into \eqref{e-CB}, we find that
\begin{equation}
\label{e-BC}
 G(\alpx;\la,\ld,\ra,\rd)
=  C\left((\ld-\la)(\ra-\rd)\alpx;\frac{1}{\ld-\la}, \frac{\la}{\ld-\la},\frac{1}{\ra-\rd},\frac{\rd}{\ra-\rd}
\right).
\end{equation}
We shall find a functional equation for
$C = C(\alpx; \lleaf,\lmarked,\rleaf,\rmarked)$ which through \eqref{e-BC} will yield \eqref{eq: functional equation2}.

Next we count increasing labeled trees by left and right leaves. Let
\begin{equation}
K(\alpx; \lleaf,\rleaf) = \sum_T \lleaf^{\leftleaves(T)}\rleaf^{\rightleaves(T)}\alpx^T,
\end{equation}
where the sum is over all increasing trees $T$.
Recall the 2-parameter weighted power series analogue of the Eulerian polynomial $A(\alpx; s,t)$ defined in \eqref{eqn: word version of multivariate Eulerian}.
By the homogenized version of a  result of MacMahon \cite[Vol. 1, p.~186]{MacMahon}, we have that
\begin{equation}\label{eq: MacMahon eulerian}
  A(\alpx;s,t) = \frac{\sum_{n\geq 1} (s-t)^{n-1}h_n}{1-t\sum_{n\geq 1}(s-t)^{n-1}h_n} = \frac{H(s-t)-1}{s-tH(s-t)}.
\end{equation}
Therefore, we have
\begin{equation}\label{e-KH}
K(\alpx; \lleaf,\rleaf)=\lleaf \rleaf A(\alpx; \lleaf,\rleaf) =  \lleaf\rleaf\frac{H(\lleaf-\rleaf)-1}{\lleaf-\rleaf H(\lleaf-\rleaf)},
\end{equation}
where the first equality follows from Lemma~\ref{lem: increasing trees}, and the second equality follows from \eqref{eq: MacMahon eulerian}.

We claim that $C$ satisfies the functional equation
\begin{equation}
\label{e-CA}
C = K(\alpx; \lleaf+\lmarked C,\rleaf+\rmarked C).
\end{equation}
To see this, given a marked  labeled tree $M$ on $[n]$, let  $P$ be the set of unmarked nodes $p$ of
$M$ with the property that no  ancestor of $p$ is marked. It is clear
that the induced subtree of $M$ on $P$ is an increasing tree $I$. Then $M$ can be recovered from $I$ by attaching marked trees to the leaves of $I$ and marking the roots of the attached trees.  Equation \eqref{e-CA} follows from this decomposition.

Using \eqref{e-KH}, we may expand \eqref{e-CA} as
\begin{equation}
\label{e-CH}
C =  (\lleaf+\lmarked C)(\rleaf+\rmarked C) \frac{H(\lleaf+\lmarked C-\rleaf-\rmarked C)-1}{(\lleaf+\lmarked C)-(\rleaf+\rmarked C) H(\lleaf+\lmarked C-\rleaf-\rmarked C)}.
\end{equation}
Solving for $H$ in \eqref{e-CH} gives
\begin{equation}
\label{e-Cfe}
\frac{(\rleaf+(\rmarked+1)C)(\lleaf+\lmarked C)}{(\lleaf+(\lmarked+1)C)(\rleaf+\rmarked C)}=H(\lleaf-\rleaf+(\lmarked-\rmarked)C).
\end{equation}
Finally, set $\lleaf=1/(\ld-\la)$, $\lmarked=\la/(\ld-\la)$, $\rleaf=1/(\ra-\rd)$ and $\rmarked=\rd/(\ra-\rd)$ in \eqref{e-Cfe}, and replace $\alpx$ with $(\ld-\la)(\ra-\rd)\alpx$. Applying \eqref{e-BC} to the resulting identity gives \eqref{eq: functional equation2}.
\end{proof}

\begin{remark}Our proof above of Theorem  \ref{thm: Gessel functional equation} is based on a proof given in \cite{gessel-forests} of a result equivalent to the case $\la=0$ of \eqref{eq: B functional equation}.
The paper \cite{gessel-forests} counts forests of labeled rooted trees by descents and leaves instead of counting binary trees, but one can convert between forests of rooted trees and binary trees by applying a simple bijection.
\end{remark}

As a corollary to Theorem~\ref{thm: Gessel functional equation}, we see that $G$ has some surprising symmetries. We state these next in Corollary~\ref{cor: symmetries}.
\begin{corollary}\label{cor: symmetries}
The following identities hold,
\begin{align}
G(\alpx;\la,\ld,\ra,\rd) &= G(\alpx;\ra,\ld,\la,\rd) = G(\alpx;\la,\rd,\ra,\ld),\label{eq: symmetries}\\
\omega(G(\alpx;\la,\ld,\ra,\rd)) &= G(\alpx;\ld,\la,\rd,\ra).\label{eq: omega}
\end{align}
\end{corollary}

\begin{proof}
We can rewrite the functional equation \eqref{eq: functional equation2} as
\begin{align}\label{eq: recursive fe}
G = (1+\ld G)(1+\rd G)\sum_{n\geq 1} ((\la\ra- \ld\rd)G + \la +\ra - \ld - \rd)^{n-1}\,h_n.
\end{align}
This gives a recursive method for computing $G_n$ in terms of $G_m$ for $m<n$. Therefore, $G$ is uniquely determined by the functional equation. Identity \eqref{eq: symmetries} then follows from the fact that the functional equation is invariant under swapping $\la$ and $\ra$ and under swapping $\ld$ and $\rd$.
To see identity \eqref{eq: omega}, apply $\omega$ to both sides of the functional equation \eqref{eq: functional equation2}, and then use the fact that $H(-z)E(z) = 1$.
\end{proof}

Letting $A\coloneqq A(\alpx;s,t)$, we may rewrite \eqref{eq: MacMahon eulerian} as
\begin{equation}\label{eq: AH identity}
  H(s-t) = \frac{1+sA}{1+tA}.
\end{equation}
At this point, we have all the tools we need to establish Theorem~\ref{thm: ribbon functional equation}, which also implies Theorem~\ref{conj: schur positivity}.
\begin{proof}[Proof 1 of Theorem \ref{thm: ribbon functional equation}]
Substituting $s= \la\ra G+\la+\ra$ and $t=\ld\rd G+\ld+\rd$ into \eqref{eq: AH identity}, then $H(s-t)$ becomes  the right-hand side of \eqref{eq: functional equation2}. Thus, we can rewrite \eqref{eq: functional equation2} as
\begin{align}\label{eq: intermediate in ribbon functional equation}
  \frac{(1+\la G)(1+\ra G)}{(1+\ld G)(1+\rd G)} = \frac{1+sA}{1+tA}.
\end{align}
Since  $(1+\la G)(1+\ra G)=1+sG$ and $(1+\ld G)(1+\rd G) = 1+tG$, the identity \eqref{eq: intermediate in ribbon functional equation} is equivalent to
\begin{equation}
\frac{1+sG}{1+tG}=\frac{1+sA}{1+tA},
\end{equation}
which implies that
\begin{equation}
G=A=A(\alpx;\la\ra \,G+\la+\ra,\ld\rd \,G+\ld+\rd),
\end{equation}
and Theorem \ref{thm: ribbon functional equation} follows by \eqref{eq: A(x;s,t)}.
\end{proof}

\begin{remark}
Theorem~\ref{thm: ribbon functional equation} together with~\eqref{eq: Product of Two Ribbons} gives a way to recursively expand each $G_n$ positively in terms of ribbon Schur functions.
Therefore, the Schur positivity of $G$ follows immediately from Theorem~\ref{thm: ribbon functional equation} by using Proposition~\ref{prop: ribbon into schur}.
However, it is not clear how to prove the Schur positivity of the more refined generating functions $G_{n,\canopy}$ using these same techniques. Instead, we prove Theorem~\ref{thm: ribbon expansion for fixed canopy} in Section~\ref{sec: ribbon expansion} after developing our weight-preserving bijection in Section~\ref{sec: Proof of Main Theorem}. Theorem~\ref{conj: refined Schur positivity} then follows from Theorem~\ref{thm: ribbon expansion for fixed canopy} by Proposition~\ref{prop: ribbon into schur}.
\end{remark} 

\section{Marked and augmented interlacing partitions}\label{sec: fixed canopy}
In this section, we present the definitions and notation used in the statements of Theorems~\ref{cor: Ribbon Expansion} and~\ref{thm: ribbon expansion for fixed canopy}.
We delay the proofs of these two theorems until Section~\ref{sec: ribbon expansion} after developing all of the key concepts in Section~\ref{sec: Proof of Main Theorem}.

\subsection{Noncrossing partitions}
We need some notions concerning the lattice of noncrossing partitions, so we recall the relevant definitions briefly. A \bemph{partition} $\pi$ of $[n]$ is a collection of pairwise disjoint nonempty subsets $B_1,\dots,B_k$ whose union is $[n]$.
We write this as $\pi\coloneqq B_1/\cdots/B_k$ where the $B_i$ are ordered in increasing order of their minimal elements.
The subsets $B_1,\dots,B_k$ are the \bemph{blocks} of $\pi$, and the number of blocks is denoted by $\blk(\pi)$.
The set of partitions of $[n]$, denoted by $\Pi_n$, can be endowed with the structure of a \bemph{graded lattice} by defining a partial order as follows.
Given partitions $\sigma$ and $\tau$, we say that $\sigma \preccurlyeq_{\Pi_n} \tau$ if each block in $\sigma$ is contained in a block in $\tau$.
In particular, a partition $\tau$ covers a partition $\sigma$ in $\Pi_n$ if $\tau$ is obtained by merging two distinct blocks in $\sigma$.
If $\sigma \preccurlyeq_{\Pi_n}\tau$, we say that $\sigma$ is \bemph{finer} than $\tau$ or equivalently that $\tau$ is \bemph{coarser} than $\sigma$.
The rank of $\pi\in \Pi_n$ is given by $n-\blk(\pi)$.
The partition of $[n]$ into singleton sets gives the unique minimal element in $\Pi_n$, and the partition of $[n]$ consisting of a single block gives the unique maximal element.

We identify a  partition of $[n]$ with its \bemph{arc diagram}, which is defined as follows.
Consider $n$ nodes $v_1,\dots,v_n$ representing  the integers $1$ through $n$ from left to right, and connect two nodes $v_i$ and $v_j$ with $i<j$ by an undirected arc if $i$ and $j$ belong to the same block and if there is no $k$ in that block such that $i<k<j$.
If $i$ and $j$ belong to the same block in  $\pi$, then we denote this equivalence by $i\sim_{\pi} j$.

A partition $\pi$ of $[n]$ is said to be \bemph{noncrossing} if there do not exist $1\leq a<b<c<d\leq n$ such that $a\sim_\pi c$, $b\sim_\pi d$ and $a\nsim_\pi b$.
The set of noncrossing partitions of $n$, denoted by $\noncrossing{n}$, inherits a  graded lattice structure from that on $\Pi_n$.
For the many interesting properties of $\noncrossing{n}$, the reader is referred to the beautiful survey by Simion \cite{Simion} and references therein.
For a more recent survey on the relevance of $\noncrossing{n}$ in various areas of mathematics, the reader is referred to McCammond \cite{McCammond}.

\subsection{Interlacing partitions}
We now introduce a special type of noncrossing partition, called an interlacing partition, and decorated generalizations of interlacing partitions. 

Let $\pi=B_1/\dots/B_k \in \noncrossing{n}$.
If $i\sim_{\pi} i+1$, then we say the node $v_i$ is a \bemph{stepper} and the arc connecting $v_i$ and $v_{i+1}$ is a \bemph{short arc}.
If $i\sim_{\pi} j$ with $j\geq i+2$ and there is an arc connecting $v_i$ and $v_j$ in the arc diagram, then we say the node $v_i$ is a \bemph{jumper} and the arc between $v_i$ and $v_j$ is a \bemph{long arc}.
Let $\max(\pi)\coloneqq\{\max(B_1),\ldots, \max(B_k)\}$ and $\min(\pi)\coloneqq\{\min(B_1),\ldots,\min(B_k)\}$.
We say that $\pi$ is \bemph{interlacing} if $i\in \max(\pi)$ implies that $i+1\notin \min(\pi)$ for all $i$.
Finally, an interlacing partition $\pi\in \noncrossing{n}$ is said to be \bemph{marked} if a subset of nodes in $\{v_1,\dots,v_n\} \setminus \{v_i \suchthat i\in \max(\pi)\}$ is marked.
Let $\mnoncrossing{n}$ denote the set of all marked interlacing partitions whose underlying interlacing partition is in $\noncrossing{n}$.
For the marked interlacing partition $\pi$ in Figure~\ref{fig: Marked NC}, the nodes $v_2$, $v_4$, and $v_8$ are marked.
Furthermore, the nodes $v_1$, $v_2$, $v_7$, and $v_9$ are steppers, and the nodes $v_3$, $v_4$, and $v_8$ are jumpers.

\begin{figure}[H]
  \includegraphics[scale=0.3]{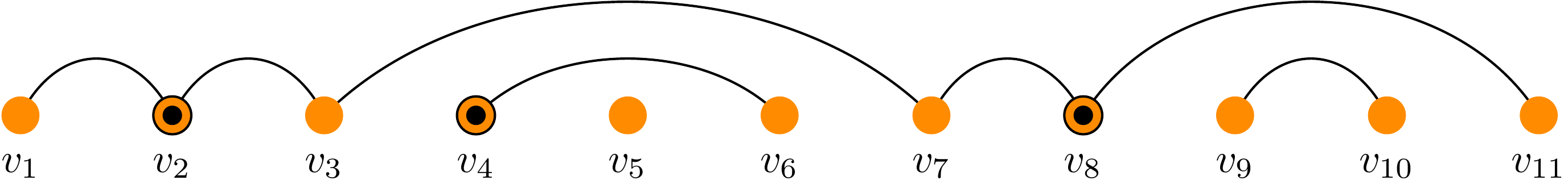}
  \caption{The arc diagram of the interlacing partition \\$\pi = 1\,2\,3\,7\,8\,11/4\,6/5/9\,10$, with nodes $v_2,v_4$, and $v_8$ marked.}
  \label{fig: Marked NC}
\end{figure}

\begin{remark}
  The cardinality of the set of interlacing partitions for $n\geq 1$ is given by the sequence of Motzkin numbers \cite[A001006]{OEIS}, which are well known to enumerate Motzkin paths, which are lattice paths from $(0,0)$  to $(n,0)$ where the steps allowed are up, down and level.
  The cardinality of the set of marked interlacing partitions for $n\geq 1$ is given by \cite[A071356]{OEIS}, which counts Motzkin paths where the up and level steps are bicolored.
  In theory we could have phrased our results in the language of Motzkin paths but it is (marked) interlacing partitions that arise naturally in our context.
\end{remark}

To each $\pi\in \mnoncrossing{n}$, we associate a sequence of compositions as follows.
Partition each block $B$ of $\pi$ into disjoint subsets by breaking $B$ after each marked node.
Define $c(B)$ to be the composition obtained by recording the sizes of these subsets.
For $\pi = B_1/B_2/B_3/B_4$ in Figure~\ref{fig: Marked NC}, we have
\begin{align*}
B_1 &= \{1,2,3,7,8,11\} & c(B_1) & = (2,3,1), \\
B_2 &= \{4,6\}  & c(B_2) & = (1,1),\\
B_3 &= \{5\} & c(B_3) & = (1),\\
B_4 &= \{9,10\}  & c(B_4) & = (2).
\end{align*}
Consider the following four statistics associated to a marked interlacing partition $\pi$. Let
\begin{align}
\sa(\pi) &= \# \text{ Unmarked steppers in }\pi \label{eq: step stat1}\\
\sd(\pi) &= \# \text{ Marked steppers in }\pi\\
\ja(\pi) &= \# \text{ Unmarked jumpers in }\pi\\
\jd(\pi) &= \# \text{ Marked jumpers in }\pi.\label{eq: step stat4}
\end{align}
Let the weight of a marked interlacing partition be
\begin{equation}
\wt{\pi} = (\la+\ra)^{\,\sa(\pi)} (\ld+\rd)^{\sd(\pi)} (\la\ra)^{\ja(\pi)} (\ld\rd)^{\jd(\pi)}.
\end{equation}
\begin{remark}
The letters $a$ and $d$ in the names of the statistics in \eqref{eq: step stat1}--\eqref{eq: step stat4} correspond to \emph{ascent} and \emph{descent}. We will see in Section~\ref{sec: ribbon expansion} how each marked interlacing partition $\pi$ corresponds to a set of labeled binary trees, where each marked node in $\pi$ represents a descent in the labeling, and each unmarked node represents an ascent in the labeling.
\end{remark}

We define an \bemph{augmented interlacing partition}  $\pi^\ast$ to be a marked interlacing partition such that each short arc in its arc diagram is labeled either $U$ or $D$.
To each augmented interlacing partition $\pi^\ast$, we associate the following six statistics. Let
\begin{align*}
\sau(\pi^\ast) &= \#\text{ Unmarked steppers in $\pi^\ast$ whose corresponding short arc is labeled $U$,}\\
\sdu(\pi^\ast) &= \#\text{ Marked steppers in $\pi^\ast$ whose corresponding short arc is labeled $U$,}\\
\sad(\pi^\ast) &= \#\text{ Unmarked steppers in $\pi^\ast$ whose corresponding short arc is labeled $D$,}\\
\sdd(\pi^\ast) &= \#\text{ Marked steppers in $\pi^\ast$ whose corresponding short arc is labeled $D$,}\\
\ja(\pi^\ast) &= \# \text{ Unmarked jumpers in }\pi^\ast,\\
\jd(\pi^\ast) &= \# \text{ Marked jumpers in }\pi^\ast.
\end{align*}
 Combining all six statistics, we define
\begin{equation}\label{eq: wt of pi ast}
\wt{\pi^\ast}=\la^{\sau(\pi^\ast)} \ld^{\sdu(\pi^\ast)} \ra^{\,\sad(\pi^\ast)} \rd^{\,\sdd(\pi^\ast)} (\la\ra)^{\,\ja(\pi^\ast)} (\ld\rd)^{\,\jd(\pi^\ast)}.
\end{equation}
For $\pi^\ast$ in Figure~\ref{fig: Augmented NC}, we have {$\wt{\pi^\ast} = \ra\ld(\la\ra)(\ld\rd)\ra(\ld\rd)\la = \la^2\ld^3\ra^3\rd^2$}.

To each augmented interlacing partition $\pi^\ast$, associate words $w(\pi^\ast)$ and $\hat{w}(\pi^\ast)$ defined recursively as follows.
If $\pi^\ast$ contains a single block, define $w(\pi^\ast)$ to be the word obtained by recording the labels on the short arcs from left to right, and define $\hat{w}(\pi^\ast)$ to be $Dw(\pi^\ast)U$.
If $\pi^\ast$ contains more than 1 block, let $B_1$ be the block of $\pi^\ast$ which contains $v_1$ and $v_n$.
Partition $B_1$ into blocks $C_1,\dots, C_p$ that are maximal under connectedness by short arcs.
 For $1\leq i\leq p-1$, let $\pi^\ast_i$ denote the augmented interlacing partition induced by $\pi^\ast$ on the nodes $v_j$ for $j\in (\max(C_i),\min(C_{i+1}))$.
Define
 \begin{equation}\label{eq: def of w}
 w(\pi^\ast)\coloneqq w(C_1)\hat{w}(\pi^\ast_1)\cdots w(C_{p-1})\hat{w}(\pi^\ast_{p-1})w(C_p)\in \{U,D\}^{n-1},
 \end{equation}
 where the dots signify that the concatenation continues and $n$ is the number of nodes of $\pi$.
Given $\canopy\in \{U,D\}^{n-1}$, let $\anoncrossing{n}{\canopy}$ be the set of augmented interlacing partitions $\pi^\ast$ such that $w(\pi^\ast)=\canopy$. For the augmented interlacing partition $\pi^\ast$ in Figure~\ref{fig: Augmented NC}, we have that $\pi^\ast \in \anoncrossing{11}{DUDDUUDDUU}$.

\begin{figure}[H]
  \includegraphics[scale=0.3]{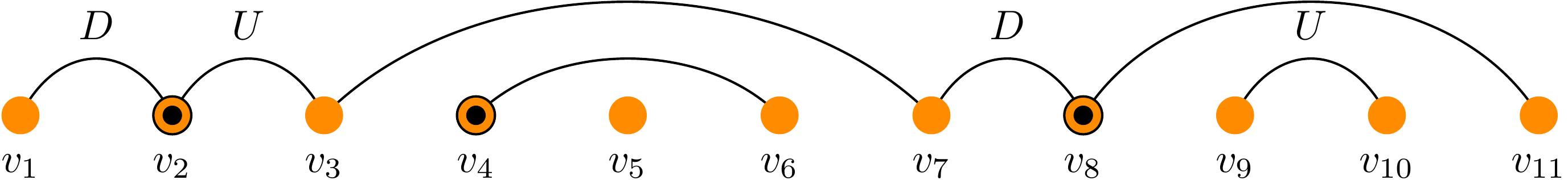}
  \caption{An augmented interlacing partition $\pi^\ast$ with $w(\pi^\ast) = DUDDUUDDUU$.}
  \label{fig: Augmented NC}
\end{figure}

\section{\texorpdfstring{$\mathfrak{S}_n$}{Sn}-modules from deformations of Coxeter arrangements}\label{sec: Hyperplane arrangements}
In this section, we provide a representation-theoretic meaning to some of the equalities in the introduction relating specializations of $B_n(\alpx;\la,\ld,\ra,\rd)$  to the number of regions in hyperplane arrangements related to the Coxeter arrangement. In particular, we focus on the semiorder and Linial arrangements.
While there is an obvious symmetric group action on regions of the semiorder arrangement $\hypsemi_n$ by permutation of the coordinates, the question of finding one on regions of the Linial arrangement $\hyplin_n$ is a bit more subtle.
Throughout this section, we assume the ribbon Schur expansions in Theorem~\ref{cor: Ribbon Expansion} and Theorem~\ref{thm: ribbon expansion for fixed canopy}, which are proven in Section~\ref{sec: ribbon expansion}.

We briefly introduce our notation pertaining to hyperplane arrangements.
For a detailed introduction, we refer the reader to \cite{Orlik-Terao, Stanley-notes}.
A \bemph{hyperplane arrangement} is a finite collection of affine hyperplanes in a vector space.
Let $\mathcal{A}$ be a hyperplane arrangement in a finite-dimensional vector space $V$ over $\mathbb{R}$.
A \bemph{region} of $\mathcal{A}$ is a connected component of $V-\bigcup_{H\in \mathcal{A}}H$.
We denote the set of regions of $\mathcal{A}$ by $\regions{\mathcal{A}}$ and the number of regions of $\mathcal{A}$ by $r(\mathcal{A}) = |\regions{\mathcal{A}}|$.

\subsection{Semiorder arrangements}\label{subsec: Semiorders}
Recall that the semiorder arrangement $\hypsemi_n$ is  the hyperplane arrangement in $\mathbb{R}^n$ given by  the hyperplanes $H_{ij}: x_i-x_j=1$ for $1\leq i\neq j\leq n$.
Note that this set of hyperplanes is stable under the natural action of the symmetric group $\mathfrak{S}_n$.
This implies that $\regions{\hypsemi_n}$ inherits an action of $\mathfrak{S}_n$.
We first expand a certain specialization of $G_n$ in terms of the Motzkin numbers and the Frobenius characteristics of Foulkes characters.
We then use this formula to show in Theorem~\ref{thm: Frobenius semiorder} that the Frobenius characteristic of the action on $\regions{\hypsemi_n}$ is a specialization of  $G_n$.

Figure~\ref{fig:semiorder action} demonstrates the action of $\mathfrak{S}_3$ on $\regions{\hypsemi_3}$.
In this case, we can compute the Frobenius characteristic of this action to be $h_3+2h_{21}+2h_{111}$.
In terms of ribbon Schur functions, this may be written as $5r_{3}+3(r_{21}+r_{12})+2r_{111}$.
Note that since ribbon Schur functions do not form a basis for the ring of symmetric functions, this  expansion in terms of ribbon Schur functions is not unique.
\begin{figure}[ht]
\centering
\begin{minipage}{.48\textwidth}
  \centering
  \includegraphics[scale=0.6]{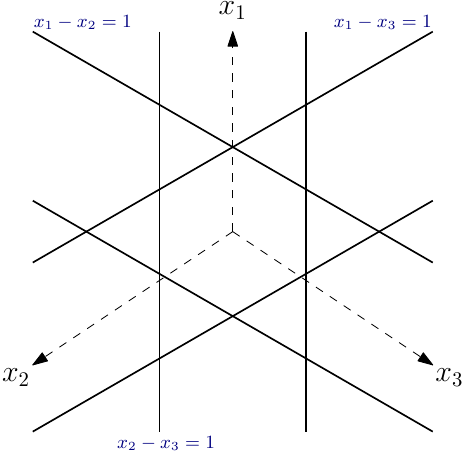}
\end{minipage}
\begin{minipage}{.48\textwidth}
  \centering
  \includegraphics[scale=0.6]{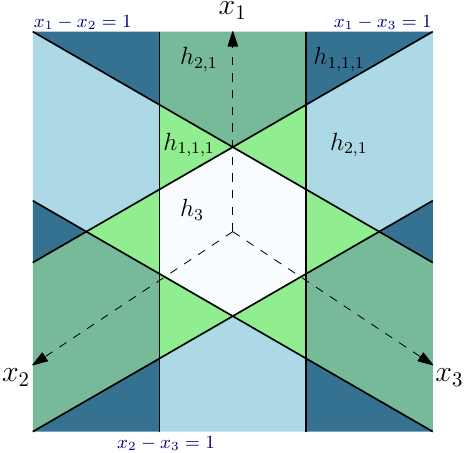}
\end{minipage}
\captionof{figure}{On the left is the arrangement $\hypsemi_3$ (projected onto $x_1+x_2+x_3=0$).
On the right is a depiction of the $\mathfrak{S}_3$-action on $\regions{\hypsemi_3}$ with regions of the same color belonging to the same orbit (see the electronic version).}
\label{fig:semiorder action}
\end{figure}

By Theorem~\ref{cor: Ribbon Expansion}, we know $G_n$ expands in terms of ribbon Schur functions with coefficients which are polynomials in $\la\ra$, $\la+\ra$, $\ld\rd$, and $\ld+\rd$.
Consider the case   $\la=\ra=1$, $\ld=\zeta_6^{-1}$, and $\rd=\zeta_6$, where $\zeta_6$ is a primitive sixth root of unity. 
Then we have that {$\la\ra=\ld+\rd=\ld\rd=1$}, whereas $\la+\ra=2$.
Thus, for this specialization we have $\wt{\pi}=2^{\sa(\pi)}$ for  $\pi\in \mnoncrossing{n}$.
From Theorem~\ref{cor: Ribbon Expansion} and~\eqref{eq: Product of Ribbons}, we obtain
\begin{equation}\label{eqn: semiorder ribbon expansion}
G_n(\alpx;1,\zeta_6^{-1},1,\zeta_6)=\sum_{\pi\in \mnoncrossing{n}  }
 2^{\sa(\pi)}\biggl(\,\sum_{\delta\in
  [\ltyp{\pi},\utyp{\pi}]}r_{\delta}\biggr),
\end{equation}
where if $\pi = B_1/\dots/B_k$, then $\ltyp{\pi} = c(B_1)\concat\cdots\concat c(B_k)$ and $\utyp{\pi} = c(B_1)\nconcat\cdots\nconcat c(B_k)$.

For $0\leq k\leq n-1$, let
\begin{align}
\foulkes_{n,k}\coloneq\sum_{\substack{\alpha\vDash n\\ \ell(\alpha)=k+1}} r_{\alpha}.
\end{align}
It can be shown that the dimension of the $\mathfrak{S}_n$-module corresponding to $F_{n,k}$ is the Eulerian number  $A_{n,k}$ enumerating the number of permutations in $\mathfrak{S}_n$ with $k$ descents \cite{Foulkes}.
 {The symmetric group character corresponding to $F_{n,k}$ under the Frobenius characteristic map} is known as the \bemph{Foulkes character}, introduced by Foulkes in his study of descents in permutations. They show up in various areas such as counting permutations by descents and cycle types \cite{Gessel-Reutenauer}, enumerating alternating permutations according to cycle type \cite{Stanley-Alternating}, and the analysis of the carrying process \cite{Diaconis-Fulman, Holte,Novelli-Thibon-Amazing}.
Foulkes characters have been generalized to other reflection groups in \cite{Miller-Foulkes}.

Throughout this section, we encode most of our formulas in terms of marked and augmented interlacing partitions. However, it should be noted that all of these formulas could just as easily be encoded in terms of binary trees via the following bijection between $\pbtun$ and $\noncrossing{n}$, which is a special case of a bijection due to Edelman. Figure~\ref{fig: Edelman example} shows a binary tree $T$ and its corresponding noncrossing partition.
\begin{theorem}[Edelman \cite{Edelman}]
Given $T\in \pbtun$, let $v_1,\ldots, v_n$ be its vertices listed in preorder.
Define $\noncross{T}$ to be the finest partition of $[n]$ with the property that distinct positive integers $1\leq i<j \leq n$ are in the same block if $v_j$ is the left child of $v_i$. The map $T\mapsto \noncross{T}$ defines a bijection between $\pbtun$ and $\noncrossing{n}$.
\end{theorem}

\begin{figure}[ht]
\centering
  \includegraphics[scale=0.33]{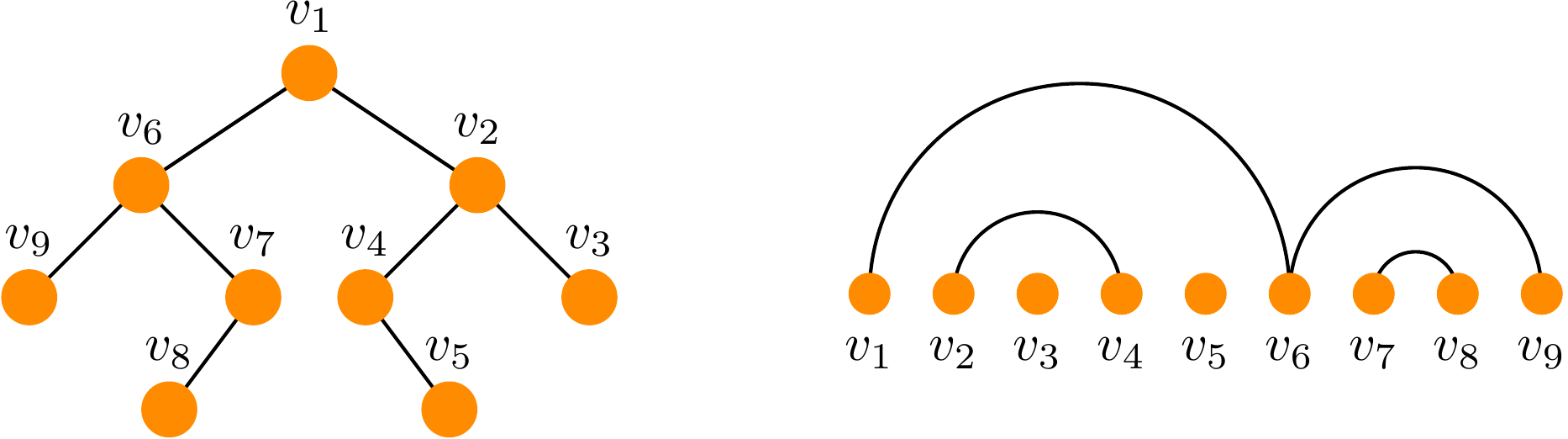}
\caption{On the left is a binary tree $T$, and on the right is its corresponding noncrossing partition $\noncross{T}$.\label{fig: Edelman example}}
\end{figure}
Given a binary tree $T\in\pbtun$, we say that it is \bemph{left-leaning} if for every node $v$ of $T$, whenever $v$ has a right child, then it also has a left child.  We denote the set of left-leaning trees on $n$ nodes by $\lpbt_n$. From the definition of $\noncross{T}$, we see that if $T\in \lpbt_n$, then $\noncross{T}$ is an interlacing partition of $[n]$. In fact, Edelman's bijection restricts to a bijection between $\lpbt_n$ and the set of interlacing partitions of $[n]$.

For $n\geq 1$, let $\Mot{n}$  denote the $n$th \bemph{Motzkin number} \cite[A001006]{OEIS} which is defined to be the number of interlacing partitions on $[n]$, or equivalently the number of left-leaning binary trees on $n$ nodes.
Let $M(x)$ denote the  generating function for the Motzkin numbers,
\begin{align}
M(x)\coloneq\sum_{n\geq 1}\Mot{n}\,x^n = x + x^2 + 2x^3 + 4x^4 + \cdots.
\end{align}

Let $\Cat{n} = |\pbt_n|$ be the $n$th \bemph{Catalan number} \cite[A000108]{OEIS}. Let $C(x)$ denote the generating function for the Catalan numbers,
\begin{align}
C(x)\coloneq\sum_{n\geq 1}\Cat{n}\,x^n = x + 2x^2 + 5x^3 + 14x^4 + \cdots.
\end{align}
Then it can be checked that $M(x)$ and $C(x)$ are related by the following identity,
\begin{align}\label{eqn: Motzkin to Catalan}
M\left(\frac{x}{1-x}\right)=C(x).
\end{align}
Equation~\eqref{eqn: Motzkin to Catalan} implies the following identity that will come in handy later, \begin{align}\label{eqn: Motzkin and Catalan exponential}
M(e^x-1)=C(1-e^{-x}).
\end{align}
Decomposing a tree into its root with two subtrees, each of which is potentially empty, yields the identity $C=x(1+C)^2$.
By using~\eqref{eqn: Motzkin to Catalan} again, we see that $M(x)$ satisfies the functional equation
\begin{align}\label{eq: M functional equation}
M=x(1+M+M^2).
\end{align}
Compare this to \cite[Equation 1]{Donaghey-Shapiro}, for example, where our indexing differs from theirs by $1$.

\begin{theorem}\label{thm: semiorder expansion}
For $n\geq 1$, we have the following expansion of $G_n(\alpx;1,\zeta_6^{-1},1,\zeta_6)$ in terms of the Frobenius characteristics of Foulkes characters,
\begin{align*}
G_n(\alpx;1,\zeta_6^{-1},1,\zeta_6)= \sum_{\alpha\vDash n}\Mot{\ell(\alpha)}\,h_{\alpha}=\sum_{j=1}^{n}\left(\sum_{k=0}^{n-j}\binom{n-j}{k}\Mot{n-k}\right)F_{n,j-1}.
\end{align*}
\end{theorem}
\begin{proof}
Just for this proof, we denote $G(\alpx;1,\zeta_6^{-1},1,\zeta_6)$ by $G$. Similarly, let $H = H(1)=\sum_{i\geq 0}h_i$. To obtain the $h$-expansion for $G_n(\alpx;1,\zeta_6^{-1},1,\zeta_6)$, we use the functional equation for $G$ from Theorem~\ref{thm: Gessel functional equation}.  When $\la=\ra=1$, $\ld=\zeta_6^{-1}$, and $\rd=\zeta_6$, we get that
\begin{align}\label{eqn: functional equation semiorder}
G=(1+G+G^2)(H-1).
\end{align}
A comparison of the functional equation in~\eqref{eqn: functional equation semiorder} with the functional equation satisfied by $M(x)$ in~\eqref{eq: M functional equation} reveals that
\begin{align}
G&=\sum_{m\geq 1}\Mot {m}\,(H-1)^{m}\\
&=\sum_{m\geq 1}\sum_{\ell(\alpha)=m}\Mot {m}\,h_{\alpha}\label{eqn: h-expansion semiorder},
\end{align}
where the above sum is over all compositions $\alpha$ of any size with length $m$.
Focusing on terms indexed by compositions of size $n$ in~\eqref{eqn: h-expansion semiorder} yields the first equality in the statement of the theorem.

To obtain the expansion in terms of the Frobenius characteristics of Foulkes characters, we utilize the fact that $h_{\alpha}=\sum_{\beta \succcurlyeq \alpha}r_{\beta}$ and~\eqref{eqn: h-expansion semiorder} to obtain
\begin{align}
G&=\sum_{m\geq 1}\sum_{\ell(\alpha)=m}\Mot {m}\sum_{\beta \succcurlyeq \alpha}r_{\beta}\\
&= \sum_{m\geq 1}\sum_{\beta\vDash m}r_{\beta}\biggl(\,\sum_{\beta\succcurlyeq\alpha}\Mot{\ell(\alpha)}\biggr)\\
&=\sum_{m\geq 1}\sum_{\beta\vDash m}r_{\beta}\biggl(\,\sum_{k=0}^{m-\ell(\beta)} \binom{m-\ell(\beta)}{k}\Mot{m-k}\biggr).
\end{align}
Note that the innermost sum is only dependent on $\ell(\beta)$.
Collecting terms corresponding to compositions of a fixed size and grouping them according to their lengths gives us the second equality in the statement of the theorem.
\end{proof}
For example, $G_3(\alpx;1,\zeta_6^{-1},1,\zeta_6)=\Mot{1}h_3+\Mot{2}(h_{12}+h_{21})+\Mot{3}h_{111}=h_3+2h_{21}+2h_{111}$, which is the Frobenius characteristic of the $\mathfrak{S}_3$-action on $\hypsemi_3$ as shown in Figure~\ref{fig:semiorder action}.
To compute the Frobenius characteristic in the general case we need the following result.
\begin{lemma}[{\cite[Lemma 7.6]{Postnikov-Stanley}}]\label{lem: Postnikov-Stanley on semiorders}
Let $\sigma\in \mathfrak{S}_n$ be a permutation with $k$ cycles. Then the number of regions in $\hypsemi_n$ fixed by $\sigma$ is equal to the  number of regions in $\hypsemi_{k}$.
\end{lemma}
We remark here that the statement in Lemma~\ref{lem: Postnikov-Stanley on semiorders} differs slightly from that in \cite{Postnikov-Stanley} as Postnikov-Stanley consider the hyperplane arrangement $\hypsemi_n$ projected onto the hyperplane $x_1+\cdots +x_n=0$ as their definition of the semiorder arrangement. This  induces a harmless shift in indices and does not affect the mathematical content.
\begin{theorem}\label{thm: Frobenius semiorder}
The symmetric function $G_n(\alpx; 1,\zeta_6^{-1},1,\zeta_6)$ is the Frobenius characteristic of the action of $\mathfrak{S}_n$ on  $\regions{\hypsemi_{n}}$.
\end{theorem}
\begin{proof}
Our proof uses the fact that the cycle indicator $Z_n$ of the $\mathfrak{S}_n$-action  on $\regions{\hypsemi_n}$,  is also the Frobenius characteristic of the character of this action. See~\cite[Section 7.24]{Stanley-EC2} for more background on the cycle indicator.
In view of Lemma~\ref{lem: Postnikov-Stanley on semiorders}, we find that the cycle indicator  is given by
\begin{align}
Z_n&=\sum_{\lambda\vdash n} r(\hypsemi_{\ell(\lambda)})\, \frac{p_{\lambda}}{z_{\lambda}}\\
&=\sum_{k=1}^{n} r(\hypsemi_{k}) \sum_{\substack{\lambda\vdash n\\\ell(\lambda)=k}} \frac{p_{\lambda}}{z_{\lambda}}.
\end{align}
Let $Z\coloneq \sum_{n\geq 1}Z_n$. Then we have that
\begin{align}
Z &=\sum_{n\geq 1}\sum_{k=1}^{n} r(\hypsemi_{k}) \sum_{\substack{\lambda\vdash n\\\ell(\lambda)=k}} \frac{p_{\lambda}}{z_{\lambda}}\\
&=
\sum_{k\geq 1} \frac{r(\hypsemi_{k})}{k!} \,\biggl(\sum_{j\geq 1}\frac{p_j}{j}\biggr)^k.\label{eqn: cycle index in terms of power sums}
\end{align}

Now let $H = \sum_{i\geq 0}h_i$ again.
Then we have  $\sum_{j\geq 1}{p_j}/{j}=\log {H}$. Using this in~\eqref{eqn: cycle index in terms of power sums} gives
\begin{align}\label{eqn: cycle index with H}
Z=\sum_{k\geq 1} r(\hypsemi_{k})\,\frac{(\log{H})^k}{k!}.
\end{align}
We have the expansion
\begin{align}
\frac{(\log{H})^k}{k!}=\sum_{m\geq 0} (-1)^{m-k}\,\Stir(m,k)\,\frac{(H-1)^m}{m!}
\end{align}
where $(-1)^{m-k}\,\Stir(m,k)$ is the \bemph{signed Stirling number of the first kind} \cite[A008275]{OEIS} enumerating  permutations in $\mathfrak{S}_m$ having exactly $k$ cycles in their cycle factorization.
Using this equality, we can rephrase \eqref{eqn: cycle index with H} as
\begin{align}
Z&=\sum_{k\geq 1} r(\hypsemi_{k})\sum_{m\geq 0} (-1)^{m-k}\,\Stir(m,k)\,\frac{(H-1)^m}{m!}\\
&= \sum_{m\geq 1} \frac{(H-1)^m}{m!}\sum_{k= 1}^{m} (-1)^{m-k}\, r(\hypsemi_{k})\, \Stir(m,k),\label{eqn: interchange order of m and k}
\end{align}
where, in changing the order of summation, we have used the fact that $\Stir(0,k)=0$ for $k\geq 1$.  By \cite[Theorem 2.3]{Stanley-Pnas}, the following equality holds
\begin{align}\label{eqn: generating function for semiorders}
\sum_{m\geq 1} r(\hypsemi_m)\,\frac{x^m}{m!}=C(1-e^{-x}).
\end{align}
Recall~\eqref{eqn: Motzkin and Catalan exponential} states that $M(e^x-1)=C(1-e^{-x})$, and hence the left hand side of~\eqref{eqn: generating function for semiorders} also equals $M(e^x-1)$.
This implies that
\begin{align}\label{eq: Motzkin Stirling}
m!\,\Mot{m}=\sum_{k= 1}^{m} (-1)^{m-k}\, r(\hypsemi_{k})\, \Stir(m,k).
\end{align}
Substituting~\eqref{eq: Motzkin Stirling} into~\eqref{eqn: interchange order of m and k}, we conclude that
\begin{align}\label{eqn: Z equals S}
Z=\sum_{m\geq 1} \Mot{m}\,(H-1)^m=G(\alpx;1,\zeta_6^{-1},1,\zeta_6).
\end{align}
The second equality in~\eqref{eqn: Z equals S} follows from~\eqref{eqn: h-expansion semiorder}. Then \eqref{eqn: Z equals S} implies that the cycle indicator $Z_n$ is equal to $G_n(\alpx;1,\zeta_6^{-1},1,\zeta_6)$, which completes the proof.
\end{proof}

We conclude this subsection with a generalization of Theorem \ref{thm: Frobenius semiorder}.
Given a positive integer $p$, define the \bemph{$p$-semiorder arrangement} $\hypsemi_{n,p}$ to be the hyperplane arrangement in $\mathbb{R}^n$ defined by the hyperplanes $x_i-x_j=\pm 1, \pm 2, \cdots,\pm p$ for $1\leq i<j\leq n$.
As before, the symmetric group $\mathfrak{S}_n$ acts on  $\regions{\hypsemi_{n,p}}$ and one can ask  about the corresponding Frobenius characteristic.
We provide a brief description of its computation next, omitting details.

Let $C_{p}(x)\coloneqq \sum_{n\geq 1}\Cat{n,p}\,x^n$ denote the generating function for the \bemph{Fuss-Catalan} numbers $\Cat{n,p}\coloneqq\frac{1}{pn+1}\binom{(p+1)n}{n}$.
Let $\pbtun^p$ be the set of rooted plane $(p+1)$-ary trees $T$ such that each node of $T$ has at most one $i$th child for each $1\leq i\leq p+1$. Then $\Cat{n,p}$ is the cardinality of the set $\pbtun^p$.
Let $\Mot{n,p}$ denote the cardinality of the set $\lpbt_n^p$ consisting of rooted plane $(p+1)$-ary trees on $n$ nodes such that every internal node which has a $(p+1)$th child also has at least one other child.

Let $M_p(x)\coloneq \sum \Mot{n,p}\,x^n$.
As was the case earlier, we have that
\begin{align}
M_p\left(\frac{x}{1-x}\right)=C_p(x),
\end{align}
which implies $M_p(e^x-1)=C_p(1-e^{-x})$.

Crucially for us, Lemma \ref{lem: Postnikov-Stanley on semiorders} continues to hold for $\hypsemi_{n,p}$ as well. Thus, we can repeat the cycle indicator computation of Theorem~\ref{thm: Frobenius semiorder}.
We only need the relation analogous to~\eqref{eqn: generating function for semiorders}.
Indeed, by \cite[Theorem 2.3]{Stanley-Pnas}, we have that
\begin{align}
\sum_{m\geq 1} r(\hypsemi_{m,p})\,\frac{x^m}{m!}=C_p(1-e^{-x}).
\end{align}
In arriving at the above equality we have used the fact that the number of regions in the \bemph{p-Catalan arrangement} defined by the hyperplanes $x_i-x_j=0,\pm 1,\ldots,\pm p$ for $1\leq i< j\leq n$ is given by $\Cat{n,p}$ (see \cite[Section 2]{Stanley-Pnas} or \cite[Section 5]{Athanasiadis-Advances}). Thus, we obtain the following theorem whose $p=1$ case is Theorem~\ref{thm: Frobenius semiorder}.
\begin{theorem}\label{thm:p-semiorder Frobenius}
 For $p\geq 1$, the Frobenius characteristic of the action of $\mathfrak{S}_n$ on  $\regions{\hypsemi_{n,p}}$ is given by $M_p(H-1)$, where $H=\sum_{i\geq 0}h_i$. Additionally, the Frobenius characteristic expands positively in terms of the $F_{n,k}$.
\end{theorem}

\subsection{Linial arrangements and local binary search trees}\label{subsec: Linial}
We turn our attention to studying the Linial arrangement $\hyplin_n$ and defining an $\mathfrak{S}_n$-action on its regions.
Observe that, unlike in the case of the semiorder arrangement, the symmetric group $\mathfrak{S}_n$ does not stabilize the set of hyperplanes defining $\hyplin_n$.
Hence it is not immediate how to construct an $\mathfrak{S}_n$-action on $\regions{\hyplin_n}$.
Another well-studied arrangement with the property that the set consisting of its defining hyperplanes is not stable under the $\mathfrak{S}_n$-action is the Shi arrangement $\hypshi_n$.
In spite of this limitation, one can define an $\mathfrak{S}_n$-action on  $\regions{\hypshi_n}$ by using one of the many ways to index its regions by parking functions of length $n$, and then using the natural $\mathfrak{S}_n$-action on them.
Drawing inspiration from this, we use certain labeled trees that we call Bernardi trees instead of parking functions to index regions of $\hyplin_n$, and then construct a natural $\mathfrak{S}_n$-action on Bernardi trees to derive one on  $\regions{\hyplin_n}$.
First, we need to understand how the symmetric function $G_n$ relates to $\regions{\hyplin_n}$.

The problem of enumerating  $r(\hyplin_n)$ was first considered by Postnikov \cite{Postnikov-JCTA}, inspired by a question of Linial and Ravid. He showed that $r(\hyplin_n)$ equals the number of intransitive trees and gave a bijection between intransitive trees and local binary search trees. It is the latter that carries a description in terms of ascents and descents in labeled  trees.

Consider the case where $\ld=\rd=0$.
In the setting of the introduction, this corresponds to considering labeled trees  that only have  left ascents and right ascents. 
We refer to such  trees  as \bemph{local binary search trees} (henceforth \bemph{LBS trees}).  {See Figure~\ref{fig:LBS tree} for an example of a local binary search tree on $9$ nodes.}
We construct another subset of $\pbtln$ that is equinumerous with standard LBS trees on $n$ nodes and use it to define an  $\mathfrak{S}_n$-action on  $\regions{\mathcal{L}_n}$.
It is worth emphasizing that the same subset of trees has been considered by Bernardi \cite{Bernardi} to solve the long-standing problem of finding a bijection between Linial regions and standard LBS trees.
\begin{figure}[h]
\centering
\includegraphics[scale=0.33]{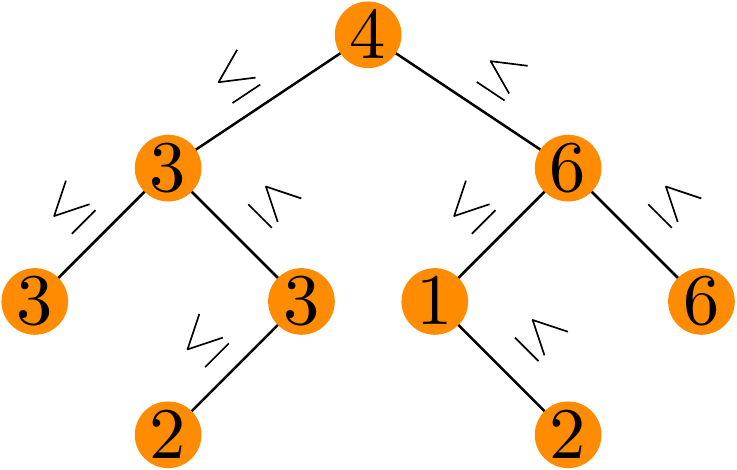}
\caption{An example of a local binary search (LBS) tree on $9$ nodes.\label{fig:LBS tree}}
\end{figure}

By writing $G_n$ as a sum of the $G_{n,\canopy}$ where $\canopy$ runs over all possible canopies of length $n-1$ and using Theorem~\ref{thm: ribbon expansion for fixed canopy}, we have
\begin{align}
G_n(\alpx; \la,0,\ra,0)&=\sum_{\pi^\ast = B_1/\dots/B_k}\la^{\sau(\pi^\ast)}\ra^{\,\sad(\pi^\ast)}(\la\ra)^{\ja(\pi^\ast)}r_{c(B_1)}\dots r_{c(B_k)}\label{eqn: r-expansion lbs}\\
&= \sum_{\pi^\ast = B_1/\dots/B_k}\la^{\sau(\pi^\ast)}\ra^{\,\sad(\pi^\ast)}(\la\ra)^{\ja(\pi^\ast)}h_{(|B_1|,\dots,|B_k|)}\label{eqn: h-expansion lbs nc}
\end{align}
where the sum runs over all augmented interlacing partitions $\pi^\ast$ on $n$ nodes such that all nodes are unmarked. The second equality comes from the fact that, for a block $B_i$ in $\pi^\ast$, since all steppers and jumpers are {unmarked}, then $r_{c(B_i)} = r_{(m_i)}=h_{m_i}$ where $m_i=|B_i|$.

Next,  we use a bijection between the set of augmented interlacing partitions on $n$ nodes such that all nodes are unmarked and the set $\pbtun$ in order to get an expansion in terms of binary trees.
Given an augmented interlacing partition $\pi^\ast$ on $n$ nodes such that all nodes are unmarked, delete all short arcs in $\pi^\ast$ which are labeled with a $U$, and then remove all $U$ and $D$ labels. This provides a bijection between these objects and $\noncrossing{n}$. Then compose this bijection with the inverse of Edelman's bijection from $\noncrossing{n}$ to $\pbtun$. {Finally, apply the bijection from $\pbtun$ to itself which flips a tree across the vertical line passing through its root.}
 {See Figure~\ref{fig:Composition Example} for an example of an augmented interlacing partition on $8$ nodes with all nodes unmarked, together with its image under the composition of these bijections.
\begin{figure}[h]
\centering
\includegraphics[scale=0.3]{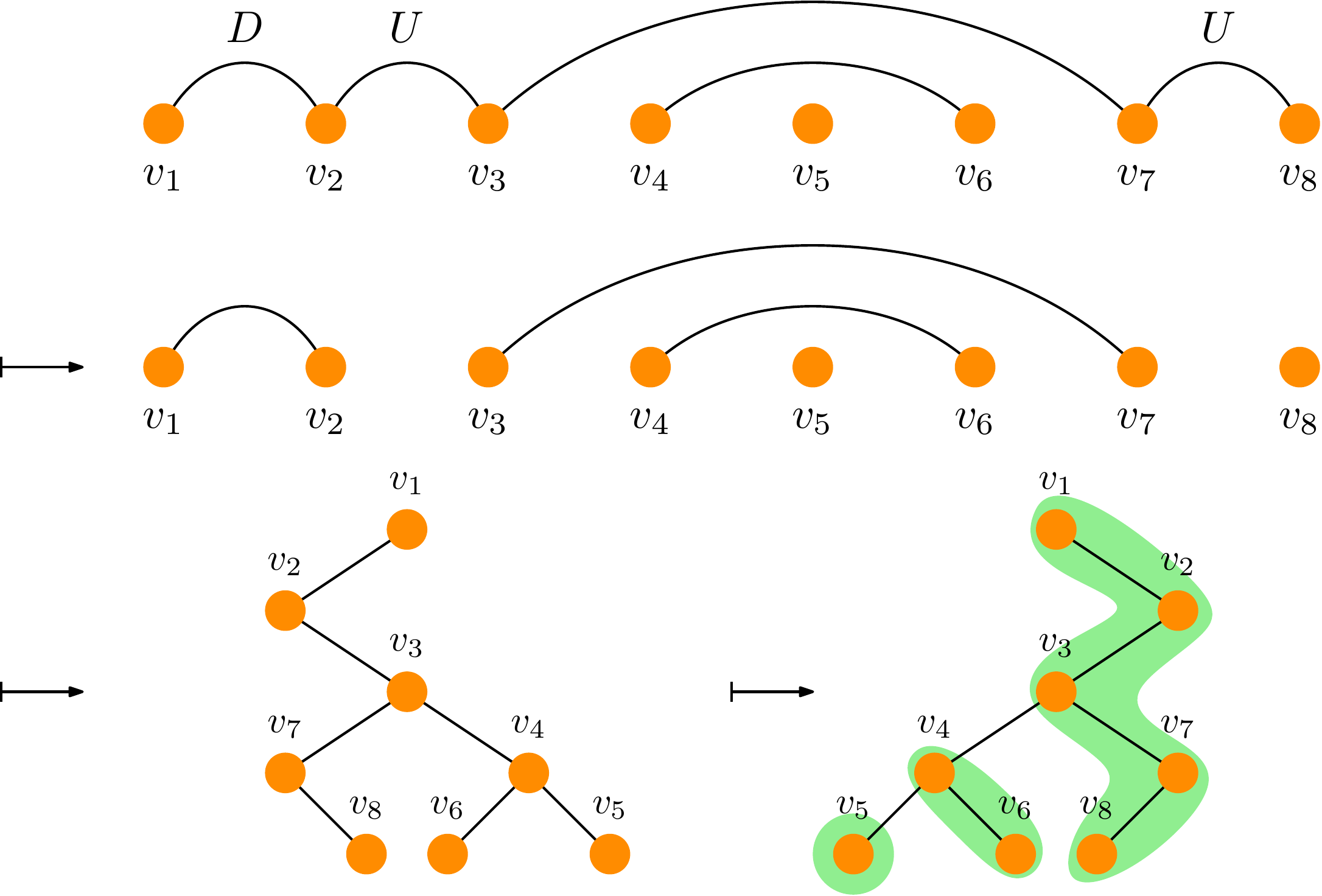}
\caption{An example of an augmented interlacing partition on $8$ nodes with all nodes unmarked, and its corresponding elements of $\noncrossing{8}$ and $\pbt_8$.\label{fig:Composition Example}}
\end{figure}
}

{We claim that $\mathrm{sort}(\typ{T})=\mathrm{sort}(|B_1|,\dots,|B_k|)$, where $\mathrm{sort}(\alpha)$ is the partition underlying a composition $\alpha$ and $\typ{T}$ denotes the composition type of $T$ defined in Subsection~\ref{subsec: unlabeled trees}. We make note of two pertinent aspects of this composition of three bijections.
Let $\pi^\ast = B_1/\dots/B_k$ be an augmented interlacing partition on $n$ nodes such that all nodes are unmarked, and let $T\in\pbt_n$ be  the corresponding binary tree. 
\begin{enumerate}
\item If $v_i$ and $v_{i+1}$ are connected by a short arc labeled $U$ (respectively $D$) in $\pi^{\ast}$, then in $T$ the node $v_i$ has only a left child (respectively right child) which is $v_{i+1}$.
\item If nodes $v_i$ and $v_j$ where $i<j$ are connected by a long arc in $\pi^{\ast}$, then $v_i$ has two children in $T$. In particular, the right child is $v_j$ and the left child is $v_{i+1}$.
\end{enumerate}
It follows that any two nodes in $T$ connected by an edge which are visited in succession in the preorder traversal correspond to nodes in $\pi^{\ast}$ connected by an arc. Hence, we have $\mathrm{sort}(\typ{T})=\mathrm{sort}(|B_1|,\dots,|B_k|)$.
Additionally, we see that $\sad(\pi^\ast)+\ja(\pi^\ast) = \rightt$, $\sau(\pi^\ast)+\ja(\pi^\ast) = \leftt$, where $\leftt$ and $\rightt$ denote the number of {left edges and right edges} in $T$, respectively.
}

Hence, from \eqref{eqn: h-expansion lbs nc} we obtain the following expansion,
\begin{align}\label{eqn: h-expansion lbs}
G_n(\alpx; \la,0,\ra,0)&=\sum_{T\in \pbt_n}\la^{\leftt}\ra^{\,\rightt} h_{\typ{T}}.
\end{align}
Since $\leftt+\rightt=n-1$ for any $T\in \pbt_n$, we set $\la=1$ and $\ra$ equal to an indeterminate $q$.
In order to give a representation-theoretic interpretation of $G_n(\alpx; \la,0,\ra,0)$, it suffices to find a $\mathfrak{S}_n$-module whose graded Frobenius characteristic is $G_n(\alpx;1,0,q,0)$, which we give next.

A \bemph{Bernardi tree} is a standard labeled binary tree satisfying the condition that every internal node has a label that is greater than the label of its right child provided it exists, otherwise it is greater than the label of its left child.
Let $\berntree_n$ denote the set of Bernardi trees on $n$ nodes.
It can be checked that the image of the homomorphism $\ex$ applied to the right-hand side of~\eqref{eqn: h-expansion lbs} is a generating function enumerating Bernardi trees by the number of left and right edges.
Since $G_n(\alpx;\la,0,\ra,0)$ is the generating function corresponding to LBS trees on $n$ nodes, applying $\ex$ to the left-hand side of \eqref{eqn: h-expansion lbs} establishes that the number of standard LBS trees on $n$ nodes is equal to the cardinality of $\berntree_n$.
Figure \ref{fig:B-trees} shows all trees in $\berntree_3$.
\begin{figure}[ht]
\centering
  \includegraphics[scale=0.8]{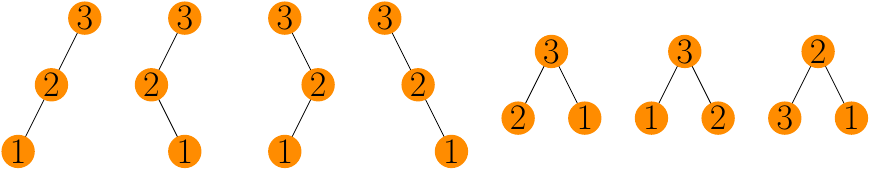}
  \captionof{figure}{The Bernardi trees on 3 nodes.}
  \vspace*{-2mm}
  \label{fig:B-trees}
\end{figure}
\begin{remark}\label{rem:why_bernardi}
For the reason behind the name Bernardi trees, we refer the reader to \cite[Example 1.1]{Bernardi} where Bernardi gives a bijection between $\berntree_n$ and $\regions{\hyplin_n}$. In the notation of \cite{Bernardi}, the bijection associates to $T\in \berntree_n$ the region of $\hyplin_n$ defined by the inequalities $ x_i-x_j<1$ where $1\leq i<j\leq n$ and either $\mathrm{drift}(i) \leq \mathrm{drift}(j)$, or $\mathrm{drift}(i)=\mathrm{drift}(j)+1$ and $i \prec_p j$. Here, $\mathrm{drift}(v)$ is the number of ancestors of $v$ (including $v$) that are right children. Figure~\ref{fig:Linial regions and trees} shows the regions of $\hyplin_3$ indexed by the corresponding Bernardi trees according to this bijection.

An equivalent characterization of Bernardi trees is as follows. 
Let $v_{1}\prec_p \cdots \prec_p v_n$ be the nodes of a standard labeled binary tree $T$ listed in preorder. If for all cover relations $v_{i}\prec_p v_{i+1}$  where $v_i$ is the parent of $v_{i+1}$, we have that $v_i^{\ell}>v_{i+1}^{\ell}$, then $T$ is a Bernardi tree.
\end{remark}
\begin{figure}[ht]
\centering
  \includegraphics[scale=0.3]{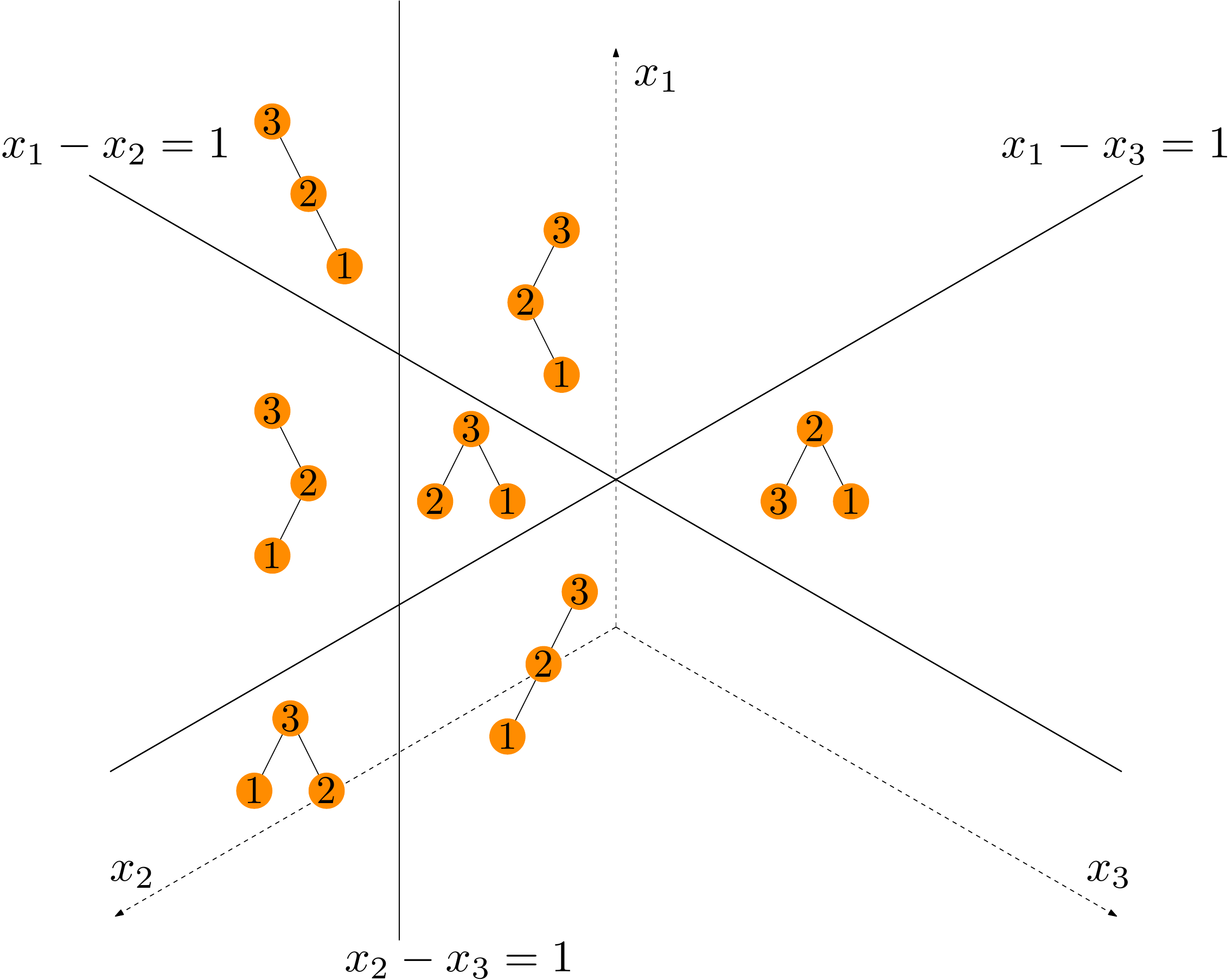}
  \captionof{figure}{Regions of $\hyplin_3$ indexed by trees in $\berntree_3$.}
  \label{fig:Linial regions and trees}

\end{figure}

For $T\in \berntree_n$, let $v_{1}\prec_p \cdots \prec_p v_n$ be the nodes of $T$ in preorder.
Among these nodes, let $v_{i_1} \prec_p \cdots \prec_p v_{i_k}$ be all the terminal nodes and set $i_0\coloneqq 1$.
Observe that the preorder reading word
$\pre{T}=v_1^{\ell}\cdots v_n^{\ell}$
can be factorized as $W_1\cdots W_k$, where for $j\geq 1$ we have
 \begin{align}\label{eqn:preorder factorization}
 W_{j}=v_{i_{j-1}+1}^{\ell}v_{i_{j-1}+2}^{\ell}\cdots  v_{i_j-1}^{\ell}v_{i_j}^{\ell},
 \end{align}
Given the definition of $\berntree_n$, we  know that each $W_j$ is strictly decreasing when read from left to right.
For the tree $T\in \berntree_9$ in Figure~\ref{fig:Linial action} on the left, the  shaded regions give us the words $W_1=731$, $W_2=86$, $W_3=952$ and $W_4=4$ and their concatenation gives us $\pre{T}= 731 \hspace{2mm}86\hspace{2mm} 952 \hspace{2mm} 4$.

The factorization of $\pre{T}$ as $W = W_1\cdots W_k$ as described earlier allows us to define an obvious $\mathfrak{S}_n$-action on $\berntree_n$ as follows.
 Given $\sigma\in \mathfrak{S}_n$, define $\widetilde{\sigma}(W_i)$ to be the word obtained by replacing every letter in $W_i$ by its image under $\sigma$, and then sorting the resulting word so that it is strictly decreasing when read from left to right.
Now define $\widetilde{\sigma}(W)$ to be $\widetilde{\sigma}(W_1)\cdots \widetilde{\sigma}(W_k)$, and let $\sigma(T)$ be the unique labeled tree such that $\shape{\sigma(T)}=\shape{T}$ and $\pre{\sigma(T)}=\widetilde{\sigma}(W)$.
{By the alternative characterization for Bernardi trees in Remark~\ref{rem:why_bernardi}, these two conditions ensure that $\sigma(T)\in \berntree_n$. }

{This is a well-defined $\mathfrak{S}_n$-action on $\berntree_n$. One way to see this is to observe that Bernardi trees $T$ with a fixed $\shape{T}$ are in bijection with tabloids \cite{Sagan} of partition shape $\mathrm{sort}(\typ{T})$. Indeed, given the factorization $W_1\cdots W_k$ of $\pre{T}$ defined above, form $k$ rows of boxes with $i_j-i_{j-1}$ many boxes in row $j\leq k$, where the boxes in row $j$ are labeled with the letters of $W_j$ in order from left to right. Then sort the rows to get a tabloid on the partition shape $\mathrm{sort}(\typ{T})$. This map defines the desired bijection and the action of $\mathfrak{S}_n$ on Bernardi trees with fixed $\shape{T}$ corresponds to the usual action of $\mathfrak{S}_n$ on tabloids under the bijection.}

In Figure~\ref{fig:Linial action}, for the tree $T$ on the left, we have $\sigma(T)$ on the right where $\sigma=(38)$ in cycle notation.
Note that for the instance under discussion, we have $\widetilde{\sigma}(W)=871 \hspace{2mm}63\hspace{2mm} 952 \hspace{2mm} 4$. Thus, $\sigma(T)$ is the unique tree whose preorder reading word is $871639524$ and whose underlying shape is that of $T$.  Had we chosen $\sigma=(37)$, then $\sigma(T)$ would be $T$ itself.
\begin{figure}[ht]
\centering
  \includegraphics[scale=0.4]{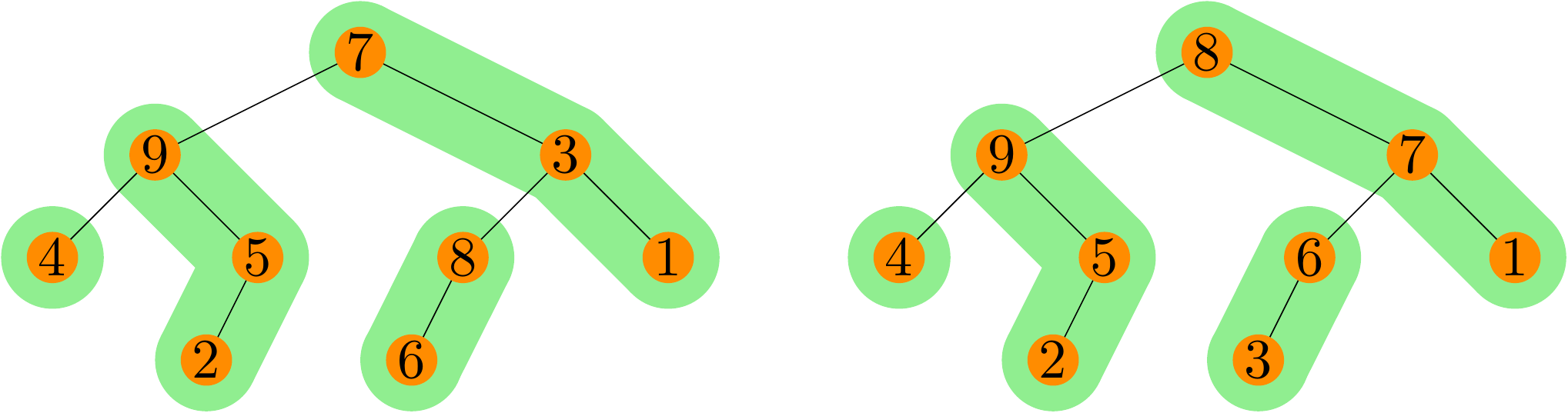}
  \captionof{figure}{On the left, a tree $T\in \berntree_9$. On the right, $\sigma(T)$ where $\sigma= (38)\in \mathfrak{S}_9$.}
  \label{fig:Linial action}
\end{figure}

Let $\bC \berntree_n$ denote the $\mathfrak{S}_n$-module whose underlying vector space is generated by formal linear combinations of trees in $\berntree_n$. Denote by $\berntree_{n,k}$ the set of Bernardi trees with exactly $k$ right edges. Let $\bC\berntree_{n,k}\subseteq\bC\berntree_n$ be the $\mathfrak{S}_n$-submodule spanned by Bernardi trees in $\berntree_{n,k}$.
We have the following equality of $\mathfrak{S}_n$-modules
\begin{align}
\bC\berntree_n=\bigoplus_{k=0}^{n-1}\bC\berntree_{n,k}.
\end{align}
Thus, we can think of $\bC\berntree_n$ as being graded by the number of right edges.
\begin{theorem}\label{thm:frob_lbs}
The graded Frobenius characteristic of the $\mathfrak{S}_n$-module $\bC\berntree_n$ is given by
\begin{align}\label{eq: LBS frob2}
G_n(\alpx;1,0,q,0) = \sum_{T\in \pbt_n}q^{\rightt}h_{\typ{T}}.
\end{align}
Using Bernardi's bijection between  $\berntree_n$ and $\regions{\hyplin_n}$, our $\mathfrak{S}_n$-action lifts to an action on $\regions{\hyplin_n}$ whose graded Frobenius characteristic is also $G_n(\alpx;1,0,q,0)$.
\end{theorem}
\begin{proof}
Observe that the identity~\eqref{eq: LBS frob2} follows from~\eqref{eqn: h-expansion lbs}. Therefore, it suffices to prove that the Frobenius characteristic of $\bC\berntree_{n,k}$ is the $q^k$ coefficient of the right-hand side of~\eqref{eq: LBS frob2}.
Given $\alpha \vDash n$ of length $m$, let $\mathfrak{S}_\alpha\coloneq\mathfrak{S}_{\alpha_1}\times \cdots \times \mathfrak{S}_{\alpha_m}$ be the corresponding Young subgroup of $\mathfrak{S}_n$. It is well known (see, e.g.\ \cite{Sagan}) that $h_\alpha$ is the Frobenius characteristic of $\mathbf{1}\uparrow_{\mathfrak{S}_\alpha}^{\mathfrak{S}_n}$, the trivial $\mathfrak{S}_\alpha$-module induced up to $\mathfrak{S}_n$.

Given an unlabeled tree $T\in \pbtun$, the Bernardi trees $T'\in \berntree_{n,k}$ such that $\shape{T'}=T$ span a $\mathfrak{S}_n$-submodule of $\bC\berntree_{n,k}$ isomorphic to $\mathbf{1}\uparrow_{\mathfrak{S}_{\typ{T}}}^{\mathfrak{S}_n}$.
Hence, we have an isomorphism of $\mathfrak{S}_n$-modules,
\begin{align}
\bC\berntree_{n,k} \cong \bigoplus_{\substack{T\in \pbtun\\ k\text{ right edges}}} \mathbf{1}\uparrow_{\mathfrak{S}_{\typ{T}}}^{\mathfrak{S}_n}.
\end{align}
The proof is then completed by using the fact that the Frobenius characteristic of a direct sum of submodules is the sum of the Frobenius characteristics.
\end{proof}
Observe that by applying the homomorphism $\ex$ to $G_n(\alpx;1,0,1,0)$, we can compute the cardinality of $\berntree_n$. This gives another formula for $r(\mathcal{L}_n)$, which is not in any way more concise than the known formula and hence is omitted.

In fact, we can recover Postnikov's formula \cite{Postnikov-JCTA} for $r(\mathcal{L}_n)$ by computing the character of our action on Bernardi trees.
This requires us to write $ G_n(\alpx;1,0,1,0)$ in terms of power sum symmetric functions.  {See~\cite{Stanley-EC2} for definitions and background pertaining to power sum symmetric functions.}
We proceed by a generating function argument involving Lagrange inversion. It would be interesting to establish the same by a combinatorial argument.

Set $q=1$ for the remainder of this subsection.
By Theorem~\ref{thm: Gessel functional equation}, we know that $G(\alpx;1,0,1,0)$ satisfies the functional equation
\begin{align}\label{eqn:func_eqn_linial}
(1+G(\alpx;1,0,1,0))^2=H(2+G(\alpx;1,0,1,0)).
\end{align}
In order to expand $G(\alpx;1,0,1,0)$ in terms of power sum symmetric functions, we need the following general result.
\begin{theorem} \label{thm:li}
Let $Q(t)$ be a polynomial, let $y$ be an indeterminate, and let $F$ be the solution of
\begin{equation*}
F = H(Q(F))^{y}.
\end{equation*}
{
For every positive integer $k$, we have
\begin{align}
F^k=\sum_{\lambda} \frac{p_{\lambda}}{z_{\lambda}}  y^{\ell(\lambda)}\sum_{m=1}^\infty   km^{\ell(\lambda)-1} [t^{m-k}] Q(t)^{|\lambda|},\label{eq:FkExpansion}
\end{align}
}
\end{theorem}
\begin{proof}
 {For $k > 0$,} by Lagrange inversion (see, e.g., \cite[equation (2.4.4)]{li}) we have
\begin{align}\label{eqn:lagrange}
F^k = \sum_{m=1}^\infty \frac{k}{m}[t^{m-k}]H(Q(t))^{my}.
\end{align}
We can write $H(Q(t))^{my}$ as
\begin{align}
H(Q(t))^{my} &= \exp\biggl(my\sum_{i=1}^\infty \frac{p_i}{i} Q(t)^i\biggr)\\
  &=\sum_{\lambda} \frac{p_{\lambda}}{z_{\lambda}} (my)^{\ell(\lambda)}Q(t)^{|\lambda|},\label{eq:SumPowerQ}
\end{align}
 {where the sum in \eqref{eq:SumPowerQ} is over all partitions $\lambda$ with $|\lambda|\geq 0$.
Substituting \eqref{eq:SumPowerQ} into the right-hand side of} \eqref{eqn:lagrange} gives
\begin{align}
F^k &= \sum_{m=1}^\infty \frac km [t^{m-k}] \sum_{\lambda} \frac{p_{\lambda}}{z_{\lambda}}  (my)^{\ell(\lambda)}Q(t)^{|\lambda|}\\
  &=\sum_{\lambda} \frac{p_{\lambda}}{z_{\lambda}}  y^{\ell(\lambda)}\sum_{m=1}^\infty   km^{\ell(\lambda)-1} [t^{m-k}] Q(t)^{|\lambda|},
\end{align}
and the theorem follows.
\end{proof}

\begin{corollary}
{Under the same assumptions as Theorem~\ref{thm:li}, if we expand $F^k$ as $\sum_{\lambda} c_{\lambda}p_{\lambda}/z_{\lambda}$ then $c_{\lambda}$ depends only on $\ell(\lambda)$, $|\lambda|$, $y$, $k$ and $Q$.}
\end{corollary}
Using Theorem~\ref{thm:li}, we have the following result.
\begin{theorem}\label{thm:character Linial}
 {We have \[G_{n}(\alpx;1,0,1,0)=\sum_{\lambda\vdash n}c_{\lambda}p_{\lambda}/z_{\lambda},\] where}
\begin{align}\label{eqn:explicit_linial}
c_{\lambda}=\frac{1}{2^{\ell(\lambda)}}\sum_{m=1}^{n+1}{m^{\ell(\lambda)-1}}\binom{n}{m-1}.
\end{align}
\end{theorem}
\begin{proof}
 {Let $F \coloneq 1+G(\alpx;1,0,1,0)$.} 
 From~\eqref{eqn:func_eqn_linial} we know that
\begin{align}
F=(H(1+F))^{1/2}.
\end{align}
Using $y=1/2$, $k=1$, and $Q(t)=1+t$ in Theorem~\ref{thm:li},  {then \eqref{eq:FkExpansion} gives}
\begin{align}
F=\sum_{\lambda} \frac{p_{\lambda}}{2^{\ell(\lambda)}z_{\lambda}}  \sum_{m=1}^\infty   m^{\ell(\lambda)-1} [t^{m-1}] (1+t)^{|\lambda|},
\end{align}
 For each $\lambda\vdash n$ with $n\geq 1$, the coefficient $c_\lambda$ of $p_\lambda/z_\lambda$ in $G_n(\alpx;1,0,1,0)$ is equal to the coefficient of $p_\lambda/z_\lambda$ in $F$. Therefore, we have
\begin{align}
c_{\lambda}=\frac{1}{2^{\ell(\lambda)}}\sum_{m=1}^{n+1}{m^{\ell(\lambda)-1}}\binom{n}{m-1},
\end{align}
{which completes the proof.}
\end{proof}
As a corollary of Theorem~\ref{thm:character Linial}, we obtain the following generalization of Postnikov's formula for $r(\mathcal{L}_n)$  {\cite{Postnikov-JCTA}}.
The reader is invited to compare it with the statement of Lemma~\ref{lem: Postnikov-Stanley on semiorders}.
\begin{corollary}\label{cor:gen_Postnikov}
Given a permutation $\sigma\in \mathfrak{S}_n$ with $k$ cycles, the number of Bernardi trees fixed by $\sigma$ equals
\[
\frac{1}{2^{k}}\sum_{m=1}^{n+1}{m^{k-1}}\binom{n}{m-1}.
\]
\end{corollary}
\noindent Observe that Postnikov's formula is obtained by setting $k=n$ in Corollary~\ref{cor:gen_Postnikov}.

\begin{remark}\label{rem:Linial character Foulkes-positive}
Given the similarity of Lemma~\ref{lem: Postnikov-Stanley on semiorders} and Corollary~\ref{cor:gen_Postnikov}, the reader may wonder if $ G_n(\alpx;1,0,1,0)$  expands as a positive integer linear combination of the Frobenius characteristics of Foulkes characters. This is not the case. Indeed, $G_3(\alpx;1,0,1,0)$ serves as a counterexample.
\end{remark}
We conclude this subsection with a curious relation between $G_{n}(\alpx;1,0,1,0)$ and certain symmetric functions that arise as special cases of Jack symmetric functions when one sets the Jack parameter $\alpha=2$.
Following Macdonald \cite[Page 407, Equation 2.20]{Macdonald}, consider the symmetric function $g_m$ defined by
\begin{align}
g_m=\sum_{\lambda\vdash m}\frac{p_{\lambda}}{2^{\ell(\lambda)}z_{\lambda}}.
\end{align}
The key property of $g_m$ for our purposes is expressed in the following relation
\begin{align}
H(y)^{\frac{1}{2}}=\prod_{i\geq 1}(1-x_iy)^{-\frac{1}{2}}=\sum_{m\geq 0}g_my^m.
\end{align}
Instead of solving \eqref{eqn:func_eqn_linial} using Lagrange inversion, one may alternative proceed by `taking square roots' on both sides and then rewrite the resulting functional equation in terms of the $g_m$.
Thus, the functional equation \eqref{eqn:func_eqn_linial} translates to
\begin{align}
  \label{eqn:zonal}
  1+G=\sum_{m\geq 0}g_m(2+G)^m.
\end{align}
One can solve this functional equation in terms of Dyck paths, similar to the case of the $h$-expansion of the parking function representation.
We keep our exposition brief.

For $\lambda=(\lambda_1,\dots,\lambda_k)$, set $g_{\lambda}\coloneqq g_{\lambda_1}\cdots g_{\lambda_k}$.
For $n\geq 1$, let $\mathcal{D}_n$ be the set of Dyck paths, which are lattice paths  {that start at $(0,0)$, end at $(n,n)$, take North and East steps, and stay weakly above the diagonal $y=x$.}
Given $D\in \mathcal{D}_n$, let $\lambda(D)$ be the partition of $n$ obtained by sorting the lengths of the vertical runs of $D$ in decreasing order.
Recall that a vertical run in a Dyck path is any maximal contiguous sequence of North steps.
Let $\mathrm{peak}(D)$ denote the number of peaks in $D$.
Solving \eqref{eqn:zonal} for $G_n$ yields the expansion
\begin{align}\label{eqn:linial_zonal}
G_n=\sum_{D\in \mathcal{D}_n}2^{n+1-\mathrm{peak}(D)}g_{\lambda(D)}.
\end{align}
Stanley (see \cite[Equation 10]{Stanley} and \cite[Proposition 2.4]{Stanley}) describes the expansion of $g_{\mu}$ in terms of Jack symmetric functions $J_{\lambda}$ at $\alpha=2$ (also called \emph{zonal symmetric functions}).
Note that Stanley uses scaled versions of our $g_{\mu}$, which he denotes by $\mathscr{J}_{\mu}$.
Thus, we may expand $G_n$ in terms of zonal symmetric functions. Since we already know by Theorem~\ref{thm:frob_lbs} that $G_n$ is $h$-positive, we arrive indirectly upon a curious combination of zonal symmetric functions that is $h$-positive.

We conclude this section by mentioning that Frobenius characteristics of $\mathfrak{S}_n$-actions on the Shi arrangement and the braid arrangement can also be interpreted in terms of specializations of $G_n(\alpx;\la,\ld,\ra,\rd)$.
Since both these cases are quite well-studied, we postpone our discussion on them to Section~\ref{sec: back to the future}.

\section{Local binary search trees and $\gamma$-nonnegativity}\label{sec:gamma}
{Theorem~\ref{cor: Ribbon Expansion}} implies that $G_n$ can be written in terms of ribbon Schur functions with coefficients in $\mathbb{N}[\la+\ra,\la\ra,\ld+\rd,\ld\rd]$.
Suppose that $G_n(\alpx;\la,\ld,\ra,\rd)=\sum_{\alpha\vDash n} c_{\alpha}r_{\alpha}$ where the $c_{\alpha}$ belong to $\mathbb{N}[\la+\ra,\la\ra,\ld+\rd,\ld\rd]$.
Applying the homomorphism $\ex$ yields
\begin{align}\label{eqn:double gamma-nonnegativity}
B_n(\la,\ld,\ra,\rd)=\sum_{\alpha\vDash n}c_{\alpha}|\{\pi\in \mathfrak{S}_n \suchthat \Dsc(\pi)=\set(\alpha)\}|{\frac{x^n}{n!}}.
\end{align}
Note that the functional equation for $B$ in \eqref{eq: B functional equation} does not immediately imply an expansion of the form in \eqref{eqn:double gamma-nonnegativity}. We use this expansion to turn our discussion to another notion of importance both in algebraic combinatorics and discrete geometry, that of $\gamma$-nonnegativity.

We say that a polynomial $P(t)$ of degree $n\geq 0$ is \bemph{$\gamma$-nonnegative} if it has an expansion of the form
\begin{align}
P(t)=\sum_{j=0}^{\lfloor \frac{n}{2}\rfloor}\gamma_{n,j}t^j(1+t)^{n-2j},
\end{align}
where $\gamma_{n,j}\geq 0$.
{If such an expansion exists, then $P(t)$ is also palindromic and unimodal.}
We refer the reader to \cite[Chapter 4]{Petersen} for a book exposition and \cite{Athanasiadis-slides} for a detailed exhaustive survey on $\gamma$-nonnegativity.
For another recent survey on the relevance and prevalence of  $\gamma$-nonnegativity and  real-rootedness of polynomials arising naturally in combinatorics, the reader is referred to \cite{Branden}.
Our focus here is the connection between intransitive trees of Postnikov \cite{Postnikov-JCTA} and regions of Linial arrangements.

Following Postnikov \cite{Postnikov-JCTA}, an \bemph{intransitive tree} on $n$ nodes is a tree whose nodes are labeled with distinct positive integers from $[n]$ such that the label of a node is either greater than labels of its neighbors, in which case we call it a \bemph{right vertex}, or is less than the labels of its neighbors, in which case we call it a \bemph{left vertex}.
Note that the trees considered by Postnikov are neither plane nor rooted, and they do not have to be binary.
We refer the reader to \cite{Postnikov-JCTA} for further details on the terminology.
Let $f_n(t)\coloneq \sum_{k\geq 1} f_{nk}t^k$ where $f_{nk}$  is the number of intransitive trees on $[n+1]$ with $k$ right vertices.
Consider the generating function
\begin{align}\label{eqn:Postnikov's generating function}
F(t,x)=\sum_{n\geq 0}f_{n}(t)\frac{x^n}{n!}.
\end{align}
By \cite[Theorem 3]{Postnikov-JCTA}, we have that $F\coloneq F(t,x)$ satisfies the functional equation
\begin{align}\label{eqn:Postnikov functional equation}
F(F+t-1)=te^{x(F+t)}.
\end{align}
We note that in Postnikov's statement of the above functional equation, the roles of $x$ and $t$ are switched.
Consider the functional equation satisfied by $\widetilde{B}\coloneq 1+\ra B(x;1,0,\ra,0)$. From~\eqref{eq: B functional equation}, it can be seen that
\begin{align}\label{eqn:Linial functional equation}
\widetilde{B}(\widetilde{B}+\ra-1)=\ra e^{x(\widetilde{B}+\ra)}.
\end{align}
By comparing \eqref{eqn:Postnikov functional equation} and \eqref{eqn:Linial functional equation}, we obtain the following proposition.
\begin{proposition}\label{prop:intransitive and lbs}
For $n\geq 1$, the number of intransitive trees on $[n+1]$ with $k$ right vertices equals the number of standard LBS trees on $[n]$ with $k-1$ right edges.
\end{proposition}

Setting $\la=1$, $\ld=0$, $\ra=t$, and $\rd=0$ in Theorem~\ref{cor: Ribbon Expansion}, we have
\begin{align}\label{eqn:gamma-nonnegativity of LBS}
G_n(\alpx;1,0,t,0) &=\sum_{\text{LBS }T\in \pbtl_n} t^{\rightt}\alpx^T=\sum_{\pi = B_1/\dots/B_k}t^{\ja(\pi)}(t+1)^{\sa(\pi)}h_{(|B_1|,\dots,|B_k|)},
\end{align}
where the last sum is over $\pi\in \mnoncrossing{n}$ such that all nodes are unmarked. Hence, it is a sum over all interlacing noncrossing partitions on $n$ nodes.
Applying the homomorphism $\ex$ to the second and third expressions in \eqref{eqn:gamma-nonnegativity of LBS}, we obtain
\begin{align}\label{eq: gamma-nonneg2}
\sum_{\substack{T \in \pbtl_n \\ \text{standard LBS }}} t^{\rightt}=\sum_{\substack{\pi = B_1/\dots/B_k\\ \text{interlacing on $[n]$}}}t^{\ja(\pi)}(1+t)^{\sa(\pi)}\binom{n}{|B_1|,\dots,|B_k|},
\end{align}
where $\binom{n}{m_1,\dots,m_k} = \frac{n!}{m_1!\cdots m_k!}$ for a composition $(m_1,\ldots,m_k)\vDash n$.
Thus, we have established that the distribution of right edges over standard LBS trees is $\gamma$-nonnegative.

In fact, we can obtain an explicit combinatorial description for the coefficients in the $\gamma$-nonnegative expansion.
For $0\leq j\leq \frac{n-1}{2}$, let $\gamma_{n,j}$ denote the number of left-leaning Bernardi trees on $n$ nodes such that exactly $j$ nodes have two children. Using Edelman's bijection between $NC(n)$ and $\pbt_n$, one can show that \eqref{eq: gamma-nonneg2} implies
\begin{align}
\sum_{\substack{T \in \pbtl_n\\ \text{standard LBS}}} t^{\rightt}=\sum_{0\leq j\leq \frac{n-1}{2}}\gamma_{n,j}t^j(1+t)^{n-1-2j}.
\end{align}
Combining this with Proposition~\ref{prop:intransitive and lbs}, we have the following theorem.
\begin{theorem}
For $n\geq 1$, the distribution of right edges over the set of standard LBS trees on $n$ nodes is $\gamma$-nonnegative. Equivalently, the polynomials $f_{n}(t)$ in~\eqref{eqn:Postnikov's generating function} considered by Postnikov are $\gamma$-nonnegative. As a corollary, we have that the sequence of coefficients of $f_n(t)$  is unimodal.
\end{theorem}
\noindent In the spirit of the theme in \cite{Branden}, we offer the following stronger conjecture.
\begin{conjecture}
 The polynomials $f_{n}(t)$ are real-rooted with all roots negative for all $n\geq 1$.
 In particular, the coefficients of $f_{n}(t)$ form a log-concave sequence.
 \end{conjecture}

 \section{A bijective proof of Theorem~\ref{thm: ribbon functional equation}}
 \label{sec: Proof of Main Theorem}

 In this section, after setting up the necessary notation, we present a bijective proof of Theorem~\ref{thm: ribbon functional equation}.
We begin by restating the identity that we seek to establish for the convenience of the reader,
 \begin{equation}\label{eqn:main_eqn_restated}
   G = \sum_{n\geq 1}\sum_{\alpha\vDash n} (\la\ra\, G+\la+\ra)^{n-\ell(\alpha)}(\ld\rd\,G+\ld+\rd)^{\ell(\alpha)-1}\,r_\alpha.
 \end{equation}
Since our proof is intricate, we present a broad outline of this section.
 \begin{enumerate}
  \item We first interpret the right-hand side of \eqref{eqn:main_eqn_restated}  as the multivariate generating function of the set of alternating sequences of labeled trees and lattice paths.
  This is accomplished in Lemma~\ref{lem: PT equation}.
  \item We then discuss in detail the case of labeled trees that contribute to $G_{3,\nu}$ for all possible canopies $\nu\in \{UU,UD,DD,DU\}$. To each standard labeled tree on $3$ nodes we associate an alternating sequence of labeled trees and lattice paths.
  \item The insight gained from understanding the $n=3$ case leads us to the crucial notions of distinguished triples, life-sustaining nodes, and prunable nodes in Subsection~\ref{subsec: prunable}.
  \item The prunable nodes in a labeled tree determine a particular partition of the nodes of the tree, which in turn determines an alternating sequence of labeled trees and lattice paths.
  In Subsection~\ref{subsec: Construction of Phi}, we show that this correspondence is in fact a weight-preserving bijection between the set of labeled binary trees and the set of alternating sequences of labeled trees and lattice paths.
  \item Finally, in Section~\ref{sec: ribbon expansion}, we use this weight-preserving bijection to prove Theorem~\ref{conj: refined Schur positivity}. The notions of (augmented and marked) interlacing partitions defined in Section~\ref{sec: fixed canopy} can be seen to arise naturally from our act of dismantling a tree into an alternating sequence of labeled trees and lattice paths.
\end{enumerate}

 A \bemph{lattice path} is a sequence $\canopy = (\canopy_1,\dots,\canopy_k)$ of points in the plane starting at $\canopy_1=(0,0)$ such that $\canopy_{i+1}$ is either $\canopy_i+(1,1)$ or $\canopy_i+(1,-1)$. We identify the lattice path $\canopy$ with its corresponding path graph where points are nodes and for each $i<k$, we have an edge joining $\canopy_i$ and $\canopy_{i+1}$.
 Let $\vstart(\canopy)\coloneq \canopy_1$ and $\vend(\canopy)\coloneq \canopy_k$.
 If $\canopy_{i+1}=\canopy_i+(1,1)$, we say that the edge between $\canopy_i$ and $\canopy_{i+1}$ is an \bemph{up step} in the path, and denote it by $U$.
 Otherwise, if $\canopy_{i+1}=\canopy_i+(1,-1)$, we say that the edge between the nodes is a \bemph{down step} in the path and denote it by $D$.
 A lattice path consisting of $k$ nodes is said to be of \bemph{length} $k-1$.
 Let $w_\canopy$ be the length $k-1$ word on the alphabet $\{U,D\}$ recording the up and down steps of $\canopy$ from left to right.

 Given a positive integer $k$, a \bemph{labeled lattice path} of \bemph{length} $k-1$ is the data of a lattice path $\canopy$ of length $k-1$ and a labeling of the nodes of $\canopy$ with positive integers. If $\canopy_1,\dots,\canopy_k$ are the nodes of $\canopy$, we let $\canopy_i^\ell$ denote the label of $\canopy_i$.
 If $\canopy$ is a labeled lattice path, its \bemph{inorder reading word} $\inorder(\canopy)$  is defined to be $\canopy_1^\ell \canopy_2^\ell\dots  \canopy_{k}^\ell$.

 To the edge joining the labeled nodes $\canopy_i$ and $\canopy_{i+1}$ in $\canopy$, we associate an \bemph{edge weight} based on its orientation as indicated in Figure~\ref{fig:weightdiagram}. Precisely, if $\canopy_i\canopy_{i+1}$ form a $U$ step, then the edge is assigned the weight $\la$ if $\canopy_i^\ell\leq \canopy_{i+1}^\ell$, or the weight $\ld$ if $\canopy_i^\ell>\canopy_{i+1}^\ell$. On the other hand, if $\canopy_i\canopy_{i+1}$ form a $D$ step, then the edge is assigned the weight $\ra$ if $\canopy_i^\ell\leq \canopy_{i+1}^\ell$, or the weight $\rd$ if $\canopy_i^\ell > \canopy_{i+1}^\ell$. Note that the barred parameters $\la$ and $\ra$ correspond to weak inequalities, while the unbarred parameters $\ld$ and $\rd$ correspond to strict inequalities.
 \begin{figure}[H]
 \includegraphics[scale=0.5]{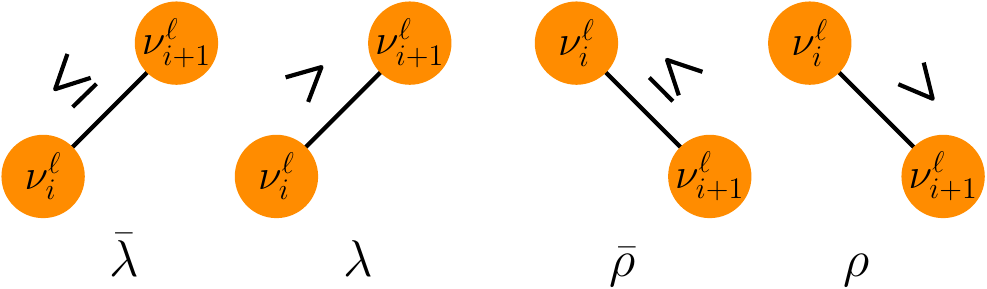}
   \caption{Determining the weight of a labeled edge.}
   \label{fig:weightdiagram}
 \end{figure}
 We define the \bemph{weight} of the labeled lattice path $\canopy$ to be the product of the weights on its edges, and denote it by $\wt{\canopy}$. 
 Observe that the weights in Figure~\ref{fig:weightdiagram} are consistent with the weights assigned to the edges of a labeled tree based on their orientation.
{Suppose we draw a labeled (plane binary) tree with the root on top, left edges going Southwest to Northeast, and right edges going Northwest to Southeast. If we defined the weight of a labeled {(plane binary)} tree to be the product of the weights along its edges according to the rules in Figure~\ref{fig:weightdiagram}, then this new weight in fact coincides with the one in Subsection~\ref{subsec: labeled trees}.}

 We define a \bemph{path-tree sequence} to be an alternating sequence
 \begin{align}
 S=(\canopy^{(0)},T_1,\canopy^{(1)},\dots,\canopy^{(m-1)},T_{m},\canopy^{(m)})
 \end{align}
 of labeled lattice paths $\canopy^{(h)}$ and labeled trees $T_h$, with $m\geq 0$, such that the sequence starts and ends with a lattice path.
 We define the \bemph{inorder reading word} of  $S$, denoted by $\inorder(S)$, to be the concatenation of the inorder reading words of the labeled lattice paths and labeled trees in the order in which they appear in $S$ from left to right,
 \begin{align}
 \inorder(S) \coloneq \inorder(\canopy^{(0)})\inorder(T_1)\inorder(\canopy^{(1)})\cdots\inorder(\canopy^{(m-1)})\inorder(T_{m})\inorder(\canopy^{(m)}),
 \end{align}
 where the ellipsis here denotes that the concatenation continues.
 Define a monomial $\alpx^S$ associated to $S$ by
 \begin{align}
 \alpx^S\coloneq \alpx^{\inorder(S)} = \alpx^{\inorder(\canopy^{(0)})} \prod_{h=1}^{m} \alpx^{T_h}\,\alpx^{\inorder(\canopy^{(h)})}.
 \end{align}
 Let the \bemph{weight} of $S$ be
 \begin{equation}\label{eq: S weight}
   \wt{S} \coloneq \wt{\canopy^{(0)}} \prod_{h=1}^m \gamma_h\,\wt{T_h}\,\wt{\canopy^{(h)}},
 \end{equation}
 where for each $1\leq h\leq m$,
 \begin{align}\label{eq: def of gamma}
   \gamma_h \coloneq \begin{cases}
     \la \ra & \text{ if } \vend(\canopy^{(h-1)})^\ell \leq \vstart(\canopy^{(h)})^\ell,\\
     \ld\rd & \text{ if } \vend(\canopy^{(h-1)})^\ell > \vstart(\canopy^{(h)})^\ell.
   \end{cases}
 \end{align}
 We define the \bemph{canopy} of $S$ to be the word on $\{U,D\}$ given by the concatenation
 \begin{align}\label{eq: canopy of S}
 \can{S} = w_{\canopy^{(0)}} \,D\, \can{T_1} \,U\, w_{\canopy^{(1)}} \cdots w_{\canopy^{(m-1)}} D\,\can{T_{m}}\, U \,w_{\canopy^{(m)}}.
 \end{align}
 Given $\canopy$ a word on $\{U,D\}$ of length $n-1$, let $\ptl_{n,\canopy}$ be the set of all path-tree sequences with canopy $\canopy$. See Figure~\ref{fig: PT example} for a path-tree sequence $S$ with $\inorder(S) = 314275434222$, $\alpx^S = x_1 x_2^4 x_3^2 x_4^3 x_5 x_7$, $\wt{S} = (\rd\ra)(\la\ra)(\ld\ra)(\ld)(\ld\rd)(\ld\ra)(1) = \la \ld^4 \ra^4 \rd^2$ and $\can{S} = DDDDUUUDUDU$.

 \begin{figure}[h]
 \includegraphics[scale=0.75]{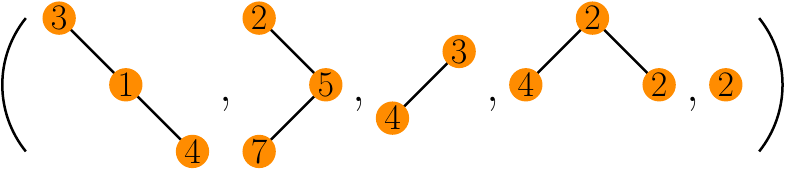}
 \caption{A path-tree sequence $S$.\label{fig: PT example}}
 \end{figure}

 We denote the set of all path-tree sequences of any finite length by $\ptl$. Recall that by our convention, the $\ell$ in this notation refers to the fact that all path-tree sequences have labeled nodes.
 The weighted generating function of $\ptl$ is equal to the right-hand side of the equation in Theorem~\ref{thm: ribbon functional equation} by the following lemma.

 \begin{lemma}\label{lem: PT equation}
 We have that
   \begin{align}\label{eq: PT equation}
     \sum_{S\in \ptl}\wt{S}\,\alpx^{S} =\sum_{n\geq 1}\sum_{\alpha\vDash n}(\la\ra\, G+\la+\ra)^{n-\ell(\alpha)}(\ld\rd\, G+\ld+\rd)^{\ell(\alpha)-1}\,r_\alpha.
   \end{align}
 \end{lemma}
 \begin{proof}
 By the definition of $r_\alpha$ in \eqref{eq: ribbon schur}, we can rewrite the right-hand side of~\eqref{eq: PT equation} as
 \begin{align}\label{eq: right side}
 \phantom{=}\sum_{n\geq 1}\sum_{\alpha\vDash n}&(\la\ra\, G+\la+\ra)^{n-\ell(\alpha)}(\ld\rd\, G+\ld+\rd)^{\ell(\alpha)-1}\,r_\alpha \\
 &= \sum_{n\geq 1}\sum_{w\in \bP^n} (\la\ra\,G+\la+\ra)^{\asc(w)}(\ld\rd\, G+\ld+\rd)^{\des(w)}\,\alpx^w.
 \end{align}
 For $w\in \bP^n$, consider the summand
 \begin{align}\label{eq: factor}
 (\la\ra\, G+\la+\ra)^{\asc(w)}(\ld\rd\, G+\ld+\rd)^{\des(w)}\,\alpx^w,
 \end{align}
 which can be rewritten as
 \begin{align}
 \left\lbrack\prod_{i\in \Asc(w)}(\la\ra\, G + \la + \ra) \prod_{j\in \Des(w)} (\ld\rd\, G + \ld + \rd)\right\rbrack\alpx^w
 =\sum_{(A_1,\dots,A_{n-1})}A_1\dots A_{n-1}\,\alpx^w,\label{eq: product form}
 \end{align}
 where the sum in the right-hand side of~\eqref{eq: product form} is over all tuples $(A_1,\dots,A_{n-1})$ such that
 \begin{align}\label{eq: asc-des condition}
 A_i\in
 \begin{cases}
  \{\la\ra\,\wt{T}\,\alpx^T \,|\, T\in \pbtl\}\cup\{\la,\ra\} & \text{if }i\in \Asc(w),\\
 \{\ld\rd\,\wt{T}\,\alpx^T \,|\, T\in \pbtl\}\cup \{\ld,\rd\} & \text{if }i\in \Des(w).
 \end{cases}
 \end{align}

 For each $w\in\bP^n$, define a function $\mathfrak{f}_w$ from the collection of tuples $(A_1,\dots,A_{n-1})$ satisfying \eqref{eq: asc-des condition} to $\ptl$ as follows.
 First, define a sequence on $\{U,D\}\cup \pbtl$ by replacing each $\la\ra\, \wt{T}\,\alpx^T$ and $\ld\rd\,\wt{T}\,\alpx^T$ with $T$ in $(A_1,\dots,A_{n-1})$ and then replacing each remaining $\la$ and $\ld$ with a $U$ and each remaining $\ra$ and $\rd$ with a $D$.
 Then replace each maximum consecutive subsequence of $U$s and $D$s with its corresponding lattice path.
 If two trees appear consecutively in the sequence, then we insert a lattice path consisting of a single node between the two trees in the sequence.
 If the sequence begins with a tree, we insert a lattice path consisting of a single node at the beginning of the sequence. Likewise, if the sequence ends with a tree, we insert a lattice path consisting of a single node at the end of the sequence.
 Observe that the total number of nodes in the disjoint union
 $\canopy^{(0)}\sqcup\canopy^{(1)}\sqcup\dots\sqcup\canopy^{(m)}$
  is $n$.
 Finally, label the nodes of the lattice paths we have just constructed from left to right with the letters in $w$ to obtain a path-tree sequence
 \begin{align}
 S = (\canopy^{(0)},T_1,\canopy^{(1)},\dots,\canopy^{(m-1)},T_{m},\canopy^{(m)})\in \ptl
 \end{align}
 with
 \begin{align}\label{eq: inorder equation}
 \inorder(\canopy^{(0)})\inorder(\canopy^{(1)})\cdots \inorder(\canopy^{(m)}) = w.
 \end{align}
 Observe that
 \begin{align}
 \wt{S}\,\alpx^S = A_1\dots A_{n-1}\,\alpx^w.
 \end{align}

 For example, in the case when $w = 314432$, one possible tuple is
   \begin{align}
 (A_1,A_2,A_3,A_4,A_5) =   (\rd,\ra,\la\ra \,\wt{T_1}\,\alpx^{T_1},\ld,\ld\rd\,\wt{T_2}\,\alpx^{T_2}),
   \end{align}
       for some labeled trees $T_1$ and $T_2$. Then $\mathfrak{f}_w(A_1,\dots,A_5) = $
       \begin{center}
 \includegraphics[scale=0.75]{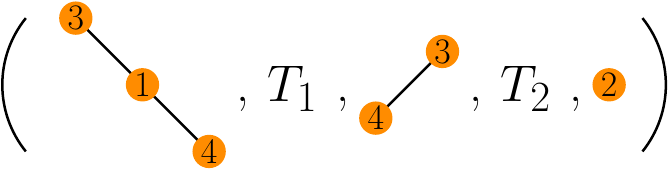},
 \end{center}
 which is an element of $\ptl$.

 The map $\mathfrak{f}_w$ is a bijection onto the subset of $\ptl$ of alternating sequences $S = (\canopy^{(0)},T_1,\canopy^{(1)},\dots,\canopy^{(m-1)},T_{m},\canopy^{(m)})$ which satisfy \eqref{eq: inorder equation}. Indeed, given a sequence $S = (\canopy^{(0)},T_1,\canopy^{(1)},\dots,\canopy^{(m-1)},T_{m},\canopy^{(m)})\in \ptl$, we can recover $w$ by reading off the labels of the lattice paths in $S$ from left to right. Furthermore, the tuple $(A_1,\dots,A_{n-1})$ can be recovered by recording the weights of the edges of the lattice paths together with $\gamma_h \wt{T_h}\alpx^{T_h}$ in the order in which the edges and trees appear in $S$. Therefore,
 \begin{align}
 \sum_{S\in \ptl}\wt{S}\,\alpx^{S} &= \sum_{n\geq 1}\sum_{w\in \bP^n} \sum_{(A_1,\dots,A_{n-1})} A_1\dots A_{n-1}\,\alpx^w\label{eq: first equality of Prop 4.1}\\
 &=\sum_{n\geq 1}\sum_{\alpha\vDash n}(\la\ra\, G+\la+\ra)^{n-\ell(\alpha)}(\ld\rd\, G+\ld+\rd)^{\ell(\alpha)-1}\,r_\alpha,
 \end{align}
 where the inner sum in the right side of \eqref{eq: first equality of Prop 4.1} is the right-hand side of \eqref{eq: product form}.
 \end{proof}

   By Lemma~\ref{lem: PT equation}, in order to prove Theorem~\ref{thm: ribbon functional equation} it suffices to show that
   \begin{align}\label{eq: GS equality}
   G = \sum_{S\in \ptl} \wt{S}\,\alpx^S.
   \end{align}
   Before we proceed to prove this identity, we first verify that $G_3$ is equal to the degree 3 homogeneous component of the right-hand side of \eqref{eq: GS equality}. This example will give us insight into how to prove the identity and will motivate many of the ideas in the rest of this section.

   \begin{example}\label{ex: 3 Nodes}
   Recall from \eqref{eq: fixed canopy sum} that for $\canopy$ a word on $\{U,D\}$ of length $n-1$, we have $G_{n,\canopy}$ is the weighted generating function summing over trees with canopy $\canopy$. We have that
   \[
   G_3 = G_{3,UU} + G_{3,DD} + G_{3,UD} + G_{3,DU}.
   \]
  Recall from \eqref{eq: canopy of S} the definition of the canopy of a path-tree sequence.  We show that for each canopy $\canopy$, we have that
   \begin{align}\label{eq: canopy sum}
   G_{3,\canopy} = \sum_{S\in \ptl_{3,\canopy}} \wt{S}\, \alpx^S.
   \end{align}
   In the case $\canopy = UU$, it is not hard to verify that
   \begin{align}
   G_{3,UU} &= \la^2r_{(3)} + \la\ld r_{(2,1)} + \ld\la r_{(1,2)} + \ld^2 r_{(1,1,1)}\\
   &= \sum_{S\in \ptl_{3,UU}} \wt{S}\, \alpx^{S}.\label{eq: UU sum}
   \end{align}
 Indeed, the path-tree sequences $S$ with $\can{S} = UU$ are those consisting of a single lattice path whose steps are given by $UU$. On the other hand, the only binary tree $T$ with $\can{T} = UU$ is a path graph starting from the root and going down and to the left. Therefore, we can define a bijection $\Phi_{3,UU}:\pbtl_{3,UU}\to\ptl_{3,UU}$ which simply converts the two left edges of $T$ into $U$ steps in a lattice path. Furthermore, the map $\Phi_{3,UU}$ preserves weight, meaning that $\wt{T} = \wt{\Phi_{3,UU}(T)}$.
 Similar statements hold with $UU$ replaced by $UD$ and $DD$.

 In the case $\canopy=DU$, there is no obvious way to convert a labeled tree $T$ with $\can{T} = DU$ into a labeled path with steps $DU$ while preserving weight. Recall from Subsection \ref{subsec: labeled trees} that the weight of a labeled tree $T$ is the same as the weight of its inorder standardization. Let us partition the set of labeled trees $T$ with canopy $DU$ according to their inorder standardization $\std(T)$. The 12 possibilities for $\std(T)$ are listed in Figure~\ref{fig: standardizations}.
 \begin{figure}[H]
 \includegraphics[width=\textwidth]{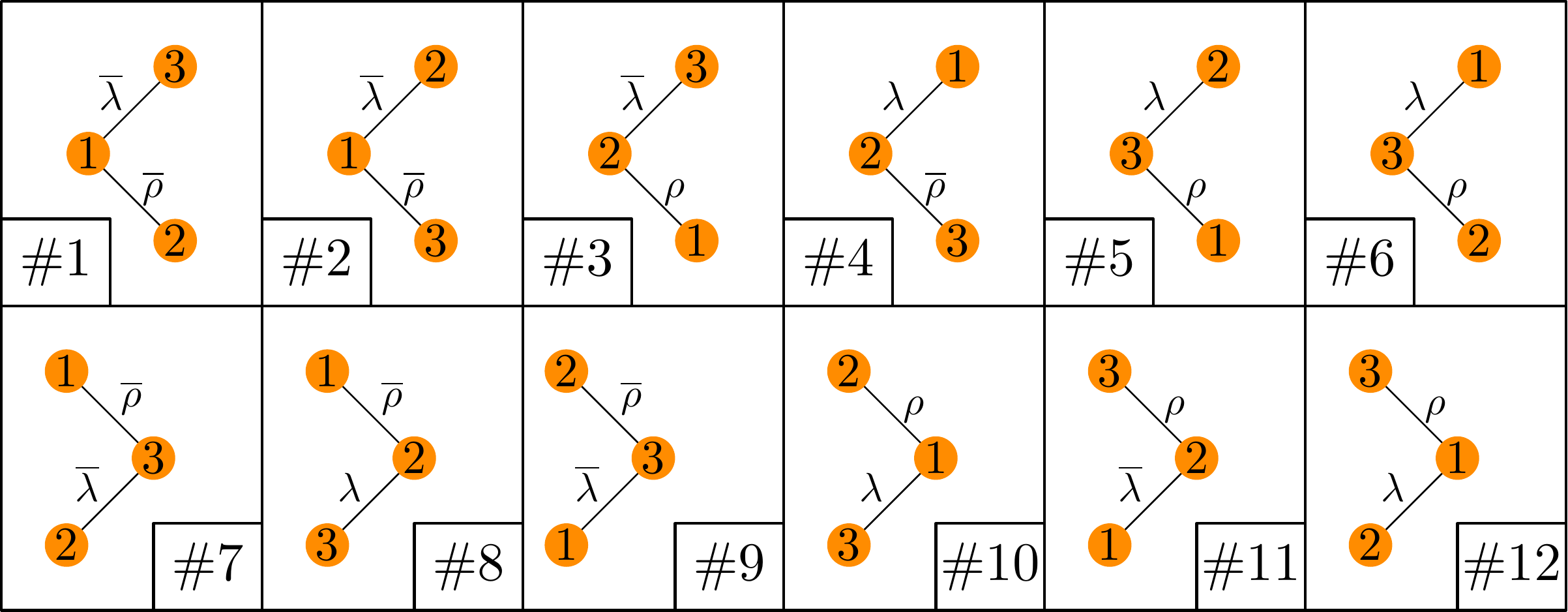}
 \caption{The 12 possibilities for $\std(T)$ for $T$ with $\can{T} = DU$.\label{fig: standardizations}}
 \end{figure}
 One can verify that if we restrict $G_3$ to summing over labelings of the trees in the first row of Figure~\ref{fig: standardizations}, the resulting multivariate formal power series is \emph{quasisymmetric} but not symmetric. Therefore, if we wish to expand $G_3$ in terms of ribbon Schur functions, we must group terms corresponding to labeled trees in the first row of Figure~\ref{fig: standardizations} with terms corresponding to labeled trees in the second row of Figure~\ref{fig: standardizations}.

 It remains to show that \eqref{eq: canopy sum} holds for $\canopy=DU$.
 Observe that if $S$ has canopy $DU$, then either $S=(\canopy^{(0)})$ where $w_{\canopy^{(0)}} = DU$, or $S=(\canopy^{(0)},T_1,\canopy^{(1)})$ where $\canopy^{(0)}$, $T_1$ and $\canopy^{(1)}$ each consist of a single node.
 Define the \bemph{inorder standardization} of a path-tree sequence $S$, denoted by $\std(S)$ to be the unique path-tree sequence whose underlying unlabeled paths and trees are the same as those of $S$, and whose inorder reading word is $\std(\inorder(S))$. Figure~\ref{fig: PT standardizations} lists the possible standardizations of the path-tree sequences with canopy $DU$.
 \begin{figure}[H]
 \includegraphics[width=\textwidth]{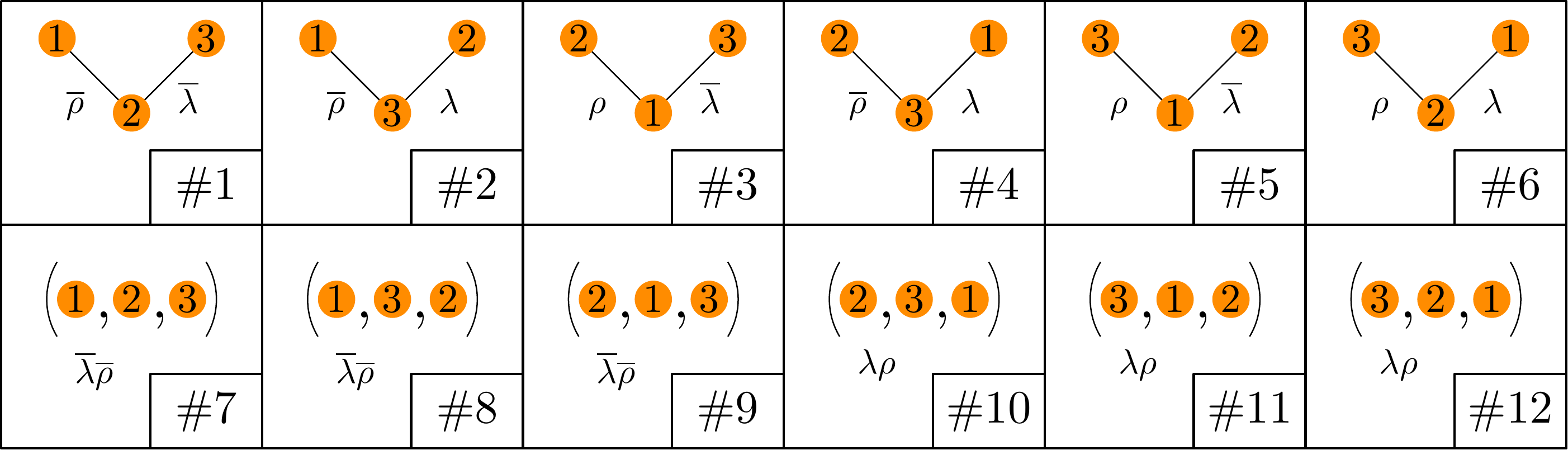}
 \caption{The 12 possible standardizations of the path-tree sequences with canopy $DU$.\label{fig: PT standardizations}}
 \end{figure}
 Our goal is to find a bijection $\Phi_{3,DU}$ between the 12 labeled trees $T$ whose inorder standardization is listed in Figure~\ref{fig: standardizations} and the 12 path-tree sequences $S$ whose inorder standardization is listed in Figure~\ref{fig: PT standardizations}. We would also like to construct the map so that $\Phi_{3,DU}$ preserves weight and inorder reading word.
 See Figure~\ref{fig: summary Phi} for a summary of the map $\Phi_{3,DU}$ we defined on standard trees with canopy $DU$. The reader is encouraged to verify that this map preserves weight and inorder reading word using Figure~\ref{fig: summary Phi}.

 Observe that for some trees the image is uniquely determined by the weight and inorder reading word requirements, while for other trees, there was some choice for the image. For example, looking at tree $\#2$ in Figure~\ref{fig: standardizations}, we see that its inorder reading word is $132$ and its weight is $\la\ra$. In order to preserve both its inorder reading word and weight, $\Phi_{3,DU}$ must map it to path-tree sequence $\#8$ in Figure~\ref{fig: PT standardizations}.
 On the other hand, both trees $\#1$ and $\#7$ have inorder reading word $123$ and weight $\la\ra$. Similarly, both path-tree sequences $\#1$ and $\#7$ have inorder reading word $123$ and weight $\la\ra$. We make the arbitrary choice to map tree $\#1$ to path-tree sequence $\#1$ and tree $\#7$ to path-tree sequence $\#7$.

 \begin{figure}[h]
 \includegraphics[width=\textwidth]{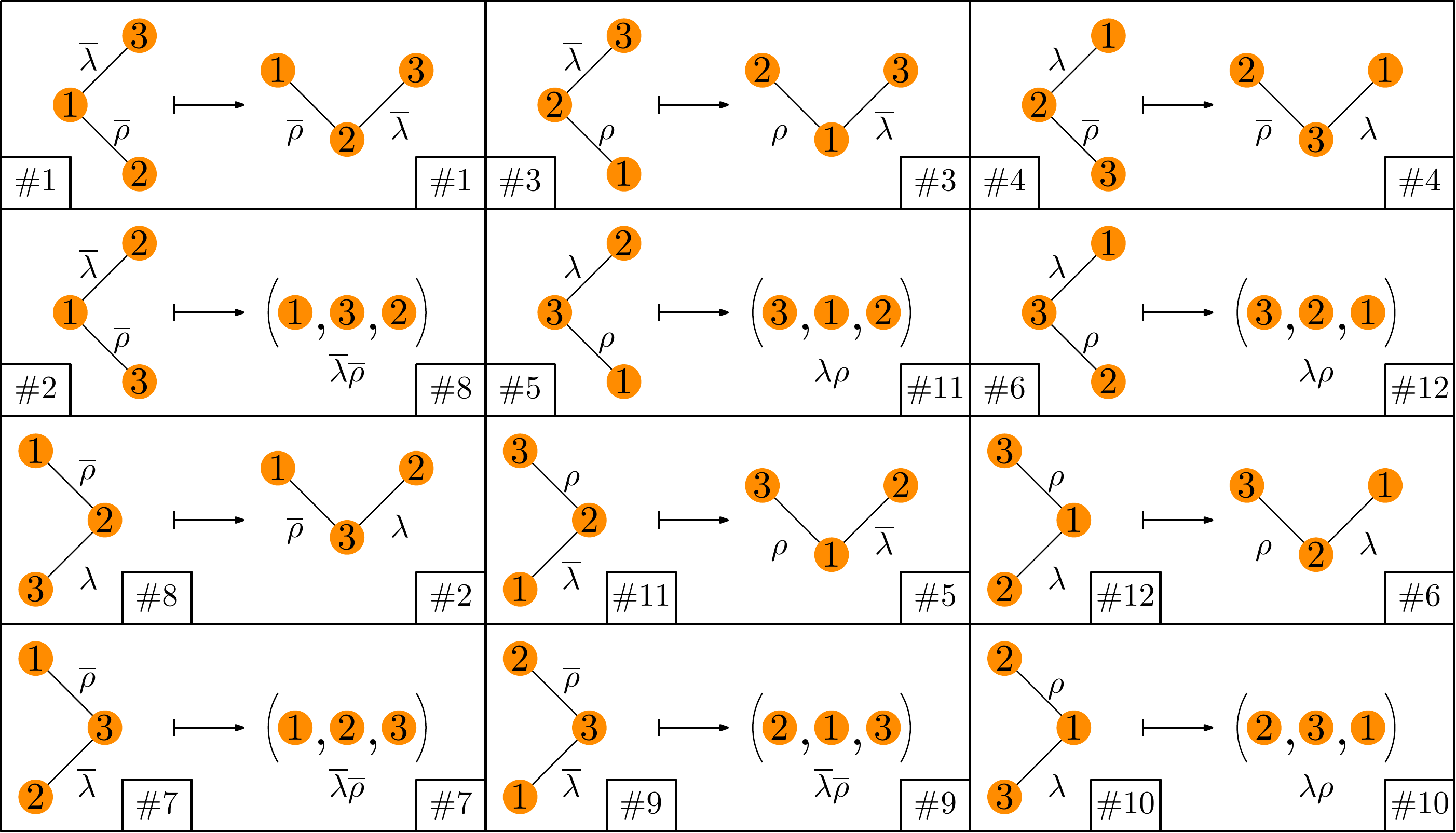}.
 \caption{A summary of the map $\Phi_{3,DU}$ defined on standard trees.\label{fig: summary Phi}}
 \end{figure}

 We extend the definition of $\Phi_{3,DU}$ to a map $\Phi_{3,DU}: \pbtl_{3,DU}\to\ptl_{3,DU}$ as follows. Given $T\in \pbtl_{3,DU}$, suppose that $\std(T)$ maps to the path-tree sequence $S$ in Figure~\ref{fig: summary Phi}. Then define $\Phi_{3,DU}(T)$ to be the unique path-tree sequence whose underlying unlabeled paths and trees are the same as those in $S$, such that the inorder reading word is $\inorder(T)$. Therefore, $\Phi_{3,DU}$ is a bijection which preserves weight, inorder reading word, and canopy. Then we have
 \begin{align}
 G_{3,DU} &= \sum_{T\in \pbtl_{3,DU}} \wt{T}\,\alpx^{T}\\
 &= \sum_{T\in \pbtl_{3,DU}} \wt{\Phi_{3,DU}(T)}\,\alpx^{\Phi_{3,DU}(T)}\\
   &= \sum_{S\in \ptl_{3,DU}} \wt{S}\, \alpx^{S},
 \end{align}
 so \eqref{eq: canopy sum} holds for $\canopy=DU$.

 We make one more observation from this example. Let $T$ be a tree whose inorder standardization is a tree listed in either the second or fourth row of Figure~\ref{fig: summary Phi}.
 Let $v_2$ be the second node of $T$ in inorder, which is the lowest node drawn in each of the trees in these rows. We can think of the map $\Phi_{3,DU}$ as ``pruning'' off $v_2$ from $T$ and mapping $T$ to $(v_1,v_2,v_3)$, where we think of $v_2$ as a single node tree sandwiched between $v_1$ and $v_3$.
   \end{example}

 The outline of the proof of Theorem~\ref{thm: ribbon functional equation} is organized as follows. In Subsection~\ref{subsec: prunable}, we define the notion of a ``prunable node'' in a tree and show that we can identify the weights of a corresponding set of edges in the tree. In Subsection~\ref{subsec: Construction of Phi}, we use the antichain of maximal prunable nodes in a labeled tree to partition the node set of a labeled tree. We then use this partition to define a map $\Phi_{n,\canopy}: \pbtl_{n,\canopy}\to \ptl_{n,\canopy}$ which generalizes the map $\Phi_{3,DU}$ in Example~\ref{ex: 3 Nodes}. Finally, we use $\Phi_{n,\canopy}$ to prove the identity
 \begin{align*}
 G = \sum_{S\in\ptl} \wt{S}\,\alpx^S
 \end{align*}
 from~\eqref{eq: GS equality}, completing the proof of Theorem~\ref{thm: ribbon functional equation}.

 \subsection{Distinguished triples and prunable nodes}\label{subsec: prunable}
 Throughout this section, let us fix a labeled tree $T\in \pbtl_n$. Let the nodes of $T$ be $v_1,v_2,\dots,v_n$ listed in inorder, as in Figure~\ref{fig:PreInorder}. After setting up some notation, we  identify a special subset of nodes of $T$ as ``prunable''. This set of nodes will allow us to give a simple description of the weights of a corresponding set of edges in $T$.

 Let the \bemph{roof} of $T$ be the set of nodes of $T$ which can be reached from the root by either traversing only left edges or traversing only right edges. Let the set of nodes which are not in the roof be the \bemph{attic}.

 \begin{definition}\label{def: distinguished ancestor}
 Let $v_j\in \node{T}$ be a node in the attic of $T$.
 \begin{itemize}
 \item If $v_j$ is the left child of $v_{k}$ for some $k>j$, let $i$ be the greatest index smaller than $j$ such that $v_i$ is an ancestor of $v_j$. Define the \bemph{distinguished triple} corresponding to $v_j$ to be $(v_i,v_j,v_k)$, and define the \bemph{distinguished ancestor} of $v_j$ to be $v_i=\da(v_j)$.
 \item If $v_j$ is the right child of $v_i$ for some $i<j$, let $k$ be the smallest index greater than $j$ such that $v_k$ is an ancestor of $v_j$. Define the \bemph{distinguished triple} of $v_j$ to be $(v_i,v_j,v_k)$, and define the \bemph{distinguished ancestor} of $v_j$ to be $v_k=\da(v_j)$.
 \end{itemize}
 \end{definition}
 {Observe that if $v_j\in \node{T}$ is in the attic of $T$, then its distinguished ancestor always exists. Indeed, assuming $v_j$ is the left child of $v_k$ for some $k>j$, if there exists no $i<j$ such that $v_i$ is an ancestor of $v_j$, then $v_j$ must be in the roof of $T$, a contradiction, hence the distinguished ancestor of $v_j$ exists. Similar reasoning holds when $v_j$ is the right child of $v_i$ for some $i<j$. See Figure~\ref{fig: distinguished ancestor schematic} for a schematic diagram of the two cases in Definition~\ref{def: distinguished ancestor}, where the blue triangles indicate arbitrary subtrees of $T$.}
 
  \begin{figure}[ht]
 \centering
 \includegraphics[scale=0.4]{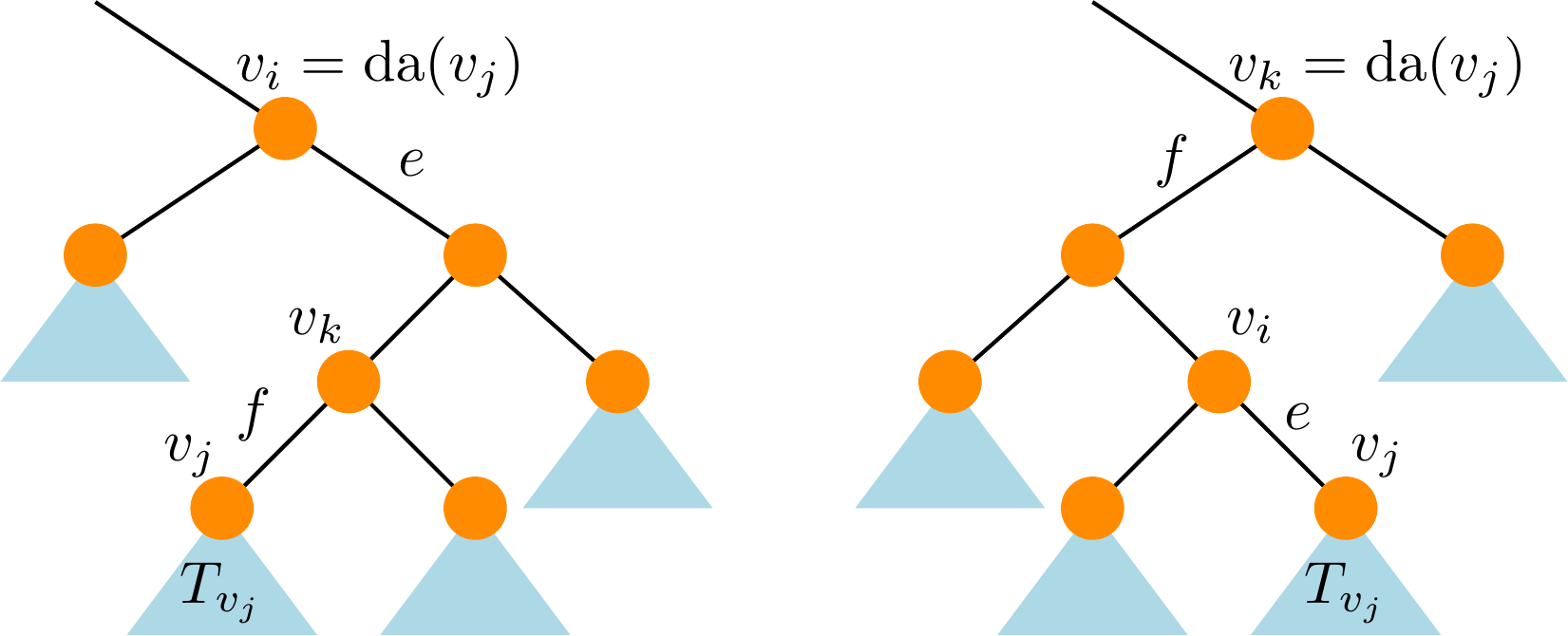}
 \caption{A schematic diagram for distinguished triples $(v_i,v_j,v_k)$.\label{fig: distinguished ancestor schematic}}
 \end{figure}
 
  \begin{figure}[t]
 \centering
 \begin{minipage}{.45\textwidth}
 \centering
 \includegraphics[scale=0.4]{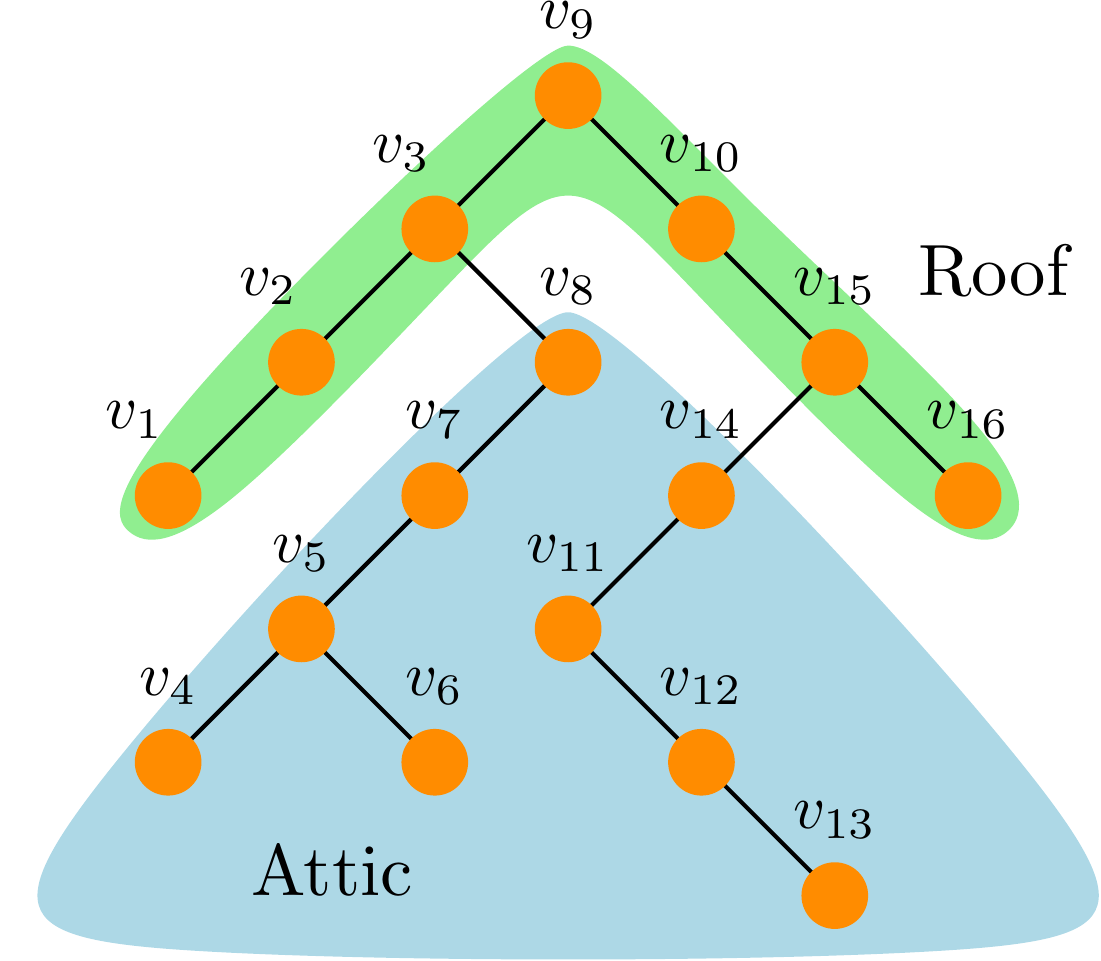}
 \end{minipage}\hfill
 \begin{minipage}{.45\textwidth}
 \centering
 $\begin{tabu}{|c|c|c|}\hline
 v_i & v_j & v_k\\\hline\hline
 v_3^* & v_4 & v_5\\\hline
 v_3^* & v_5 & v_7\\\hline
 v_5 & v_6 & v_7^*\\\hline
 v_3^* & v_7 & v_8\\\hline
 v_3 & v_8 & v_9^*\\\hline
 v_{10}^* & v_{11} & v_{14}\\\hline
 v_{11} & v_{12} & v_{14}^*\\\hline
 v_{12} & v_{13} & v_{14}^*\\\hline
 v_{10}^* & v_{14} & v_{15}\\\hline
 \end{tabu}$
 \end{minipage}
   \caption{On the left, a tree $T\in \pbt_{16}$ with its roof and attic highlighted. On the right, a table listing the distinguished triples $(v_i,v_j,v_k)$, where in each row the distinguished ancestor of the node $v_j$ is starred.}
   \label{fig: roof-attic}
 \end{figure}
 
 We call the unique path from $v_j$ to $\da(v_j)$, including both $v_j$ and $\da(v_j)$, the \bemph{distinguished path} of the node $v_j$.
 Figure~\ref{fig: roof-attic} shows a tree $T$ with its roof and attic highlighted, with a list of the distinguished triples for each $v_j$ in the attic of $T$.

 Observe that in each case of Definition \ref{def: distinguished ancestor}, we have $i<j<k$. Furthermore, $v_i$ has a right edge $e$, and $v_k$ has a left edge $f$.
{Further observe that 
 \begin{align}\label{eq: Nodes of Tvj}
 \node{T_{v_j}} = \{v_h\,\suchthat\, i+1\leq h\leq k-1\},
 \end{align}
 where $T_{v_j}$ is the subtree of $T$ consisting of $v_j$ and all of its descendants. Indeed, it can be seen from Figure~\ref{fig: distinguished ancestor schematic} that $i$ is maximum such that $i<j$ and $v_i$ is not in $T_{v_j}$, and $k$ is minimum such that $k>j$ and $v_k$ is not in $T_{v_j}$.
 }

   Next we prove a lemma relating the canopy of $T$ to distinguished ancestors. Recall from Definition~\ref{def: canopy2} that $\can{T}$ is the word on $\{U,D\}$ whose $j$th letter is a $D$ if and only if $v_j$ has a right child.
   \begin{lemma}\label{lem: steps}
 Consider a node $v_j$ in $T$ for $j\leq n-1$. Then the following hold.
  \begin{enumerate}[(a)]
  \item Suppose the $j$th letter of $\can{T}$ is a $U$. Then either $v_{j}$ is the left child of $v_{j+1}$, or $v_j$ is in the attic and $v_{j+1} = \da(v_j)$.
  \item Suppose the $j$th letter of $\can{T}$ is a $D$. Then either $v_{j+1}$ is the right child of $v_j$, or $v_{j+1}$ is in the attic and $v_j = \da(v_{j+1})$.
  \end{enumerate}
   \end{lemma}
   \begin{proof}
   If the $j$th letter of $\can{T}$ is a $U$, then $v_j$ does not have a right child. Therefore, $v_{j+1}$ must be an ancestor of $v_j$. In this case, $v_j$ cannot be the root, so $v_j$ must have a parent. If $v_j$ is the left child of a node, then $v_{j+1}$ must be its parent by the definition of inorder and the fact that $v_j$ has no right child. In this case, the condition in (a) is satisfied, so we can assume that $v_j$ is the right child of some node. By the definition of inorder and the fact that $v_{j+1}$ is its ancestor, $v_j$ must be in the attic. By Definition~\ref{def: distinguished ancestor}, $\da(v_j)=v_k$ for the minimal $k>j$ such that $v_k$ is an ancestor of $v_j$. Since $v_{j+1}$ is an ancestor of $v_j$, then we have $j+1=k$. Hence, $v_{j+1} = \da(v_j)$. This concludes the proof of part (a). The proof of part (b) follows similarly.
   \end{proof}
The reader may verify the claim in Lemma~\ref{lem: steps} in the case of the tree in Figure~\ref{fig: roof-attic}  whose canopy equals $UUDUDUUUDDDDUUD$.

 \begin{definition}
 Given $v_j$ in the attic of $T$, let $(v_i,v_j,v_k)$ be its distinguished triple. We call the node $v_j$ \bemph{life-sustaining} if and only if one of the following criteria holds:
 \begin{enumerate}
 \item[(S1)] The node $v_j$ is a left child, and
 \[\std (v_i^{\ell}v_j^{\ell}v_k^{\ell}) \in \{132,312,321\}.\]
 \item[(S2)] The node $v_{j}$ is a right child, and
 \[\std (v_i^{\ell}v_j^{\ell}v_k^{\ell}) \in \\\{123,213,231\}.\]
 \end{enumerate}
 If $v_j$ is in the attic of $T$ and it is not life-sustaining, we call it \bemph{prunable}.
 \end{definition}
 See Figure~\ref{fig: Noninfection Cases} for a summary of the six cases where $v_j$ is life-sustaining and Figure~\ref{fig: Infection Cases} for a summary of the six cases where $v_j$ is prunable. Observe that if $T$ is a labeled tree on 3 nodes with $\can{T} = DU$, then $v_2$ is life-sustaining if and only if the standardization of $T$ appears in either rows 1 or 3 of Figure~\ref{fig: summary Phi}. These are exactly the cases where $\Phi_{3,DU}(T)$ consists of a single lattice path in Example~\ref{ex: 3 Nodes}. On the other hand, $v_2$ is prunable if and only if the standardization of $T$ appears in either rows 2 or 4 of Figure~\ref{fig: summary Phi}. These are exactly the cases where $\Phi_{3,DU}(T)$ is a path-tree sequence of length 3 in Example~\ref{ex: 3 Nodes}.
 
  \begin{figure}[H]
 \includegraphics[width=\textwidth]{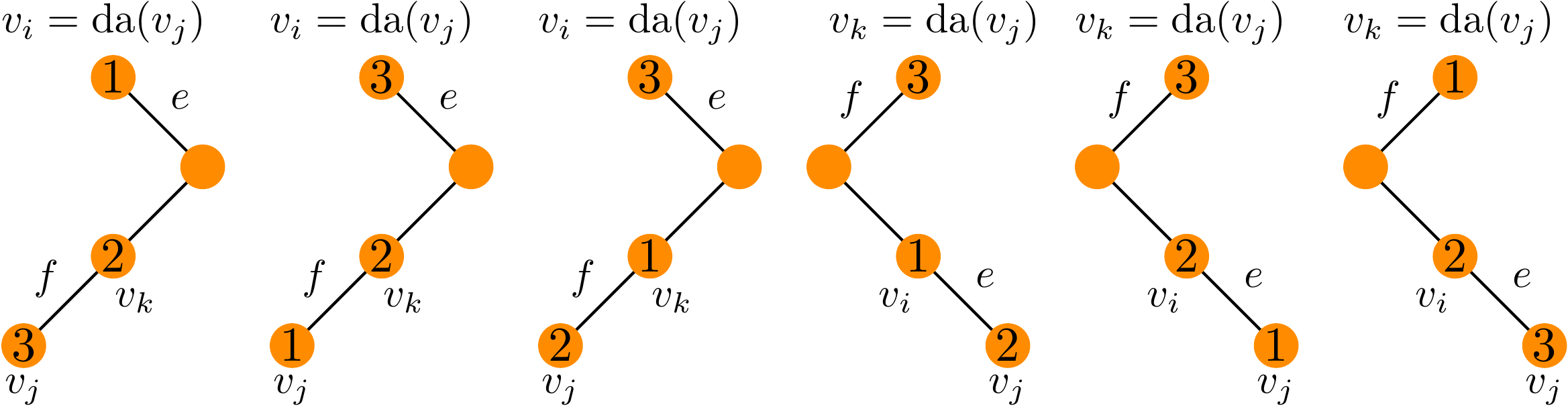}
   \caption{The six cases where $v_j$ is life-sustaining.}
   \label{fig: Noninfection Cases}
 \end{figure}
 \begin{figure}[H]
 \includegraphics[width=\textwidth]{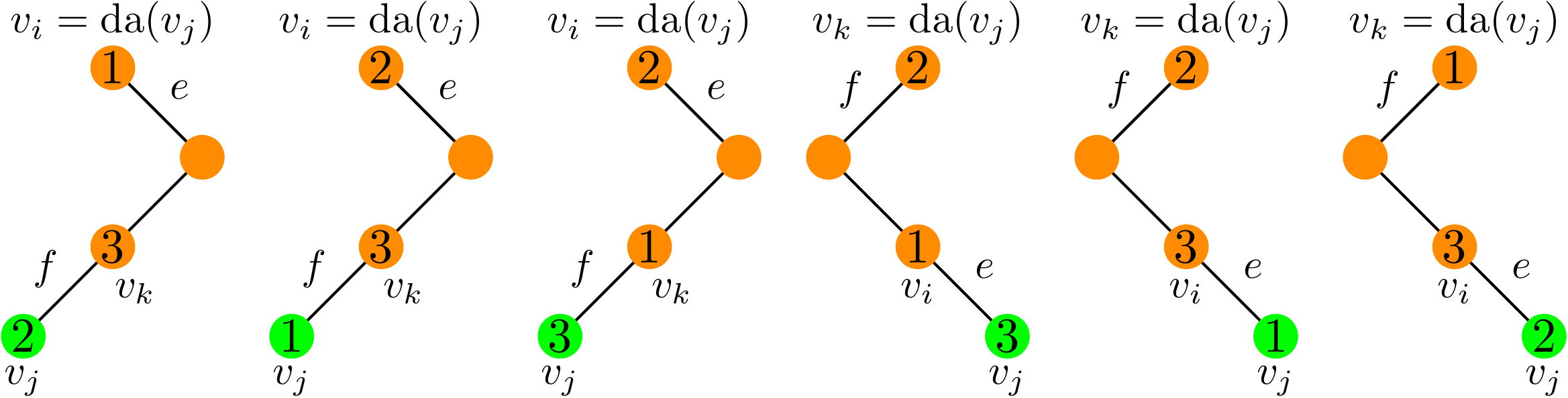}
   \caption{The six cases where $v_j$ is prunable.}
   \label{fig: Infection Cases}
   \end{figure}

 Next we prove a lemma which identifies certain edge weights in the tree associated to life-sustaining nodes. We then use it to prove Lemma~\ref{lem: infection lemma}, which identifies products of edge weight of the form $\la\ra$ and $\ld\rd$. Lemma~\ref{lem: infection lemma} will help explain the appearance of the $\gamma_h$ weights in the definition of the weight of a path-tree sequence~\eqref{eq: S weight}.
 \begin{lemma}\label{lem: Technical Lemma}
   Let $v_j$ be a node in the attic of $T$ such that all of the nodes in the distinguished path of $v_j$ are life-sustaining nodes of $T$. Let $(v_i,v_j,v_k)$ be the distinguished triple of $v_j$, let $e$ be the right edge of $v_i$, and let $f$ be the left edge of $v_k$. Then the following hold.
   \begin{enumerate}[(a)]
    \item If $v_j$ is the left child of $v_k$, then
     \[
     \wt{e} = \begin{cases} \ra & \text{if } v_i^\ell \leq v_j^\ell,\\ \rd & \text{if } v_i^\ell > v_j^\ell.\end{cases}
     \]
     \item If $v_j$ is the right child of $v_i$, then
     \[
     \wt{f} = \begin{cases} \la & \text{if }v_j^\ell\leq v_k^\ell,\\ \ld &\text{if }v_j^\ell > v_k^\ell.\end{cases}
     \]
       \end{enumerate}
 \end{lemma}

 \begin{proof}
 \begin{enumerate}[(a)]
 \item The proof is by case analysis. First, suppose that the distinguished path of $v_j$ consists of only nodes $v_i$, $v_j$ and $v_k$. Looking at Figure~\ref{fig: Noninfection Cases}, we see that if $\std(v_i^\ell v_j^\ell v_k^\ell) = 132$, then $v_i^\ell$ is weakly smaller than both $v_j^\ell$ and $v_k^\ell$. Therefore, we are in the case where $v_i^\ell\leq v_j^\ell$, and $\wt{e} = \wt{v_i v_k} = \ra$. On the other hand, if $\std(v_i^\ell v_j^\ell v_k^\ell) \in \{312,321\}$, then $v_i^\ell$ is strictly greater than both $v_j^\ell$ and $v_k^\ell$. Therefore, we are in the case where $v_i^\ell > v_j^\ell$, and $\wt{e} = \wt{v_i v_k} = \rd$.

 Now, suppose that the distinguished path of $v_j$ consists of more than three nodes. Since all of the nodes along the distinguished path are life-sustaining, there are only two cases. Either $v_i^\ell$ is weakly smaller than all of the other labels along the distinguished path, or $v_i^\ell$ is strictly greater than all of the other labels along the distinguished path. In the first case, $v_i^\ell \leq v_j^\ell$ and $\wt{e} = \ra$. In the second case, $v_i^\ell > v_j^\ell$ and $\wt{e} = \rd$.
 \item A similar case analysis shows that either $v_k^\ell$ is weakly greater than all of the other labels along the distinguished path of $v_j$, or $v_k^\ell$ is strictly smaller than all of the other labels along the distinguished path. In the first case, $v_j^\ell\leq v_k^\ell$ and $\wt{f} = \la$. In the second case, $v_j^\ell > v_k^\ell$ and $\wt{f} = \ld$.\qedhere
 \end{enumerate}
 \end{proof}

 \begin{lemma}\label{lem: infection lemma}
 Let $v_j$ be a node in the attic of $T$. Suppose that $v_j$ is prunable and that all other nodes in the distinguished path of $v_j$ are life-sustaining. 
 Let $(v_i,v_j,v_k)$ be the distinguished triple of $v_j$, let $e$ be the right edge of $v_i$, and let $f$ be the left edge of $v_k$.
 Then
 \begin{align}
 \wt{e}\wt{f} = \begin{cases} \la\ra & \text{if }v_i^\ell \leq v_k^\ell,\\ \ld\rd & \text{if }v_i^\ell > v_k^\ell.\end{cases}
 \end{align}
 \end{lemma}

 \begin{proof}It suffices to show that
 \begin{align}
 \wt{e} &= \begin{cases} \ra & \text{if }v_i^\ell \leq v_k^\ell, \\ \rd &\text{if }v_i^\ell > v_k^\ell,\end{cases}\label{eq: e weight}\\
 \wt{f} & = \begin{cases} \la & \text{if }v_i^\ell \leq v_k^\ell, \\ \ld &\text{if }v_i^\ell > v_k^\ell.\end{cases}\label{eq: f weight}
 \end{align}
 First, suppose $v_j$ is the left child of $v_k$, and let $v_i = \da(v_j)$. If the distinguished path of $v_j$ consists of only the nodes $v_i$, $v_j$ and $v_k$, then $v_k$ is the right child of $v_i$. Therefore, $e = v_iv_k$, so $\wt{e} = \wt{v_iv_k}$, and \eqref{eq: e weight} holds. Otherwise, $v_k$ is in the attic of $T$, and $v_i = \da(v_k)$. By hypothesis, we know that every node on the distinguished path of $v_k$ is life-sustaining. Therefore, applying Lemma~\ref{lem: Technical Lemma}(a) to the distinguished path of $v_k$, {we have that \eqref{eq: e weight} continues to hold}.

 Since $v_j$ is the left child of $v_k$, then \eqref{eq: f weight} follows from a case analysis. Indeed, looking at Figure~\ref{fig: Infection Cases}, then $v_i^\ell\leq v_k^\ell$ exactly when $\std(v_i^\ell v_j^\ell v_k^\ell)\in \{213,123\}$. In both of these cases we have that $v_j^\ell \leq v_k^\ell$, which means exactly that $\wt{f} = \wt{v_j v_k} = \la$, so \eqref{eq: f weight} holds. On the other hand, $v_i^\ell > v_k^\ell$ exactly when $\std(v_i^\ell v_j^\ell v_k^\ell) = 231$. In this case, we have $v_j^\ell > v_k^\ell$, so \eqref{eq: f weight} continues to hold.

 This concludes our proof of the lemma in the case when $v_j$ is the left child of $v_k$. The case when $v_j$ is the right child of $v_i$ follows by a similar application of Lemma~\ref{lem: Technical Lemma}(b) and a case analysis.
 \end{proof}

 \begin{lemma}\label{lem: path edge weight}
 Consider a node $v_j$ in $T$ for $j\leq n-1$. Suppose $v_j$ and all its ancestors of $v_j$ are life-sustaining. Then the following hold.
 \begin{enumerate}[(a)]
 \item If the $j$th letter of $\can{T}$ is a $U$, then {$v_{j+1}$ has a left edge which has} the same weight and orientation as a $U$ step in a lattice path labeled $v_j^\ell v_{j+1}^\ell$.
 \item If the $j$th letter of $\can{T}$ is a $D$, then {$v_j$ has a right edge which has} the same weight and orientation as a $D$ step in a lattice path labeled $v_ j^\ell v_{j+1}^\ell$.
 \end{enumerate}
 \end{lemma}

 \begin{proof}
 If the $j$th letter of $\can{T}$ is a $U$, then by Lemma~\ref{lem: steps}(a), either $v_j$ is the left child of $v_{j+1}$, or $v_j$ is in the attic and $v_{j+1} = \da(v_j)$. In the first case, simply observe that the left edge $v_jv_{j+1}$ has the same weight and orientation as a $U$ step labeled $v_j^\ell v_{j+1}^\ell$. In the second case, apply Lemma~\ref{lem: Technical Lemma}(b) where $k=j+1$ and $f$ is the left edge of $v_{j+1}$. Then we have
 \begin{align}
 \wt{f} = \begin{cases} \la & \text{if }v_j^\ell \leq v_{j+1}^\ell,\\ \ld & \text{if }v_j^\ell > v_{j+1}^\ell,\end{cases}
 \end{align}
 which is exactly the weight of a $U$ step labeled with $v_j^\ell v_{j+1}^\ell$. This concludes the proof of part (a).
 Part (b) follows from a similar application of Lemma~\ref{lem: steps}(b) and Lemma~\ref{lem: Technical Lemma}(a) using the edge labeled $e$.
 \end{proof}

 \subsection{Partition of a tree}\label{subsec: partition}
 Throughout this subsection, let us continue to fix a labeled tree $T\in \pbtl_n$. Let the nodes of $T$ be $v_1,v_2,\dots,v_n$ listed in inorder.
 Let us define some notation which we use throughout the rest of Section~\ref{sec: Proof of Main Theorem}.
 \begin{definition}\label{def: Tree Decomp}
 Let $T\in \pbtun$, and let $\mathcal{A}=\{v_{\nodetwo_1},v_{\nodetwo_2},\dots,v_{\nodetwo_{m}}\}$ be an antichain under the partial ordering $\leq_T$ of nodes in the attic such that $j_1<j_2<\cdots <j_{m}$. For each $1\leq h\leq m$, define the following notation.
 \begin{itemize}
 \item Let $T_h$ be the subtree of $T$ consisting of $v_{\nodetwo_h}$ and all of its descendants.
 \item Let $(v_{\nodeone_h},v_{\nodetwo_h},v_{\nodethree_h})$ be the distinguished triple of the node $v_{j_h}$.
 \item Let $e_h$ be the right edge of $v_{\nodeone_h}$, and let $f_h$ be the left edge of $v_{\nodethree_h}$.
 \item Let $k_0 \coloneq 1$ and $i_{m+1} \coloneq n$.
 \item Let $w^{(h)}$ be the word on $\{U,D\}$ given by traversing the nodes in the interval $[v_{k_h},v_{i_{h+1}})$ of $T$ in inorder and recording a $D$ if the node has a right child, and a $U$ otherwise.
 \end{itemize}
 \end{definition}
 Observe that $\nodeone_h<\nodetwo_h<\nodethree_h\leq i_{h+1}$ for all $h$. {Indeed, let $v_m$ be the lowest common ancestor of $v_{j_h}$ and $v_{j_{h+1}}$. Then we have $j_h < m < j_{h+1}$ by the definition of inorder. Hence, by the definition of $k_h$ and $i_{h+1}$, we have $k_h\leq m \leq i_{h+1}$.} Therefore, we have partitioned the set $[n]$ into a disjoint union of intervals,
 \begin{align}\label{eq: partition of nodes}
 {\color{blue}[1,i_1]}\sqcup {\color{mygreen}[i_1+1,k_1-1]}\sqcup {\color{blue}[k_1,i_2]} \sqcup \cdots\sqcup {\color{blue}[k_{m-1},i_m]}\sqcup {\color{mygreen}[i_m+1,k_m-1]}\sqcup {\color{blue}[k_m,n]}.
 \end{align}
Further observe that $[i_h+1,k_h-1] = \{p \,\suchthat\, v_p\in \node{T_h}\}$ by \eqref{eq: Nodes of Tvj}. 

 \begin{figure}[H]
   \includegraphics[scale=0.5]{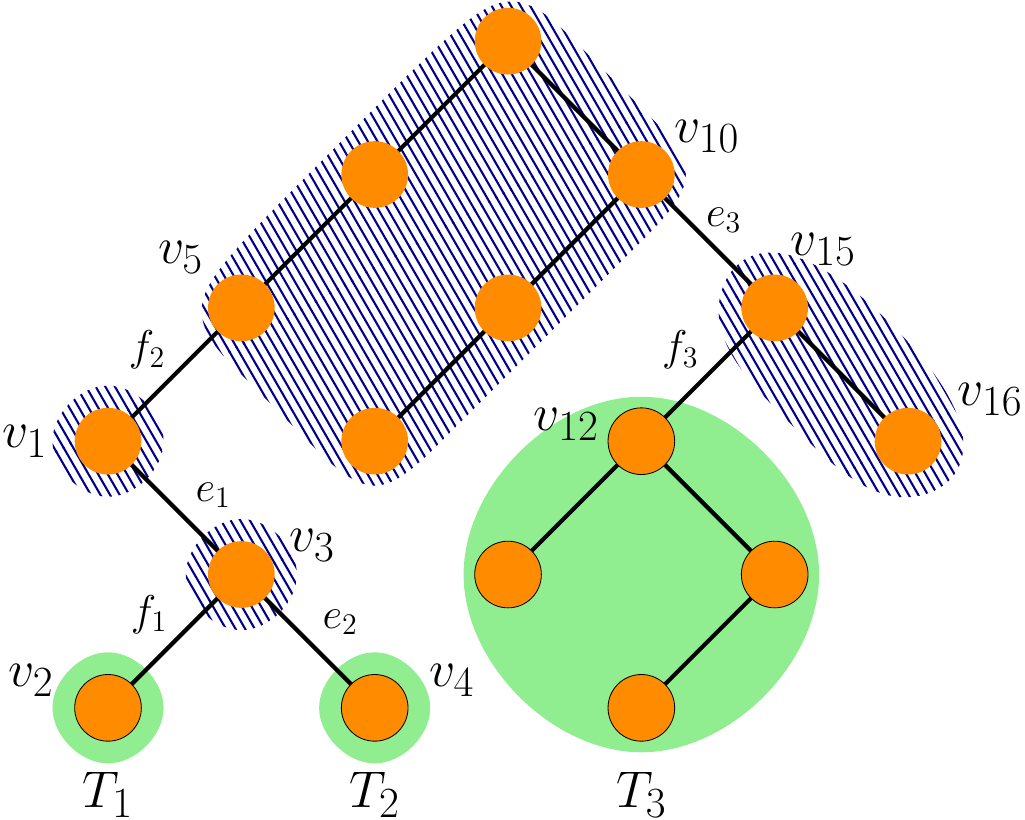}
   \caption{A tree $T\in \pbt_{16}$ with canopy $DUDUUUDUUDUDUUD$, partitioned using the antichain $\{v_2,v_4,v_{12}\}$.}
   \label{fig: Tree partition}
 \end{figure}

 \begin{lemma}\label{lem: partition}
 Let $T\in \pbtun$, and let $\mathcal{A}=\{v_{\nodetwo_1},v_{\nodetwo_2},\dots,v_{\nodetwo_{m}}\}$ be an antichain in the attic such that $j_1<j_2<\cdots <j_{m}$. Using the notation in Definition~\ref{def: Tree Decomp}, the canopy of $T$ is equal to the following concatenation
 \begin{equation}\label{eq: partition of canopy}
   \can{T} = {\color{blue}w^{(0)}}\,{\color{mygreen}(D\,\can{T_1}\,U)}\,{\color{blue}w^{(1)}}\cdots {\color{blue}w^{(m-1)}}\,{\color{mygreen}(D\,\can{T_{m}}\,U)}\,{\color{blue}w^{(m)}}.
 \end{equation}
 \end{lemma}

 \begin{proof}
 Recalling the definition of the canopy of a tree, Definition~\ref{def: canopy2}, we label each node in $\node{T}\setminus\{v_n\}$ with a $D$ if it has a right child, or with a $U$ if it does not have a right child. {Recall the partitioning \eqref{eq: partition of nodes} of the set $[n]$, which corresponds to a partition of the nodes of $T$, where the interval $[i_h+1,k_h-1]$ corresponds to the set of nodes of $T_h$. By Definition~\ref{def: Tree Decomp}, we have that $w^{(h)}$ is the word read from the labels of the nodes in $[v_{\nodethree_{h}},v_{\nodeone_{h+1}})$ in inorder.} For $h<m$, the node $v_{\nodeone_h}$ has a right edge, so it is labeled with a $D$. Let $t_h$ be the final node of $T_h$ in inorder. Then the word read from the labels of the nodes in $\node{T_h}\setminus\{t_h\}$ in inorder is exactly $\can{T_h}$. Since $t_h$ is the last node of $T_h$ in inorder, it does not have a right child, so it is labeled with a $U$.  By definition, $\can{T}$ is the word read from the labels of the nodes in $\node{T}\setminus\{v_n\}$ in inorder, which is exactly the right-hand side of \eqref{eq: partition of canopy}.
 \end{proof}

 Figure~\ref{fig: Tree partition} shows a binary tree $T$, together with a choice of antichain $\{v_{\nodetwo_1},v_{\nodetwo_2},v_{\nodetwo_3}\}$. Each blue striped region corresponds to an interval of nodes ${\color{blue}[v_{\nodethree_{h}},v_{\nodeone_{h+1}}]}$, and each solid green region corresponds to a subtree ${\color{mygreen}T_h}$. In this case, we have
 \begin{align*}
   \can{T} &= {\color{blue}()}{\color{mygreen}(DU)}{\color{blue}()}{\color{mygreen}(DU)}{\color{blue}(UUDUU)}{\color{mygreen}(DUDUU)}{\color{blue}(D)}\\
   &= DUDUUUDUUDUDUUD,
 \end{align*}
 where $\can{T_1}=\can{T_2}$ are the empty word and $\can{T_3} = U\!DU$. The reader can check that Lemma~\ref{lem: partition}(b) holds in this case.
 \begin{remark}
 Note that even though the subgraph formed by each interval $[v_{\nodethree_{h}},v_{\nodeone_{h+1}}]$ is connected in Figure~\ref{fig: Tree partition}, these subgraphs may not be connected in general.
 \end{remark}

 \subsection{Construction of the map \texorpdfstring{$\Phi_{n,\canopy}$}{phin}}\label{subsec: Construction of Phi}
 In this subsection, we define a map $\Phi_{n,\canopy}: \pbtl_{n,\canopy}\to \ptl_{n,\canopy}$ which extends the map $\Phi_{3,\canopy}$ defined in Example~\ref{ex: 3 Nodes}. We then prove that it is a weight-preserving bijection and use it to prove Theorem~\ref{thm: ribbon functional equation}.

 \begin{definition}\label{def: Phi}
   We define the map $\Phi_{n,\canopy} : \pbtl_{n,\canopy}\to \ptl_{n,\canopy}$ as follows. Given $T\in\pbtl_{n,\canopy}$, let $\mathcal{A} = \{v_{\nodetwo_1},v_{\nodetwo_2},\dots,v_{\nodetwo_m}\}$, ordered so that $\nodetwo_1 < \nodetwo_2 < \dots < \nodetwo_m$, be the subset of prunable nodes of $T$ which are maximal among prunable nodes under the partial order $\leq_T$. By definition, $\mathcal{A}$ is an antichain, and all of its elements are in the attic of $T$. Recall the notation defined in Definition~\ref{def: Tree Decomp}.

    Let $\canopy^{(h)}$ be the unique labeled lattice path whose steps are $w_{\canopy^{(h)}} = w^{(h)}$ and whose inorder reading word is $\inorder(\canopy^{(h)}) = v_{\nodethree_{h}}^\ell v_{\nodethree_h+1}^\ell \cdots v_{\nodeone_{h+1}}^\ell$. Define
   \[
     \Phi_{n,\canopy}(T) \coloneq (\canopy^{(0)},T_1,\canopy^{(1)},\dots,\canopy^{(m-1)},T_{m},\canopy^{(m)}),
   \]
   which lies in $\ptl_{n,\canopy}$ by Lemma~\ref{lem: partition}.
 \end{definition}

 Figure~\ref{fig: Infected Tree} shows a tree $T$ with its prunable nodes highlighted in green. 
 In this labeled tree, we have $\mathcal{A}=\{v_{2}, v_{4}, v_{12}\}$. 
 Figure~\ref{fig: Phi Decomp} shows the image of  $T$ under  $\Phi_{n,\canopy}$.
 The reader can check that in this instance, 
 \begin{align*}
 \wt{T} &= \la^3\ld^6\ra^3\rd^3= \la\ld^3\ra\rd(\la\ra\,\wt{T_1})(\la\ra\,\wt{T_2})(\ld\rd\,\wt{T_3})= \wt{\Phi_{n,\canopy}(T)},\\ \inorder(T) &= 3235461765768131 = \inorder(\Phi_{n,\canopy}(T)).
 \end{align*}

 \begin{figure}[H]
   \includegraphics[scale=0.5]{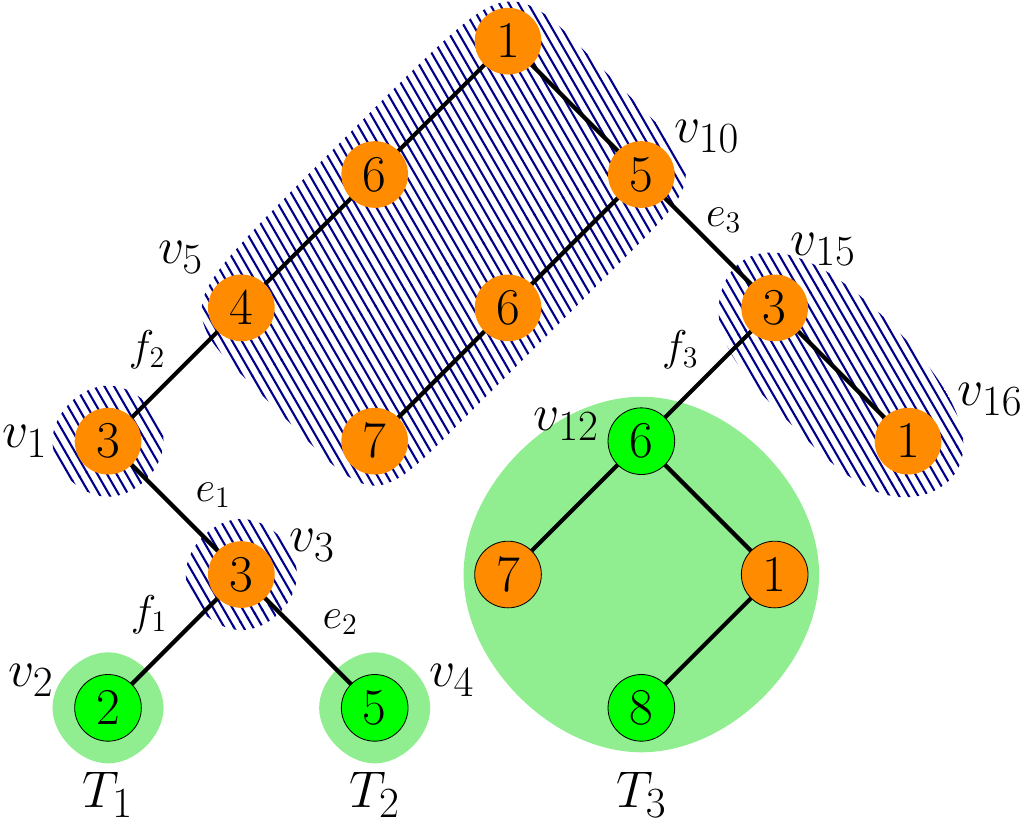}
   \caption{A tree $T\in \pbtl_{16}$ whose prunable nodes are highlighted in green with $\can{T} = DUDUUUDUUDUDUUD$.}
   \label{fig: Infected Tree}
 \end{figure}

   \begin{figure}[H]
     \includegraphics[scale=0.8]{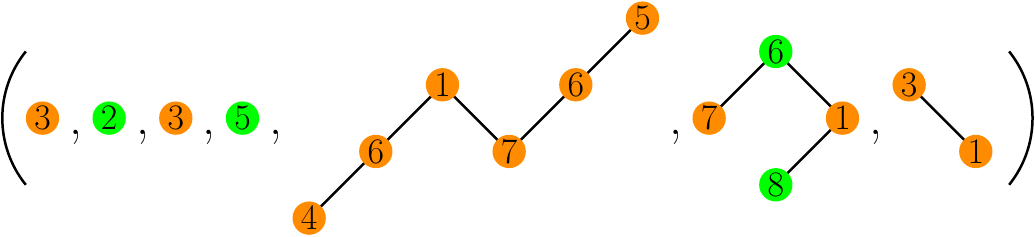}
     \caption{The image under $\Phi_{n,\canopy}$ of the tree in Figure~\ref{fig: Infected Tree}.}
     \label{fig: Phi Decomp}
   \end{figure}
 \begin{lemma}\label{lem: Weight-Preserving}
   For all $T\in\pbtl_{n,\canopy}$, we have that $\wt{T} = \wt{\Phi_{n,\canopy}(T)}$ and $\inorder(T)=\inorder(\Phi_{n,\canopy}(T))$.
 \end{lemma}
 \begin{proof}
   Definition~\ref{def: Phi} states that $\inorder(\canopy^{(h)}) = v_{\nodethree_h}^\ell v_{\nodethree_h+1}^\ell\cdots v_{\nodeone_{h+1}}^\ell$. Therefore, the fact that $\inorder(T) = \inorder(\Phi_{n,\canopy}(T))$ follows from the partition of $[n]$ in \eqref{eq: partition of nodes}.

  Recall from \eqref{eq: S weight} that
   \begin{equation}\label{eq: Phi weight}
     \wt{\Phi_{n,\canopy}(T)} = \wt{\canopy^{(0)}} \prod_{h=1}^m \gamma_h\,\wt{T_h}\,\wt{\canopy^{(h)}},
   \end{equation}
   where
   \begin{align}
     \gamma_h &= \begin{cases}
     \la \ra & \text{ if }\vend(\canopy^{(h-1)})^\ell \leq \vstart(\canopy^{(h)})^\ell,\\
     \ld\rd & \text{ if }\vend(\canopy^{(h-1)})^\ell > \vstart(\canopy^{(h)})^\ell.
   \end{cases}
 \end{align}
 Observe that $v_{\nodeone_h}^\ell = \vend(\canopy^{(h-1)})^\ell$ and $v_{\nodethree_h}^\ell = \vstart(\canopy^{(h)})^\ell$. Since $v_{\nodetwo_h}$ is maximal among prunable nodes in $T$, then all ancestors of $v_{\nodetwo_h}$ are life-sustaining. Applying Lemma~\ref{lem: infection lemma} with $i = i_h$, $j=j_h$, $k=k_h$, $e=e_h$ and $f=f_h$, we have
 \begin{align}
 \wt{e_h}\wt{f_h} &= \begin{cases}
 \la\ra &\text{if }v_{\nodeone_h}^\ell \leq v_{\nodethree_h}^\ell,\\
 \ld\rd &\text{if }v_{\nodeone_h}^\ell > v_{\nodethree_h}^\ell,
 \end{cases}\\
 &= \begin{cases}
     \la \ra & \text{ if }\vend(\canopy^{(h-1)})^\ell \leq \vstart(\canopy^{(h)})^\ell,\\
     \ld\rd & \text{ if }\vend(\canopy^{(h-1)})^\ell > \vstart(\canopy^{(h)})^\ell,
   \end{cases}\\
   &= \gamma_h.
 \end{align}

  Let us define a map on the disjoint union of the sets of edges of the paths $\canopy^{(h)}$,
  \begin{align}
  \Lambda : \bigsqcup_{h=0}^m \edge{\canopy^{(h)}}\to \edge{T}\setminus\left(\bigcup_{h=1}^{m}\{e_h,f_h\}\cup \edge{T_h})\right)
  \end{align}
 with the property that $\wt{e} = \wt{\Lambda(e)}$.
   Given $e\in\edge{\canopy^{(h)}}$, let its endpoints be the $j$th and $(j+1)$th nodes of $S$ in inorder. First, suppose that $e$ is a $U$ step. Since $\canopy^{(h)}$ is defined so that $w_{\canopy^{(h)}} = w^{(h)}$, then the $j$th letter of $\can{T}$ is a $U$. It follows from Lemma~\ref{lem: steps}(a) that $v_{j+1}$ has a left edge, which we then define to be $\Lambda(e)$.
   Similarly, if $e$ is a $D$ step, then the $j$th letter of $\can{T}$ must be a $D$. It follows from Lemma~\ref{lem: steps}(b) that $v_j$ must have a right edge, which we then define to be $\Lambda(e)$.
   By construction, we have that $\Lambda$ is a bijection.

    In the case that $e$ is a $U$ step, then by Lemma~\ref{lem: path edge weight}(a) we have $\wt{e} = \wt{\Lambda(e)}$.  In the case that $e$ is a $D$ step, then by Lemma~\ref{lem: path edge weight}(b) we have $\wt{e} = \wt{\Lambda(e)}$. Therefore, $\Lambda$ is a weight-preserving bijection.
  Hence,
  \begin{align}
     \wt{\Phi_{n,\canopy}(T)} &= \prod_{h=0}^m\wt{\canopy^{(h)}} \prod_{h=1}^m \gamma_h\wt{T_h}\\
     &= \prod_{h=0}^m \wt{\canopy^{(h)}} \prod_{h=1}^m \wt{e_h}\wt{f_h} \wt{T_h}\\
     &= \wt{T}.\qedhere
   \end{align}
  \end{proof}
 Next, we prove that the map $\Phi_{n,\canopy}$ is a bijection by constructing an inverse map. Our inverse map to $\Phi_{n,\canopy}$ involves an algorithm inspired by the Push-Gliding algorithm of Pr\'eville-Ratelle and Viennot~\cite{Preville-Ratelle-Viennot}, which gives a bijection between binary trees and certain pairs of lattice paths.

 Define $\Psi_{n,\canopy} : \ptl_{n,\canopy}\to\pbtl_{n,\canopy}$ as follows.
   Let
   \[
   S = (\canopy^{(0)},T_1,\canopy^{(1)},\dots,\canopy^{(m-1)},T_{m},\canopy^{(m)})
   \]
   be an element of $\ptl_{n,\canopy}$. Define $\hat{\canopy}$ to be the unique labeled lattice path such that
   \begin{equation}\label{eq: starting canopy}
     w_{\hat{\canopy}} = w_{\canopy^{(0)}}\,DU\,w_{\canopy^{(1)}}\,DU\cdots DU\,w_{\canopy^{(m)}}
   \end{equation}
   and
   \begin{equation}
     \inorder(\hat{\canopy}) = \inorder(\canopy^{(0)})\,\rt(T_1)^\ell\,\inorder(\canopy^{(1)})\cdots \inorder(\canopy^{(m-1)}) \rt(T_m)^\ell \inorder(\canopy^{(m)}).
   \end{equation}
   For easy comparison with~\cite{Preville-Ratelle-Viennot}, we switch momentarily to drawing our lattice paths with North and East steps instead of $U$ and $D$ steps, respectively. Similarly, we draw our trees so that right edges are horizontal and left edges are vertical.

  Recall from Pr\'eville-Ratelle and Viennot~\cite{Preville-Ratelle-Viennot} that an $\Omega$-tree $T$ is defined to be a lattice path with some binary trees rooted at some nodes of the path. Define the \bemph{inorder traversal} of $T$ by traversing the lattice path Southwest to Northeast, except that whenever a node is reached which is the root of a binary tree, traverse the binary tree in inorder and continue on with the rest of the path.
  If the highest node of $T$ is at height $y=h$, let $R$ denote the subset of points $(x,y)$ on the integer lattice strictly above the lattice path with $x\geq 0$ and $y\leq h$.

 Iteratively construct a tree from $\hat{\canopy}$ and the trees $T_h$ as follows.  Initialize the $\Omega$-tree $T$ to be the path $\hat{\canopy}$ together with each tree $T_h$ rooted at the node of $\canopy$ labeled $v_{\nodetwo_h}^\ell$.
 Identify the $j$th node of $T$ in inorder with the $j$th node of $S$ in inorder. Both nodes will simultaneously be denoted by $v_j$.
 Draw $T$ in the plane with the first node inorder situated at the origin $(0,0)$ in Cartesian coordinates.
 For example, if $S$ is the element of $\ptl$ in Figure~\ref{fig: Phi Decomp}, then the initial $\Omega$-tree $T$ and set of lattice points $R$ are shown in Figure~\ref{fig: Omega-Tree}.
 At any step, we keep track of an $\Omega$-tree $T$, all of whose nodes are labeled.
  Suppose $\rt(T_h) = v_{j_h}$ for each $h$.

  \begin{figure}
   \includegraphics[scale=0.8]{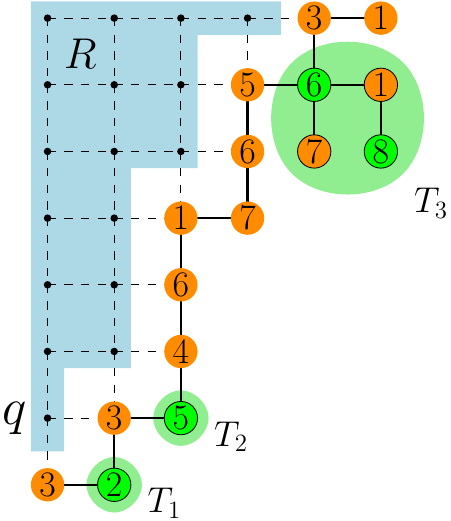}
   \caption{The $\Omega$-tree corresponding to the path-tree sequence in Figure~\ref{fig: Phi Decomp}.\label{fig: Omega-Tree}}
 \end{figure}

 If $R$ is nonempty, let $q$ be the Eastern-most point of the bottom row of $R$. Let $v_\nodeone$, $v_\nodetwo$ and $v_\nodethree$ be the nodes of the path which are immediately South, Southeast, and East of $q$, respectively.
 We say the node $v_j$ of $T$ is \bemph{prunable} if $v_j = \rt(T_h)$ for some $T_h$ and $v_j$ is \bemph{life-sustaining} otherwise.

 \begin{itemize}
 \item[(C1)] If $v_\nodetwo$ is prunable and $\std(v_\nodeone^\ell v_\nodetwo^\ell v_\nodethree^\ell)\in \{132,312,321\}$, or if $v_\nodetwo$ is life-sustaining and ${\std(v_\nodeone^\ell v_\nodetwo^\ell v_\nodethree^\ell)\in \{123,213,231\}}$, then delete the edge $v_\nodetwo v_\nodethree$, shift the part of $T$ which is weakly Northeast of the node $v_\nodethree$ $1$ unit West, and draw a new edge between $v_\nodeone$ and the new location of $v_\nodethree$. Update the $\Omega$-tree $T$ accordingly.

 \item[(C2)] If $v_\nodetwo$ is prunable and $\std(v_\nodeone^\ell v_\nodetwo^\ell v_\nodethree^\ell)\in \{123,213,231\}$, or if $v_\nodetwo$ is life-sustaining and $\std(v_\nodeone^\ell v_\nodetwo^\ell v_\nodethree^\ell)\in \{132,312,321\}$, then delete the edge $v_\nodeone v_\nodetwo$, shift the part of the $\Omega$-tree which is weakly Northeast of the node $v_\nodetwo$ $1$ unit South, and create a new edge between $v_\nodeone$ and the new location of $v_\nodethree$. Update the $\Omega$-tree $T$ accordingly.
 \end{itemize}

 \begin{figure}[t]
   \includegraphics[scale=0.6]{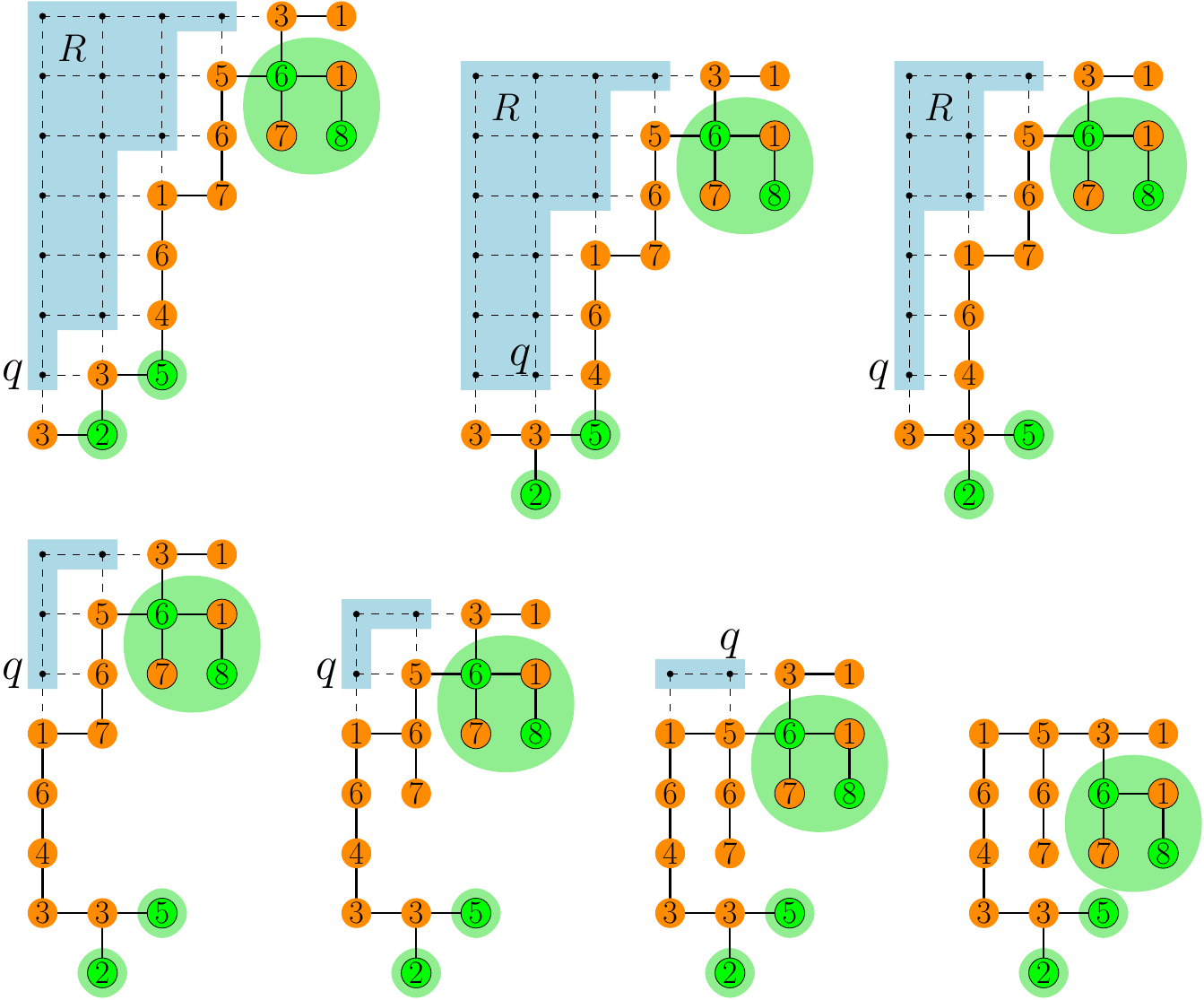}
   \caption{The inverse bijection $\Psi_{n,\canopy}$ applied to the path-tree sequence $S$ from Figure~\ref{fig: Phi Decomp}.}
   \label{fig: Inverse Bijection}
 \end{figure}

 Repeat this procedure until the set $R$ is empty. At the end of this algorithm, we obtain a labeled plane binary tree $T$ rooted at the Northwestern-most node.
 
 \begin{remark}
 This procedure of converting $S$ into $\Psi_{n,\canopy}(S)$ is essentially a modified version of the Push-Gliding algorithm of Pr\'eville-Ratelle and Viennot \cite{Preville-Ratelle-Viennot}, where translating part of the path West by 1 unit is a Glide operation and translating part of the path South by 1 unit is part of a Push operation.
 \end{remark}

{Define the canopy of an $\Omega$-tree to be the word on $\{U,D\}$ obtained by recording the $U$ and $D$ steps of $\hat{\canopy}$ from left to right, except that whenever a node is reached which is the root of a binary tree, we record the canopy of that binary tree and continue on with the rest of the path.}
 By~\eqref{eq: starting canopy}, the canopy of the initial $\Omega$-tree is $\can{S}$. From~\cite{Preville-Ratelle-Viennot}, Push and Glide operations preserve the canopy of the $\Omega$-tree $T$.
 Define $\Psi_{n,\canopy}(S) \coloneq T$, which is in $\ptl_{n,\canopy}$.

 Figure~\ref{fig: Inverse Bijection} shows the construction of $\Psi_{n,\canopy}(S)$ where $S = \Phi_{n,\canopy}(T)$ shown in Figure~\ref{fig: Phi Decomp} and $T$ is the labeled tree in Figure~\ref{fig: Infected Tree}. Note that we do indeed reconstruct the original tree in this example.

 Note that throughout this procedure, we need to avoid creating cycles when drawing $T$. Therefore, some of the edge lengths may have to be drawn lengthened in order to avoid collisions in the grid (see the footnote at the bottom of~\cite[page 5230]{Preville-Ratelle-Viennot}).

 \begin{lemma}\label{lem: Bijection}
 The map $\Phi_{n,\canopy}$ is a bijection.
 \end{lemma}

 \begin{proof}
 The choice of either applying (C1) or (C2) at each step guarantees that the node $v_{\nodetwo_h}$ of the tree $\Psi_{n,\canopy}(S)$ is prunable and that it is maximal among prunable nodes in $\Psi_{n,\canopy}(S)$. Therefore, the subtrees in the path-tree sequence $\Phi_{n,\canopy}(\Psi_{n,\canopy}(S))$ will be exactly $T_1$,\dots,$T_{m}$.

 From~\cite{Preville-Ratelle-Viennot}, Push and Glide operations preserve the canopy and inorder reading word of an $\Omega$-tree. Therefore, the sequences of $U$ and $D$ steps in the $m+1$ lattice paths of $\Phi_{n,\canopy}(\Psi_{n,\canopy}(S))$ must be $w_{\canopy^{(0)}},w_{\canopy^{(1)}},\dots,w_{\canopy^{(m)}}$, respectively. It is easy to see that their inorder reading words are $\inorder(\canopy^{(0)}),\inorder(\canopy^{(1)}),\dots,\inorder(\canopy^{(m)})$, respectively, so the lattice paths of $\Phi_{n,\canopy}(\Psi_{n,\canopy}(S))$ must be $\canopy^{(0)},\dots,\canopy^{(m)}$. Therefore, $\Phi_{n,\canopy}(\Psi_{n,\canopy}(S))=S$ so $\Phi_{n,\canopy}$ is surjective.

{
 Given an integer $n$, a word $w\in \bP^n$, and a sequence $\canopy\in \{U,D\}^{n-1}$, let $\pbtl_{n,\canopy,w}$ be the subset of $\pbtl_{n,\canopy}$ of trees with inorder reading word $w$, and let $\ptl_{n,\canopy,w}$ be the subset of $\ptl_{n,\canopy}$ with inorder reading word $w$. Observe that $\Phi_{n,\canopy}$ maps $\pbtl_{n,\canopy,w}$ surjectively onto $\ptl_{n,\canopy,w}$ by the argument above, so $|\pbtl_{n,\canopy,w}| \geq |\ptl_{n,\canopy,w}|$. Furthermore, observe that $\pbtl_{n,\canopy,w}$ is in bijection with the set of unlabeled trees on $n$ nodes with canopy $\canopy$. Similarly, the set $\ptl_{n,\canopy,w}$ is in bijection with unlabeled path-tree sequences on $n$ nodes with canopy $\canopy$. 
 }

{
 A simple generating function argument shows that the number of unlabeled path-tree sequences on $n$ nodes is $|\pbt_n| = \Cat{n}$, the $n$th Catalan number. 
 Indeed, the ordinary generating function for unlabeled path-tree sequences is given by 
 \begin{align}\label{eqn:ogf_unlabeled_path_tree}
 \sum_{i\geq 0}\left(\frac{x}{1-2x}\right)^{i+1}C(x)^i=\frac{x}{1-2x-xC(x)}.
 \end{align}
Recall from Subsection~\ref{subsec: Semiorders} that $C(x)=\sum_{n\geq 1}\Cat{n}x^n$ and that $C\coloneqq C(x)$ satisfies the relation $C=x(1+C)^2$. 
This relation implies in turn that the right hand side of \eqref{eqn:ogf_unlabeled_path_tree} equals $C$. Therefore, for fixed $n$ and $w$ we have
 \begin{align}\label{eq:StringOfInequalities}
\Cat{n} = |\pbt_{n}| = \sum_{\canopy} |\pbtl_{n,\canopy,w}| \geq \sum_{\canopy} |\ptl_{n,\canopy,w}| = \Cat{n}.
 \end{align}
 Hence, the inequality in \eqref{eq:StringOfInequalities} must be an equality, and in fact each inequality $|\pbtl_{n,\canopy,w}| \geq |\ptl_{n,\canopy,w}|$ must be an equality, so $|\pbtl_{n,\canopy,w}| = |\ptl_{n,\canopy,w}|$.
 }
 
 {
 Since $\Phi_{n,\canopy}$ maps $\pbtl_{n,\canopy,w}$ surjectively onto $\ptl_{n,\canopy,w}$, and these two sets have the same cardinality, then $\Phi_{n,\canopy}$ must restrict to a bijection between $\pbtl_{n,\canopy,w}$ and $\ptl_{n,\canopy,w}$. As we run over $w$, the sets $\pbtl_{n,\canopy,w}$ and $\ptl_{n,\canopy,w}$ partition the domain and codomain, respectively, of the map $\Phi_{n,\canopy}$. Therefore, $\Phi_{n,\canopy}$ is a bijection.
 }
 \end{proof}

 \begin{proof}[Proof 2 of Theorem~\ref{thm: ribbon functional equation}]
 By Lemmas~\ref{lem: Weight-Preserving} and~\ref{lem: Bijection}, we have that $\Phi_{n,\canopy} : \pbtl_{n,\canopy}\to\ptl_{n,\canopy}$ is a bijection such that $\wt{T} = \wt{\Phi_{n,\canopy}(T)}$ and $\inorder(T)=\inorder(\Phi_{n,\canopy}(T))$. Therefore,
 \begin{align}
 \sum_{T\in\pbtl} \wt{T}\alpx^T &= \sum_{n\geq 1}\sum_{\canopy\in \{U,D\}^{n-1}} \sum_{T\in \pbtl_{n,\canopy}}\wt{\Phi_{n,\canopy}(T)}\alpx^{\Phi_{n,\canopy}(T)}\\
 &= \sum_{S\in \ptl} \wt{S}\,\alpx^S\\
 &=\sum_{n\geq 1}\sum_{\alpha\vDash n}(\la\ra\, G+\la+\ra)^{n-\ell(\alpha)}(\ld\rd\, G+\ld+\rd)^{\ell(\alpha)-1}\,r_\alpha,
 \end{align}
 where the last equality holds by Lemma~\ref{lem: PT equation}.
 \end{proof}
 In addition to proving Theorem~\ref{thm: ribbon functional equation}, the same proof allows us to obtain a noncommutative version of Theorem~\ref{thm: ribbon functional equation}.

 \sout{\begin{remark}
 \emph{Note that our bijection $\Phi_{n,\canopy}$ has an interesting unlabeled analogue. If we apply our bijection above to labeled trees on $n$ nodes with a fixed inorder reading word, say the identity permutation in $\mathfrak{S}_n$, then we get a bijection between $\pbtun$ and $\mathscr{P\!T}_n$, the set of sequences of interlacing paths and trees where both are unlabeled. Of course, one can establish that $|\mathscr{P\!T}_n|=\Cat{n}$ by using an argument relying on generating functions.}
 \end{remark}}

 \subsection{Noncommutative version}
 Let $\mathbf{x} = \{\mbxs{1},\mbxs{2},\dots\}$ be a set of noncommuting variables. Given a word $w = w_1\dots w_n\in \bP^n$, define the noncommutative monomial $\mbx{w}$ to be $\mbxs{w_1}\dots\mbxs{w_n}$.
 Let
 \[
 \mathbf{G} \coloneq \mathbf{G}(\mathbf{x};\la,\ld,\ra,\rd) = \sum_{n\geq 1} \sum_{T\in \pbtl_n} \wt{T} \mbx{\inorder(T)}.
 \]
 Let $\alps$ and $\alpt$ be parameters which do not commute with the $\mbxs{i}$.
 For $w\in \bP^+$, let $\mbx{w}(\alps,\alpt)$ be the \emph{extended noncommutative monomial} given by inserting $\alps$ between $\mbxs{w_i}$ and $\mbxs{w_{i+1}}$ in the monomial $\mbx{w}$ for each ascent $w_i\leq w_{i+1}$ and inserting $\alpt$ for each  descent $w_i>w_{i+1}$.
 For instance, if $w=21131$,  then $\mbx{w}(\alps,\alpt)=\mbxs{2}\alpt\mbxs{1}\alps\mbxs{1}\alps\mbxs{3}\alpt\mbxs{1}$.

 As a corollary of Proof 2 of Theorem~\ref{thm: ribbon functional equation}, we obtain the following noncommutative analogue of the functional equation in Theorem~\ref{thm: ribbon functional equation}.

 \begin{corollary}\label{cor: noncommutative inorder}
 We have the following functional equation in terms of substitutions into the extended noncommutative monomials,
 \[
     \mathbf{G} = \sum_{w\in \bP^+} \mbx{w}(\la\ra\,\mathbf{G}+\la+\ra,\ld\rd\,\mathbf{G}+\ld+\rd).
 \]
 \end{corollary}


 \section{Proofs of Theorem~\ref{cor: Ribbon Expansion}~and Theorem~\ref{thm: ribbon expansion for fixed canopy}}
 \label{sec: ribbon expansion}

In this section, we use the weight-preserving bijection defined in Definition~\ref{def: Phi} to prove Theorem~\ref{thm: ribbon expansion for fixed canopy}. 
We then show how Theorem~\ref{cor: Ribbon Expansion} follows from Theorem~\ref{thm: ribbon expansion for fixed canopy}.

{
Our key tool in this section is a weight-preserving surjective map $\Theta_{n,\canopy}$ from $\ptl_{n,\canopy}$ onto $\anoncrossing{n}{\canopy}$ with the following crucial property: The preimage of $\pi^{\ast}\in \anoncrossing{n}{\canopy}$ consists of certain $S\in \ptl_{n,\canopy}$ such that the sum over the monomials $\alpx^{S}$ is a product of ribbon Schur functions.
The map $\Theta_{n,\canopy}$ is defined recursively. 
Informally put, all nodes that belong to paths in a path-tree sequence $S$ contribute the block $B_1$ in our augmented interlacing partition $\pi^{\ast}$. 
Furthermore, individual paths in $S$ correspond to sub-blocks in $B_1$ that are maximal under connectedness by short arcs. 
The labels of the short arcs record the sequences of up and down steps in individual paths.
To determine the remaining blocks in $\pi^{\ast}$, we first apply $\Phi_{n',\canopy'}$ for the appropriate $n'$ and $\canopy'$ to each tree in $S$, and subsequently apply $\Theta_{n',\canopy'}$ to the resulting path-tree sequence. In this manner we recursively compute all blocks in $\pi^{\ast}$.
Finally, in view of the definitions of weights of elements of $\ptl_{n,\canopy}$ and $\anoncrossing{n}{\canopy}$, the choice of which nodes get marked is essentially `forced' by our demand that $\Theta_{n,\canopy}$ be weight-preserving.
With this informal description serving as a crutch, we  give precise details.
}

 Given $S = (\canopy^{(0)},T_1,\dots,T_m,\canopy^{(m)}) \in \ptl_{n,\canopy}$, let $v_1,\dots,v_n$ be the nodes of $S$ listed in inorder. For $0\leq h\leq m$, let $[k_h,i_{h+1}]$ be the set of indices of nodes in $\canopy^{(h)}$ and $[i_h+1,k_h-1]$ be the set of indices of nodes in $T_h$. 

 Define a map $\Theta_{n,\canopy}:\ptl_{n,\canopy}\to \anoncrossing{n}{\canopy}$ recursively as follows.
 If $n=1$, define $\Theta_{1,\emptyset}$ to be the map which sends each sequence $(\canopy^{(0)})$ consisting of a single labeled node to the unique augmented interlacing partition on a single node. Now, given $n>1$ and $S\in \ptl_{n,\canopy}$, define the arc diagram of an augmented interlacing partition as follows.
 \begin{enumerate}
 \item Initialize a set of nodes $u_1,\dots,u_n$, and draw arcs so that all of the nodes $u_p$ with $p\in [1,i_1]\sqcup [k_1,i_{2}]\sqcup\dots\sqcup [k_{m},n]$ are in the single block $B_1$.
 \item For $0\leq h\leq m$ and $p\in[\nodethree_h,\nodeone_{h+1})$, observe that $v_pv_{p+1}$ is an edge $e$ of $\canopy^{(h)}$. If $e$ is an up step, then label the short arc $u_p u_{p+1}$ drawn in step (1) with a $U$. Otherwise, label the short arc $u_pu_{p+1}$ with a $D$. In addition, if $v_p^\ell>v_{p+1}^\ell$, then mark the node $u_p$.
 \item For $1\leq h\leq m$, mark the node $u_{\nodeone_h}$ if and only if $v_{\nodeone_h}^\ell > v_{\nodethree_h}^\ell$.
 \item For $1\leq h \leq m$, let $n_h=|\node{T_h}|$ and $\nu_h = \can{T_h}$. Recursively construct the augmented interlacing partition $\pi^\ast_h \coloneq \Theta_{n_h,\nu_h}(\Phi_{n_h,\nu_h}(T_h))$ on the set of nodes $u_{\nodeone_h+1},u_{\nodeone_h+2},\dots,u_{\nodethree_h-1}$.
 \item Define $\Theta_{n,\canopy}({S}) \coloneq B_1/\pi^\ast_1/\cdots/\pi^\ast_{m}$, i.e.\ the augmented interlacing partition whose blocks are $B_1$ together with the blocks in $\pi^\ast_{1}$,\dots,$\pi^\ast_{m}$.
 \end{enumerate}
  For the path-tree sequence $S$ in Figure~\ref{fig: Phi Decomp}, its image $\Theta_{n,\canopy}(S)$ is shown in Figure~\ref{fig: Theta Image}.
  
   \begin{figure}[t]
   \includegraphics[scale=0.35]{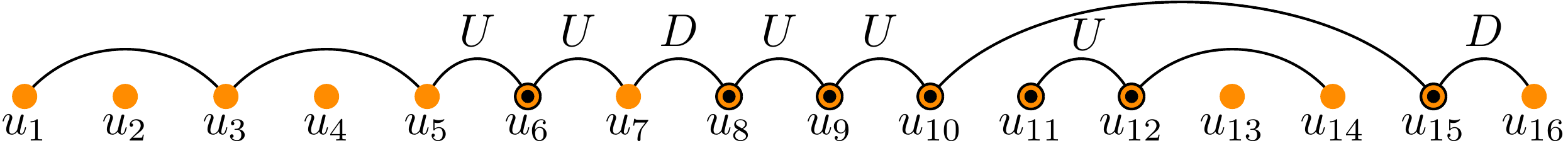}
   \caption{The augmented interlacing partition {$\Theta_{n,\canopy}(S)$ for $S$ in Figure~\ref{fig: Phi Decomp}.}}
   \label{fig: Theta Image}
 \end{figure}

  \begin{lemma}\label{lem: Theta Canopy lemma}
{ We have that $\Theta_{n,\canopy}$ is a well-defined map to $\anoncrossing{n}{\canopy}$ which is weight-preserving and surjective.}
  \end{lemma}

  \begin{proof}
We prove that $\Theta_{n,\canopy}$ is well-defined, weight-preserving, and surjective by induction on $n$. The $n=1$ case of each part is trivial, because in this case the only element of $\ptl_{n,\canopy}$ is a path-tree sequence $(\canopy^{(0)})$ consisting of a single node, which maps to the augmented interlacing partition consisting of a single node. Let us assume by way of induction that the lemma holds for all $n<N$ and all $\canopy$ and $S$. 

\indent To show $\Theta_{N,\canopy}$ is well-defined, we must show $w(\Theta_{N,\canopy}(S)) = \can{S}$ for all $S = (\canopy^{(0)},T_1,\dots,T_m,\canopy^{(m)}) \in \ptl_{N,\canopy}$. Let the nodes of $S$ be $v_1,\dots,v_N$ listed in inorder, let $n_h = |\node{T_h}|$, and let $\nu_h = \can{T_h}$. By our inductive hypothesis, we have 
\begin{align}
w(\Theta_{n_h,\nu_h}(\Phi_{n_h,\nu_h}(T_h))) = \can{\Phi_{n_h,\nu_h}(T_h)} = \can{T_h}.
\end{align} 
Letting $B_1$ be the block of $\Theta_{N,\canopy}(S)$, partition $B_1$ into blocks $C_0, C_1,\dots,C_p$ that are maximal under connectedness by short arcs. Then $w(C_i) = w_{\canopy^{(i)}}$ for $0\leq i\leq m$ by step (2) of the definition of $\Theta_{N,\canopy}$. Therefore, by the definition of $w(\Theta_{N,\canopy}(S))$, see \eqref{eq: def of w}, and the definition of $\can{S}$, see \eqref{eq: canopy of S}, we have $w(\Theta_{N,\canopy}(S)) = \can{S}$.

 To show $\Theta_{N,\canopy}$ is weight-preserving, we must show $\wt{\Theta_{N,\canopy}(S)} = \wt{S}$ for all $S\in \ptl_{N,\canopy}$. Using the same notation as above, by our inductive hypothesis we have
\begin{align}\label{eq: theta wt induction}
\wt{\Theta_{n_h,\nu_h}(\Phi_{n_h,\nu_h}(T_h))} = \wt{\Phi_{n_h,\nu_h}(T_h)} = \wt{T_h}.
\end{align}
Recall the definition of $\wt{\Theta_{n,\canopy}(S)}$, see \eqref{eq: wt of pi ast}. Observe that by step (2) of the construction of $\Theta_{N,\canopy}$, the number of unmarked steppers whose short arc is labeled $U$ in $C_i$ is equal to the number of up steps in $\canopy^{(i)}$ which form a descent. Similar equalities hold for the three other types of steppers. Furthermore, the contribution of the marked and unmarked jumpers to $\wt{\Theta_{n,\canopy}(S)}$ is equal to the contribution of the $\gamma_i$ to $\wt{S}$ by step (3). From these observations together with \eqref{eq: theta wt induction}, we conclude $\wt{\Theta_{N,\canopy}(S)} = \wt{S}$.

To show $\Theta_{N,\canopy}$ is surjective, let $\pi^\ast\in \anoncrossing{N}{\canopy}$ an augmented interlacing partition on nodes $u_1,\dots,u_N$. Letting $B_1$ be the part of $\pi^\ast$ containing $u_1$ and $u_N$, partition $B_1$ into blocks $C_0,\dots, C_m$ that are maximal under connectedness by short arcs, and let $\pi_h^\ast$ be the augmented interlacing partition induced by $\pi^\ast$ on the nodes $u_j$ for $j\in (\max(C_{h-1}),\min(C_{h}))$. By our inductive hypothesis, we have $\Theta_{n_h,\canopy_h}$ is surjective, where $n_h$ is the number of nodes of $\pi_h^\ast$ and $\canopy_h = w(\pi_h^\ast)$. Hence, there exists $S_h$ such that $\Theta_{n_h,\canopy_h}(S_h) = \pi_h^\ast$. For $0\leq i\leq m$, let $\canopy^{(i)}$ be the unlabeled lattice path on $|C_i|$ many nodes such that $w_{\canopy^{(i)}}$ is given by the labels of the short arcs in $C_i$. Furthermore, label the nodes of the $\canopy^{(i)}$ so that if we define $S = (\canopy^{(0)},\Phi_{n_1,\canopy_1}^{-1}(S_1),\dots,\Phi_{n_m,\canopy_m}^{-1}(S_m),\canopy^{(m)})$, then steps (2) and (3) would give us the marking of $\pi^\ast$ back under applying $\Theta_{n,\canopy}$. Finally, by step (4) we see that $S$ has been chosen exactly so that we recover the $\pi_h^\ast$. Hence, we have $\Theta_{N,\canopy}(S) = \pi^\ast$, so $\Theta_{N,\canopy}$ is surjective, which completes the induction.
  \end{proof}
  The reader can verify that for the path-tree sequence $S$ in Figure~\ref{fig: Phi Decomp} and $\Theta_{n,\canopy}(S)$ in  Figure~\ref{fig: Theta Image}, we have $\wt{\Theta_{n,\canopy}(S)} = \la\ld^4\ra\rd(\la\ra)^2(\ld\rd)^2 = \wt{S}$.
  
 \begin{proof}[Proof of Theorem~\ref{thm: ribbon expansion for fixed canopy}]
By Lemma~\ref{lem: Theta Canopy lemma}, we have
 \begin{align}
 G_{n,\canopy} & = \sum_{T\in \pbtl_{n,\canopy}} \wt{T}\,\alpx^T\\
 &= \sum_{S\in \ptl_{n,\canopy}} \wt{S} \,\alpx^{S}\label{eq: weight pres}\\
 &= \sum_{S \in \ptl_{n,\canopy}} \wt{\Theta_{n,\canopy}(S)} \,\alpx^{S}\\
 & = \sum_{\pi^\ast\in \anoncrossing{n}{\canopy}} \wt{\pi^\ast}\sum_{S\in \Theta_{n,\canopy}^{-1}(\pi^\ast)}\alpx^S,
 \end{align}
 where \eqref{eq: weight pres} follows from the fact that $\Phi_{n,\canopy}$ is a weight-preserving bijection.

 It remains to show that for an augmented interlacing partition $\pi^\ast= B_1/\dots/B_k\in \mnoncrossing{n,\canopy}$, we have
 \begin{align}\label{eq: ribbon comp}
 \sum_{S\in\Theta_{n,\canopy}^{-1}(\pi^\ast)} \alpx^S = r_{c(B_1)}r_{c(B_2)}\dots r_{c(B_k)}.
 \end{align}
 We prove this by induction on $n$. The case of $n=1$ is immediate since $G_{n,\emptyset} = r_1$.
 Let $n>1$, and let $\pi^\ast\in \anoncrossing{n}{\canopy}$. Letting $B_1$ be the block of $\pi^\ast$ containing $1$ and $n$, decompose $B_1$ into blocks $C_0,\dots,C_m$  maximal under connectedness by short arcs, as in \eqref{eq: def of w}. Let $\pi^\ast_1,\dots,\pi^\ast_m$ be the induced augmented interlacing partitions on the remaining nodes. Let $n_h$ be the number of nodes in $\pi^\ast_h$, and let $\canopy_h = w(\pi^\ast_h)$.
 Then $S = (\canopy^{(0)},T_1,\dots,T_{m},\canopy^{(m)})\in \Theta_{n,\canopy}^{-1}(\pi^\ast)$ if and only if
 \begin{enumerate}[(a)]
 \item The word obtained by concatenating the inorder reading words of $\canopy^{(0)},\dots,\canopy^{(m)}$ has strict descents in positions corresponding to marked nodes of $C_0,\dots,C_m$ and weak ascents in positions corresponding to unmarked nodes,
 \item For $0\leq h\leq m$, $w_{\canopy^{(h)}}$ is equal to the word read from the labeling of the short arcs connecting the nodes in $C_h$,
 \item For $1\leq h\leq m$, we have $\Phi_{n_h,\canopy_h}(T_h)\in \Theta_{n_h,\canopy_h}^{-1}(\pi^\ast_h)$.
 \end{enumerate}
 For fixed $\pi^\ast$, the steps of $\canopy^{(h)}$ are determined by (b). If $T_1,\dots,T_m$ are fixed trees which satisfy (c), then by (a) and the definition of $c(B_1)$ in Section~\ref{sec: fixed canopy}, we have that
 \begin{equation}\label{eq: sum to ribbon}
 \sum_{S} \alpx^{\inorder(\canopy^{(0)})}\alpx^{\inorder(\canopy^{(1)})}\dots\alpx^{\inorder(\canopy^{(m)})} = r_{c(B_1)},
 \end{equation}
 where the sum is over all $S = (\canopy^{(0)},T_1,\dots,T_m,\canopy^{(m)}) \in \Theta_{n,\canopy}^{-1}(\pi^\ast)$.

 Finally, by~\eqref{eq: sum to ribbon} and our induction hypothesis applied to each $\pi^\ast_h$, we have
 \begin{align}
 \sum_{S\in \Theta_{n,\canopy}^{-1}(\pi^\ast)} \alpx^{S} &= r_{c(B_1)}\prod_{h=1}^{m}\left(\sum_{T_h\in \Phi_{n_h,\canopy_h}^{-1}(\Theta_{n_h,\canopy_h}^{-1}(\pi^\ast_h))}\alpx^{T_h}\right)\\
 & = r_{c(B_1)}\prod_{h=1}^{m}\left(\sum_{S_h\in \Theta_{n_h,\canopy_h}^{-1}(\pi^\ast_h)}\alpx^{S_h}\right) \\
 &= r_{c(B_1)}r_{c(B_2)}\dots r_{c(B_k)},
 \end{align}
 which completes the proof.
 \end{proof}

 \begin{proof}[Proof of Theorem~\ref{cor: Ribbon Expansion}]
 We have that $G_n = \sum_{\canopy\in \{U,D\}^{n-1}} G_{n,\canopy}$. Furthermore, for each $\pi \in \mnoncrossing{n}$, then summing over all augmented interlacing partitions $\pi^\ast$ whose underlying element of $\mnoncrossing{n}$ is $\pi$, we get
 \begin{align}
 \sum_{\pi^\ast}\wt{\pi^\ast} &= \sum_{\pi^\ast} \la^{\sau(\pi^\ast)} \ld^{\sdu(\pi^\ast)} \ra^{\,\sad(\pi^\ast)} \rd^{\,\sdd(\pi^\ast)} (\la\ra)^{\,\ja(\pi^\ast)} (\ld\rd)^{\,\jd(\pi^\ast)}\\
 	&= (\la+\ra)^{\,\sa(\pi)} (\ld+\rd)^{\sd(\pi)} (\la\ra)^{\ja(\pi)} (\ld\rd)^{\jd(\pi)}\\
 	&= \wt{\pi}.
 \end{align}
 Then the result follows by summing both sides of the equation in Theorem~\ref{thm: ribbon expansion for fixed canopy} over all $\canopy\in \{U,D\}^{n-1}$.
 \end{proof}

 \section{Final remarks}\label{sec: back to the future}
 We conclude this article with some avenues to pursue and various remarks containing isolated results.

 \begin{enumerate}
 \item We briefly consider some special cases that were not covered in Section~\ref{sec: Hyperplane arrangements}. Our first specialization concerns the Shi arrangement $\hypshi_n$.
 Set $\la=\ld=\ra=1$ and $\rd=0$.
 This corresponds to considering  trees that do not have any right descents.
 If we restrict our attention to such trees on $n$ nodes such that the labeling is a standard labeling, then there are $(n+1)^{n-1}$ of them (see \cite{PakPostnikov} or \cite[p.~7]{Pak-slides}).
  Additionally, the functional equation in Theorem~\ref{thm: Gessel functional equation} is equivalent to
 \begin{align}
 (1+G)=\sum_{i\geq 0}h_i (1+G)^{i}.
 \end{align}
 Let $\mathrm{PF}_n$ be the Frobenius characteristic of the $\mathfrak{S}_n$-action on parking functions of length $n$ given by permuting coordinates. Solving this functional equation \cite[Section 4.1]{Haiman} implies that $G_n(\alpx;1,1,1,0)$ is the Frobenius characteristic of the natural $\mathfrak{S}_n$-action on parking functions of length $n$ (see also \cite[Proposition 2.2]{Stanley- Noncrossing}).
 Using one of the many bijections between $\regions{\hypshi_n}$ and parking functions of length $n$, we may obtain an $\mathfrak{S}_n$-action on $\regions{\hypshi_n}$.
 Thus, we can think of $G_n(\alpx;1,1,1,0)$ as the Frobenius characteristic of this action.
 
 We also consider the specialization $G=G(\alpx;1,1,q,0)$.
 Theorem~\ref{thm: Gessel functional equation} translates to
 \begin{align}
   1+qG=\sum_{i\geq 0}h_iq^{i}(1+G)^i.
 \end{align}
 Solving this functional equation yields
 \begin{align}\label{eqn:pf_n_with_q}
   G_n(\alpx;1,1,q,0)=\sum_{D\in\mathcal{D}_n}q^{n-\mathrm{peak}(D)}h_{\lambda(D)}.
 \end{align}
Recall that $\mathcal{D}_n$ is the set of Dyck paths with $2n$ steps, $\lambda(D)\vdash n$ is obtained by sorting the lengths of the vertical runs of $D\in \mathcal{D}_n$, and $\mathrm{peak}(D)$ is the number of peaks in $D$. The reader should compare the expansion in ~\eqref{eqn:pf_n_with_q} with \eqref{eqn:linial_zonal}.

 Another case of  interest  is when $\ld=\ra=1$ and $\la=\rd=0$. We have that $G_{n,v}(\alpx;0,1,1,0)$ is the generating function over increasing binary trees on $n$ nodes with canopy $\canopy$.
 We leave it to the reader to verify that Theorem~\ref{thm: ribbon expansion for fixed canopy} implies that $G_{n,\canopy}(\alpx;0,1,1,0)$ is a single ribbon Schur function $r_{\alpha}$ where $\alpha$ is the composition associated with the canopy $\canopy=\canopy_1\cdots \canopy_{n-1}$ define as follows. Consider the subset $S$ of $[n-1]$ where $i\in S$ if and only if $\canopy_i=U$, and define $\alpha\coloneq\comp(S)$. That the fact that there are $n!$ standard increasing binary trees on $n$ nodes is simply a reflection of the fact that $\sum_{\alpha\vDash n } r_{\alpha}$ is the Frobenius characteristic of the regular representation of $\mathfrak{S}_n$.

 Another specialization is when $\la=\ld=0$ and $\ra=\rd=1$. We have $G_n(\alpx;0,0,1,1)$ is the generating function over words of length $n$. It is well known that this is the Frobenius characteristic of the regular representation of $\mathfrak{S}_n$.
 It is also the Frobenius characteristic of the coinvariant algebra of the polynomial ring on $n$ generators, which is isomorphic to the regular representation as a symmetric group module.
 If we keep the $\rd$ parameter, this is equivalent to recording descents in the generating function over words.
 In \cite{Gnedin-Gorin-Kerov}, Gnedin-Gorin-Kerov construct a filtration of symmetric group submodules of the coinvariant algebra, defined in terms of its Garsia-Stanton descent basis \cite{Garsia-Stanton}.
 They show that $G_n(\alpx;0,0,1,q)$ is the Frobenius characteristic of the associated graded module corresponding to this filtration.
 For general results pertaining to multigraded Frobenius characteristics of the coinvariant algebra according the descents, the reader is referred to \cite{Huang}.
 It would be interesting to draw connections between this work and ours.

 \item We return to the topic of $\gamma$-nonnegativity. Note that our expansion for $B_n$ in \eqref{eqn:double gamma-nonnegativity} involves a sort of ``double" $\gamma$-nonnegativity: one in terms of $\la\ra$ and $\la+\ra$, and the other in terms of $\ld\rd$ and $\ld+\rd$.
 It would be interesting to identify other instances of such a phenomenon.
 There has been a great amount of interest in $\gamma$-nonnegativity recently, in no small measure due to Gal's conjecture \cite{Gal} which is essentially a statement in topology, and the implications it has for the Charney-Davis conjecture \cite{Charney-Davis}.
 For more work in this direction, see \cite{Postnikov-Reiner-Williams}.
 We do not know of any topological motivation for our ``double'' $\gamma$-nonnegativity, and we intend to explore this in the future.

 \item Another natural question to consider is that of extending our results to labeled $k$-ary trees for $k\geq 3$.
 Note that while there is an analogue of $B(x;\la,\ld,\ra,\rd)$ for $k$-ary trees where $k\geq 3$, it is unclear how to obtain an analogue of $G_n(\alpx;\la,\ld,\ra,\rd)$ for $k$-ary trees, as there does not appear to be natural definition for ascent-descent statistics in this generalized setting.
 \end{enumerate}

\appendix
\section{The values of \texorpdfstring{$B_n$}{bn} and \texorpdfstring{$G_n$}{gn} for small \texorpdfstring{$n$}{en}}\label{app: expansions}
For the sake of brevity, the expansions below assume that $\alpha=\la\ra$, $\beta=\la+\ra$, $\gamma=\ld\rd$ and $\delta=\ld+\rd$.

\begin{table}[h]\label{tab: B_n}
\begin{tabular}{ l|l }
\hline
\multicolumn{2}{ c }{Table of values for $B_n$} \\
\hline
\\
$n=1$ & $1$ \\[0.2em]
\hline
\\[-0.7em]
$n=2$ & $\delta + \beta$\\[0.2em]
\hline
\\[-0.7em]
$n=3$ & $\delta^{2} + 4\beta\delta+ \beta^2  +3\alpha+3\gamma$\\[0.2em]
\hline
\\[-0.7em]
$n=4$ & $\beta^3 + 11\beta^2\delta + 11\beta\delta^2 + \delta^3 + 14\alpha\beta + 22\beta \gamma + 22\alpha\delta + 14\gamma\delta$\\[0.2em]
\hline
\\[-0.7em]
\multirow{2}{*}{$n=5$} & $\beta^4 + 26\beta^3\delta + 66\beta^2\delta^2 + 26\beta\delta^3 + \delta^4 + 45\alpha\beta^2 + 105\beta^2\gamma +
210\alpha\beta\delta + 210\beta \gamma\delta $\\ &$+ 105\alpha\delta^2 + 45\gamma\delta^2 + 50\alpha^2 + 140\alpha \gamma + 50\gamma^2$ \\[0.2em]
\hline
\end{tabular}
\end{table}

\medskip

\begin{table}[h]\label{tab:G_n}
\begin{tabular}{ l|l }
\hline
\multicolumn{2}{ c }{Expansion for $G_n$ in terms of ribbon Schurs} \\
\hline
\\[-0.7em]
$n=1$ & $r_{1}$ \\[0.2em]
\hline
\\[-0.7em]
$n=2$ & $\delta r_{11} + \beta r_{2}$\\[0.2em]
\hline
\\[-0.7em]
$n=3$ & $(\delta^{2} + \gamma)r_{111} + (\beta \delta + \gamma)r_{12} +
(\beta\delta + \alpha)r_{21} + (\beta^{2} + \alpha)r_{3}$\\[0.2em]
\hline
\\[-0.7em]
\multirow{2}{*}{$n=4$} & $(\delta^{3} + 3 \gamma \delta)r_{1111} + (\beta \delta^{2} + \beta \gamma + 2 \gamma
\delta)r_{112} + (\beta \delta^{2} + \beta \gamma + \alpha \delta + \gamma \delta)r_{1 2
1}$\\ &$ + (\beta^{2} \delta + 2 \beta \gamma + \alpha \delta)r_{1 3}+ (\beta \delta^{2} + \beta \gamma
+ 2 \alpha \delta)r_{2 1 1} + (\beta^{2} \delta + \alpha \beta + \beta \gamma + \alpha
\delta)r_{2 2} $\\ &$+ (\beta^{2} \delta + 2 \alpha \beta + \alpha \delta)r_{3 1} +
(\beta^{3} + 3 \alpha \beta)r_{4}$\\[0.2em]
\hline
\\[-0.7em]
\multirow{7}{*}{$n=5$} & $(\delta^{4} + 6 \gamma \delta^{2} + 2 \gamma^{2})r_{11111} + (\beta
\delta^{3} + 3 \beta \gamma \delta + 3 \gamma \delta^{2} + 2 \gamma^{2})r_{1112}$\\ & $ + (\beta
\delta^{3} + 3 \beta \gamma \delta + \alpha \delta^{2} + 2 \gamma \delta^{2} + \alpha \gamma + \gamma^{2})r_{1121} $\\ & $ + (\beta^{2} \delta^{2} + \beta^{2} \gamma + 4 \beta \gamma \delta + \alpha \delta^{2} + \alpha \gamma +
\gamma^{2})r_{113}$\\ & $ + (\beta \delta^{3} + 3 \beta \gamma \delta + 2 \alpha \delta^{2} + \gamma
\delta^{2} + \alpha \gamma + \gamma^{2})r_{1 2 1 1}$\\ & $ + (\beta^{2} \delta^{2} + \beta^{2} \gamma
+ \alpha \beta \delta + 3 \beta \gamma \delta + \alpha \delta^{2} + \alpha \gamma + \gamma^{2})r_{1 2 2}$\\ &$ +
(\beta^{2} \delta^{2} + \beta^{2} \gamma + 2 \alpha \beta \delta + 2 \beta \gamma \delta + \alpha \delta^{2} + 2 \alpha
\gamma)r_{131} + (\beta^{3} \delta + 3 \beta^{2} \gamma + 3 \alpha \beta \delta + 2 \alpha
\gamma)r_{14}$\\&$ + (\beta \delta^{3} + 3 \beta \gamma \delta + 3 \alpha \delta^{2} + 2 \alpha
\gamma)r_{2111} + (\beta^{2} \delta^{2} + \beta^{2} \gamma + 2 \alpha \beta \delta + 2 \beta \gamma
\delta + \alpha \delta^{2} + 2 \alpha \gamma)r_{212}$ \\ & $+ (\beta^{2} \delta^{2} + \beta^{2} \gamma +
3 \alpha \beta \delta + \beta \gamma \delta + \alpha \delta^{2} + \alpha^{2} + \alpha \gamma)r_{221}$\\ & $ +
(\beta^{3} \delta + \alpha \beta^{2} + 2 \beta^{2} \gamma + 3 \alpha \beta \delta + \alpha^{2} + \alpha \gamma)r_{2
3}$\\&$ + (\beta^{2} \delta^{2} + \beta^{2} \gamma + 4 \alpha \beta \delta + \alpha \delta^{2} + \alpha^{2} + \alpha
\gamma)r_{311} $\\&$+ (\beta^{3} \delta + 2 \alpha \beta^{2} + \beta^{2} \gamma + 3 \alpha \beta \delta +
\alpha^{2} + \alpha \gamma)r_{32} + (\beta^{3} \delta + 3 \alpha \beta^{2} + 3 \alpha \beta \delta + 2
\alpha^{2})r_{41}$\\&$ + (\beta^{4} + 6 \alpha \beta^{2} + 2 \alpha^{2})r_{5}$ \\[0.2em]
\hline
\end{tabular}
\end{table}

\addresseshere

\newpage

\begin{thebibliography}{99}
\bibitem{Athanasiadis-Advances}
{\sc C.~Athanasiadis},
{\em {Characteristic polynomials of subspace arrangements and finite fields}},
Adv. Math. 122 (1996), 193--233.

\bibitem{Athanasiadis-survey}
{\sc C.~Athanasiadis},
{\em {Deformations of Coxeter hyperplane arrangements and their characteristic polynomials}},
Adv. Stud. Pure Math. 27 (2000), 1--26.

\bibitem{Athanasiadis-JACo}
{\sc C.~Athanasiadis},
{\em {Extended Linial hyperplane arrangements for root systems and a conjecture of Postnikov and Stanley}},
J. Algebraic Combin. 10 (1999), 207--225.

\bibitem{Athanasiadis-slides}
{\sc C.~Athanasiadis},
{\em {Gamma-positivity in Combinatorics and Geometry}},
S\'em. Lothar. Combin. 77 ([2016-2018]), Art. B77i, 64 pp.


\bibitem{Athanasiadis-EuJC}
{\sc C.~Athanasiadis},
{\em {On free deformations of the braid arrangement}},
European J. Combin. 19 (1998), 7--18.


\bibitem{Athanasiadis-Linusson}
{\sc C.~Athanasiadis and S.~Linusson},
{\em{A simple bijection for the regions of the Shi arrangement of hyperplanes}},
Discrete Math. 204 (1999), 27--39.

\bibitem{Bernardi}
{\sc O.~Bernardi},
{\em{Deformations of the braid arrangement and trees}},
Adv. Math. 335 (2018), 466--518.

\bibitem{Branden}
{\sc P.~Br\"anden},
{\em{Unimodality, log-concavity, real-rootedness and beyond}},
Handbook of enumerative combinatorics, CRC Press, Boca Raton, FL, 2015, 437--483.

\bibitem{Charney-Davis}
{\sc R.~Charney and M.~Davis},
{\em {The Euler characteristic of a nonpositively curved, piecewise Euclidean manifold}},
 Pacific J. Math. 171 (1995),  117--137.

\bibitem{Corteel-Forge-Ventos}
{\sc S.~Corteel, D.~Forge and V.~Ventos},
{\em{Bijections between affine hyperplane arrangements and valued graphs}},
European J. Combin. 50 (2015), 30--37.

\bibitem{Diaconis-Fulman}
{\sc P.~Diaconis and J.~Fulman},
{\em{Foulkes characters, Eulerian idempotents, and an amazing matrix}},
J. Algebraic Combin. 36 (2012), 425--440.

\bibitem{Donaghey-Shapiro}
{\sc R.~Donaghey and L.~Shapiro},
{\em{Motzkin numbers}},
J. Combin. Theory, Ser. A 23 (1977), 291--301.

\bibitem{Drake}
{\sc B.~Drake},
{\em {An inversion theorem for labeled trees and some limits of areas under lattice paths}},
Phd. Dissertation, Brandeis University (2008).

\bibitem{Edelman}
{\sc P.~Edelman},
{\em{Multichains, non-crossing partitions and trees}},
Discrete Math. 40 (1982), 171--179.

\bibitem{Forge}
{\sc D.~Forge},
{\em {Linial arrangements and local binary search trees}},
arxiv preprint: \url{http://arxiv.org/abs/1411.7834}.

\bibitem{Foulkes}
{\sc H.~Foulkes},
{\em{Eulerian numbers, Newcomb's problem and representations of symmetric groups}},
Discrete Math. 30 (1980), 3--49.

\bibitem{Gal}
{\sc S.~Gal},
{\em {Real root conjecture fails for five- and higher-dimensional spheres}},
 Discrete Comput. Geom. 34 (2005), 269--284.

\bibitem{Garsia-Stanton}
{\sc A.~Garsia and D.~Stanton},
{\em {Group actions of Stanley-Reisner rings and invariants of permutation groups}},
Adv.~Math.~51 (1984), 107--201.

\bibitem{gessel-forests}
  {\sc I.M.~Gessel},
   \emph{Counting forests by descents and leaves},
    Electron. J. Combin. 3 (1996), no. 2, Research Paper 8, 5 pp.

\bibitem{li}
{\sc I.M.~Gessel},
{\em{Lagrange inversion}},
J. Combin. Theory Ser. A 144 (2016), 212--249.

\bibitem{Gessel-Oberwolfach}
{\sc I.~Gessel},
{\em{Oberwolfach Reports (Enumerative Combinatorics)}},
\url{https://www.mfo.de/document/1410/OWR_2014_12.pdf}, 2014, Page 709.

\bibitem{Gessel-unpublished}
{\sc I.~Gessel},
{\em {Unpublished work}}, 1995.




\bibitem{FPSAC-17}
{\sc I.M.~Gessel, S.~Griffin and V.~Tewari},
{\em {Schur positivity and labeled binary trees}},
FPSAC 2017, London, UK, to appear.



\bibitem{Gessel-Reutenauer}
{\sc I.M.~Gessel and C.~Reutenauer},
{\em{Counting permutations with given cycle structure and descent set}},
 J. Comb. Theory Ser. A 64 (1993), 189--215.

\bibitem{Gnedin-Gorin-Kerov}
  {\sc A.~Gnedin, V.~Gorin, S.~Kerov},
  {\em {Block characters of the symmetric groups}},
  J.~Algebraic Combinatorics 38 (2013), 79--101.

 \bibitem{Haiman}
 {\sc M.~Haiman},
{\em{Conjectures on the quotient ring by diagonal invariants}},
J. Algebraic Combin. 3 (1994), 17--76.

\bibitem{Headley}
{\sc P.~Headley},
{\em {On a family of hyperplane arrangements related to affine Weyl groups}},
J. Algebraic Combin. 6 (1997), 331--338.

\bibitem{Holte}
{\sc J.~M.~Holte},
{\em{Carries, combinatorics, and an Amazing Matrix}},
Amer. Math. Monthly 104 (1997), 138--149.

\bibitem{Huang}
{\sc J.~Huang},
{\em {0-Hecke algebra action on the Stanley-Reisner ring of the Boolean algebra}}, Ann. Comb. 19 (2015), 293--323.

\bibitem{Kalikow}
{\sc L.~Kalikow},
{\em {Symmetries in trees and parking functions}},
Discrete Math. 256 (2002), 719--741.

\bibitem{Macdonald}
{\sc I.~G.~Macdonald}
{\em{Symmetric functions and Hall polynomials}},
Second edition. Oxford Mathematical Monographs. Oxford Science Publications. The Clarendon Press, Oxford University Press, New York, 1995.

\bibitem{MacMahon}
{\sc P.A.~MacMahon},
{\em{Combinatory Analysis}}, 2 volumes, Cambridge University Press, London, 1915-1916.

\bibitem{McCammond}
{\sc J.~McCammond},
{\em{Noncrossing partitions in surprising locations}},
Amer. Math. Monthly 113 (2006), 598--610.

\bibitem{Miller-Foulkes}
{\sc A.~Miller},
{\em {Foulkes characters for complex reflection groups}},
Proc. Amer. Math. Soc. 143 (2015), 3281--3293.

\bibitem{Novelli-Thibon-Amazing}
{\sc J.-C.~Novelli and J.-Y.~Thibon},
{\em{Noncommutative symmetric functions and an amazing matrix}},
Adv. in Appl. Math. 48 (2012), 528--534.

\bibitem{OEIS}
{\sc OEIS Foundation Inc. (2019)},
{\em {The On-Line Encyclopedia of Integer Sequences}},
\url{http://oeis.org/}

\bibitem{Pak-slides}
{\sc I.~Pak},
{\em{Tree bijections}}, slides:\url{http://www.math.ucla.edu/~pak/lectures/bij-trees-talk1.pdf}.

\bibitem{PakPostnikov}
{\sc I.~Pak and A.~Postnikov},
{\em{Enumeration of trees and one amazing representation of $S_n$}},
Proc. FPSAC'96 Conf., Minneapolis, MN.

\bibitem{Petersen}
{\sc T.~Kyle Petersen},
{\em {Eulerian numbers}},
Birkha\"user Advanced Texts: Basler Lehrb\"ucher, Birkha\"user/Springer, New York, 2015.

\bibitem{Postnikov-JCTA}
{\sc A.~Postnikov},
{\em{Intransitive trees}},
J. Combin. Theory Ser. A 79 (1997), 360--366.

\bibitem{Postnikov-Stanley}
{\sc A.~Postnikov and R.~Stanley},
{\em {Deformations of Coxeter hyperplane arrangements}},
J. Combin. Theory Ser. A 91 (2000), 544--597.

\bibitem{Postnikov-Reiner-Williams}
{\sc A.~Postnikov, V.~Reiner and L.~Williams},
{\em {Faces of generalized permutohedra}},
 Doc. Math. 13 (2008), 207--273.

\bibitem{Preville-Ratelle-Viennot}
{\sc L.-F.~Pr\'eville-Ratelle and X.~Viennot},
{\em {The enumeration of generalized Tamari intervals}},
 Trans. Amer. Math. Soc. 369 (2017), 5219--5239.

\bibitem{Orlik-Terao}
{\sc P.~Orlik and H.~Terao},
{\em {Arrangements of hyperplanes}},
Grundlehren der Mathematischen Wissenschaften, Springer-Verlag, Berlin, 1992.

\bibitem{Sagan}
{\sc B.~Sagan},
{\em {The Symmetric Group}},
Grad. Texts in Math., Springer-Verlag, New York, 2001.

\bibitem{ShareshianWachs-1}
{\sc J.~Shareshian and M.~Wachs},
{\em{Chromatic quasisymmetric functions}},
Adv. Math. 295 (2016), 497--551.

\bibitem{ShareshianWachs-2}
{\sc J.~Shareshian and M.~Wachs},
{\em{Eulerian quasisymmetric functions}},
 Adv. Math. 225 (2010), 2921--2966.

\bibitem{Shi-Lecture notes}
{\sc J.~Shi},
{\em {The Kazhdan-Lusztig cells in certain affine Weyl groups}},
Lecture Notes in Mathematics 1179, Springer-Verlag, Berlin, 1986.

\bibitem{Shi-JLMS}
{\sc J.~Shi},
{\em {Sign types corresponding to an affine Weyl group}},
J. London Math. Soc. 35 (1987), 56--74.

\bibitem{Simion}
{\sc R.~Simion},
{\em {Noncrossing partitions}},
 Discrete Math. 217 (2000), 367--409.

\bibitem{Stanley}
{\sc R.~P.~Stanley},
{\em{Some combinatorial properties of Jack symmetric functions}},
 Adv. Math. 77 (1989), 76--115.


\bibitem{Stanley-Alternating}
{\sc R.~P.~Stanley},
{\em{Alternating permutations and symmetric functions}},
 J. Comb. Theory Ser. A 114 (2007), 436--460.

 \bibitem{Stanley-Catalan}
 {\sc R.~P.~Stanley},
 {\em{Catalan numbers}},
Cambridge Univ. Press, Cambridge, 2015.



\bibitem{Stanley-EC1}
  {\sc R.~P. Stanley},
  {Enumerative Combinatorics, Vol. 1}, Second Ed.
   Cambridge Studies in Adv. Math., 49, Cambridge Univ. Press, Cambridge, 2012.


\bibitem{Stanley-EC2}
  {\sc R.~P.~Stanley},
  {Enumerative Combinatorics, Vol. 2},
  Cambridge Studies in Adv. Math., 62, Cambridge Univ. Press, Cambridge, 1999.

\bibitem{Stanley-notes}
{\sc R.P.~Stanley},
{\em {Hyperplane arrangements}},
\url{http://www-math.mit.edu/~rstan/arrangements/arr.html}.

\bibitem{Stanley-Pnas}
{\sc R.~P.~Stanley},
{\em{Hyperplane arrangements, interval orders, and trees}},
Proc. Nat. Acad. Sci. U.S.A. 93 (1996), 2620--2625.

\bibitem{Stanley-Park}
{\sc R.~P.~Stanley},
{\em{Hyperplane arrangements, parking functions and tree inversions}},
Prog. Math., 161, Birkh\"auser Boston, MA (1998), 359--375.

\bibitem{Stanley- Noncrossing}
{\sc R.~P.~Stanley},
{\em{Parking functions and noncrossing partitions}},
Electron. J. Combin. 4 (1997), 14 pp.

\bibitem{Tewari}
{\sc V.~Tewari},
{\em {Gessel polynomials, rooks, and extended Linial arrangements}},
J. Combin. Theory Ser. A 163 (2019), 98--117.

\bibitem{WachsPosetTopology}
{\sc M.~L.~Wachs},
{\em{Poset topology: tools and applications}},
Geometric Combinatorics,  IAS/Park City Math. Ser., 13, Amer. Math. Soc., Providence, RI, 2007, 497--615.

\end{thebibliography}
\end{document}